\documentclass[11pt,reqno]{amsart}
\usepackage[a4paper,left=22mm,right=22mm,top=35mm,bottom=37mm,marginpar=20mm]{geometry} 
\usepackage{libertine}
\usepackage{libertinust1math}

\usepackage[backend=biber, bibstyle=ieee, citestyle=numeric, sorting=nyt,url=false]{biblatex}

\AtBeginBibliography{\small}

\usepackage[T1]{fontenc}
\usepackage[utf8]{inputenc}
\usepackage{enumerate}
\usepackage{mathtools}
\usepackage{enumitem}
\usepackage{stmaryrd}
\usepackage{graphicx}
\usepackage{wrapfig}

\usepackage[acronym]{glossaries}
\usepackage{longtable}
\usepackage{amsmath}
\usepackage{amsthm}
\usepackage{amssymb}
\usepackage{esint} 
\usepackage{mathrsfs} 
\usepackage{appendix} 
\usepackage{tikz}
\usepackage{tikz-cd}
\usetikzlibrary{calc}
\usetikzlibrary{arrows.meta, positioning,shadows,decorations.pathreplacing,shadows.blur}

\usepackage[hidelinks]{hyperref}

\newcommand{\mres}{%
  \,\raisebox{-.127ex}{\reflectbox{\rotatebox[origin=br]{-90}{$\lnot$}}}%
}

\AtBeginBibliography{%
  \DeclareFieldFormat{prefixnumber}{#1}%
  \DeclareFieldFormat{labelnumber}{#1}}

\definecolor{url}{HTML}{0000FF} 
\definecolor{bluesy}{HTML}{0057e7}
\definecolor{redsy}{HTML}{d62d20}
\definecolor{greensy}{HTML}{008744}

\hypersetup{
  linkcolor  = black,
  citecolor  = black,
  urlcolor   = black,
  colorlinks = true,
}

\newtheorem{theorem}{Theorem}[section]
\newtheorem{lemma}[theorem]{Lemma}
\newtheorem{proposition}[theorem]{Proposition}
\newtheorem{corollary}[theorem]{Corollary}

\newtheorem{thmx}{Theorem}

\newcommand{\toweakstar}{\overset{*}\rightharpoonup}

\theoremstyle{definition}
\newtheorem{definition}[theorem]{Definition}

\newtheorem{notation}[theorem]{Notation}
\newtheorem{example}[theorem]{Example}

\theoremstyle{remark}
\newtheorem{remark}[theorem]{Remark}

\newcommand{\set}[2]{\left\{\, #1 \  \textup{\textbf{:}}\  #2 \,\right\}}

\newcommand{\R}{\mathbb R}

\newcommand{\spn}{\mathrm{span}}

\newcommand{\Kcal}{\mathcal K}

\newcommand{\Hcal}{\mathcal H}

\newcommand{\Fbf}{\mathbf F}

\newcommand{\Mbf}{\mathbf M}
\newcommand{\Nbf}{\mathbf N}
\newcommand{\Nbb}{\mathbb N}

\newcommand{\eps}{\varepsilon}

\newcommand{\dpr}[1]{\langle #1 \rangle}
\newcommand{\lpr}[1]{\llbracket #1 \rrbracket}

\newcommand{\embed}{\hookrightarrow}

\DeclareMathOperator{\im}{im}

\DeclareMathOperator{\Lip}{Lip}
\DeclareMathOperator{\spt}{spt}
\DeclareMathOperator{\KR}{KR}
\DeclareMathOperator{\BV}{\mathbb{BV}}
\DeclareMathOperator{\SBV}{\mathbb{SBV}}

\DeclareMathOperator{\dist}{dist}

\DeclareMathOperator{\Frag}{Frag}

\newglossary[slg]{symbol}{sym}{sbl}{List of Symbols}

\usepackage{hhline}
\newglossarystyle{symbolsstyle}{
    {\begin{longtable}{lllllr}}{\end{longtable}}%
  \renewcommand*{\glsgroupheading}[1]{}%
}

\tikzset{
  if/.code n args=3{\pgfmathparse{#1}\ifnum\pgfmathresult=0
    \pgfkeysalso{#3}\else\pgfkeysalso{#2}\fi},
  lower cantor/.initial=.3333, upper cantor/.initial=.6667, y cantor/.initial=.5,
  declare function={
    cantor_l(\lowerBound,\upperBound)=
      (\pgfkeysvalueof{/tikz/lower\space cantor})*(\upperBound-\lowerBound)+\lowerBound;
    cantor_u(\lowerBound,\upperBound)=
      (\pgfkeysvalueof{/tikz/upper\space cantor})*(\upperBound-\lowerBound)+\lowerBound;
    cantor(\lowerBound,\upperBound)=
      (\pgfkeysvalueof{/tikz/y\space cantor})*(\upperBound-\lowerBound)+\lowerBound;},
  cantor start/.style n args=5{%
    insert path={(#1,#3)},
    cantor={#1}{#2}{#3}{#4}{#5}{0},
    insert path={to[every cantor edge/.try, cantor 1 edge/.try] (#2,#4)}},
  cantor/.style n args=6{%
    /utils/exec=%
      \pgfmathsetmacro\lBx{cantor_l(#1,#2)}%
      \pgfmathsetmacro\uBx{cantor_u(#1,#2)}%
      \pgfmathsetmacro\y{cantor(#3,#4)},
    style/.expanded={
      if={#6<#5}{cantor={#1}{\lBx}{#3}{\y}{#5}{#6+1}}{},
      insert path={
        to[every cantor edge/.try, cantor 1 edge/.try] (\lBx,\y)
        to[every cantor edge/.try, cantor 2 edge/.try] (\uBx,\y)},
      if={#6<#5}{cantor={\uBx}{#2}{\y}{#4}{#5}{#6+1}}{}}}}

\NewDocumentCommand{\newsy}{O{#2}mmmmm}{%
  \newglossaryentry{#1}{
    type=symbol,
    name={#2},
    symbol={#2},
    description={#3},
   user4={},    
    user5={\thepage} 
  }%
  \label{symbol:#1}
}

\numberwithin{equation}{section} 

\newsy[Mass1]{$\mathbf{M}(T)$}{mass of a Euclidean current $T$}{}{}{}{}
\newsy[FFk]{$\mathscr M_k(\R^d)$}{$k$-dimensional Euclidean currents in $\R^d$}{}{}{}{}
\newsy[FFNk]{$\mathscr N_k(\R^d)$}{$k$-dimensional Euclidean normal currents in $\R^d$}{}{}{}{}
\newsy[Flatk]{$\mathscr F_k(\R^d)$}{$k$-dimensional flat chains in $\R^d$}{}{}{}{}
\newsy[Flat_norm]{$\mathbf{F}(T)$}{flat norm of a Euclidean or metric current $T$}{}{}{}{}
\newsy[F_0]{$\mathbf{F}_0(\partial T)$}{homogeneous flat norm of a Euclidean or metric current $T$}{}{}{}{}
\newsy[Lip]{$\Lip(f)$}{Lipschitz seminorm of a function $f: X \to \R$}{}{}{}{}
\newsy[Lip_b]{$\Lip_b(X)$}{space of bounded real-valued Lipschitz functions on $X$}{}{}{}{}
\newsy[Metric_k]{$\mathscr D^k(X)$}{metric $k$-form on $X$}{}{}{}{}
\newsy[mass_measure_1]{$\| T\|$}{mass measure of an FF current $T$}{}{}{}{}
\newsy[mass_measure]{$\| T\|$}{mass measure of a metric current $T$}{}{}{}{}
\newsy[mass2]{$\mathbf M(T)$}{mass of a metric current $T$}{}{}{}{}
\newsy[metric_k_currents]{$\mathbf M_k(X)$}{$k$-dimensional metric currents on $X$}{}{}{}{}
\newsy[metric_k_normal]{$\mathbf N_k(X)$}{$k$-dimensional metric normal currents on $X$}{}{}{}{}
\newsy[metric_T]{$T$}{metric current}{}{}{}{}
\newsy[boundary]{$\partial T$}{boundary of a current $T$}{}{}{}{}
\newsy[KR_boundary]{$\|\partial T \|_{\text{KR}}$}{Kantorovich--Rubinstein norm of $\partial T$}{}{}{}{}
\newsy[integral_curve]{$\lpr{\theta}$}{integral curve current associated with a Lipschitz curve $\theta$}{}{}{}{}
\newsy[length]{$\ell(\theta)$}{parametric length of a Lipschitz curve $\theta$}{}{}{}{}
\newsy[poly_k]{$\mathscr P_1(X)$}{$1$-dimensional polyhedral currents on $X$}{}{}{}{}
\newsy[rect_1]{$\mathcal R_1(X)$}{$1$-dimensional rectifiable currents on $X$}{}{}{}{}
\newsy[BV_curves]{$\mathbb{BV}(I,X)$}{space of compact $BV$ curves with values on $X$}{}{}{}{}
\newsy[SBV_curves]{$\mathbb{SBV}(I,X)$}{space of compact $SBV$ curves with values on $X$}{}{}{}{}
\newsy[approx_metric]{$|\dot u|$}{approximate metric (upper) differential}{}{}{}{}
\newsy[abs_current]{$\lpr{u}^a$}{absolutely continuous current associated with $u \in \mathbb BV(I,X)$}{}{}{}{}
\newsy[cantor_current]{$\lpr{u}^c$}{cantorian current associated with $u \in \mathbb BV(I,X)$}{}{}{}{}
\newsy[jump_current]{$\lpr{u}^j$}{jump current associated with $u \in \mathbb BV(I,X)$}{}{}{}{}
\newsy[cadlag]{$\mathbb{D}([0,1],X)$}{space of Càdlàg curves with values on $X$}{}{}{}{}
\newsy[K_parameter]{$u_K$}{optimal transport map from $\mathcal L^1 \mres [0,1]$ to $|K|^{-1}\mathcal L \mres K$}{}{}{}{}
\newsy[u-]{$u^-$}{left-continuous representative of $u$}{}{}{}{}
\newsy[u+]{$u^+$}{right-continuous representative of $u$}{}{}{}{}
 \newsy[upper_gradient]{$|\dot{\theta}|$}{upper metric gradient of a Lipschitz curve $\theta$}{}{}{}{}
\newsy[AE]{$\text{\AE}$(X)}{Arens--Eells space associated to a (pointed) metric space $X$}{}{}{}{}
\newsy[Lip(X)]{$\Lip(X)$}{space of real-valued Lipschitz functions on $X$}{}{}{}{}
\newsy[Lip_0]{$\Lip_0(X)$}{space of  real-valued Lipschitz functions on $X$ vanishing at a point $x_o$}{}{}{}{}
\newsy[push]{$\Phi_\#(T)$}{push-forward of a current $T$ under a map $\Phi:X \to X'$}{}{}{}{}
\newsy[mres]{$T \mres E$}{Restriction of a current $T$ to a Borel set $E$}{}{}{}{}
\newsy[set]{$\lpr{A}$}{current associated with a Borel set $A \subset \R$}{}{}{}{}
\newsy[curve rectifiable]{$X$}{curve rectifiable space}{}{}{}{}
\newsy[rectifiable set]{$E$}{$1$-rectifiable set in a metric space}{}{}{}{}
\newsy[fragment]{$\gamma:B \to X$}{Lipschitz fragment}{}{}{}{}
\newsy[metric_space]{$(X,d)$}{metric space endowed with a metric $d$}{}{}{}{}
\newsy[piece_qc]{}{piecewise quasiconvex space}{}{}{}{}
\newsy[qc]{}{quasiconvex space}{}{}{}{}
\newsy[KR]{$\|\lambda\|_{\KR}$}{Kantorovich Rubinstein seminorm of $\lambda \in \Lip(X)^*$}{}{}{}
\newsy[I]{$I$}{the open interval $(0,1)$ in $\R$}{}{}{}{}
\newsy[oriented_interval]{$\lpr{x,y}$}{$1$-current associated with the oriented interval $[x,y]$ in a metric vector space}{}{}{}{}
\newsy[Haus_k]{$\mathcal H^k$}{$k$-dimensional Hausdorff measure}{}{}{}{}
\newsy[kFCC]{$k$-FCC}{$k$-dimensional flat chain conjecture on a metric space}{}{}{}{}
\newsy[|K|]{$|K|$}{one-dimensional Lebesgue measure of a Borel set $K \subset \R$; also denoted $\mathscr L^1(K)$}{}{}{}{}
\newsy[K(Y)]{$\mathcal K(Y)$}{compact subsets of a metric space $Y$}{}{}{}{}
\newsy[frag]{$\Frag(X)$}{fragments $f: [0,1] \to X$, topologized as a subspace of $\mathcal K([0,1]\times X)$}{}{}{}{}
\newsy[Theta_SBV]{$\Theta_{\SBV}(X)$}{injective curves $u \in \SBV(I,X) \cap \mathbb D([0,1],X)$ with constant metric speed}{}{}{}{}
\newsy[Theta_X]{$\Theta(X)$}{injective curves $u \in \Lip([0,1],X)$ with constant metric speed}{}{}{}{}
\newsy[total_variation]{$|Du|$}{total variation of a curve $u : (0,1) \to X$}{}{}{}{}
\newsy[total_ac]{$|D^a u|$}{absolutely continuous part of the total variation of a curve $u : (0,1) \to X$}{}{}{}{}
\newsy[total_j]{$|D^j u|$}{jump part of the total variation of a curve $u : (0,1) \to X$}{}{}{}{}
\newsy[total_c]{$|D^c u|$}{Cantor part of the total variation of a curve $u : (0,1) \to X$}{}{}{}{}

\title[One dimensional metric currents]{Structural properties of one-dimensional metric currents: \\ SBV-representations, connectedness and the flat chain conjecture}
\author{Adolfo Arroyo-Rabasa}
\address[A.\ Arroyo-Rabasa]{Dipartimento di Matematica,
Largo Bruno Pontecorvo,
5 – 56127 Pisa}
\email{adolfo.rabasa@unipi.it}

\author{Guy Bouchitt\'e}
\address[G.\ Bouchitt\'e]{
	Laboratoire IMATH, Universit\'e de Toulon\\
	BP 20132,  83957 La Garde (France)}
	\email{bouchitte@univ-tln.fr}

\date{}

\makenoidxglossaries

\begin{document}

\begin{abstract}A comprehensive study of one-dimensional metric currents and their relationship to the geometry of metric spaces is presented. We resolve the one-dimensional flat chain conjecture in this general setting, by proving that its validity is equivalent to a simple geometric connectedness property. More precisely, we prove that metric currents can be approximated in the mass norm by normal currents if and only if every $1$-rectifiable set can be covered by countably many Lipschitz curves up to an $\mathcal H^1$-negligible set. 
Building on this, we demonstrate that any $1$-current in a Banach space can be completed into a cycle by a rectifiable current, with the added mass controlled by the Kantorovich--Rubinstein norm of its boundary. We further refine our approximation result 
by showing that these currents can be approximated by polyhedral currents modulo a cycle. Finally, in arbitrary complete metric spaces, we establish a Smirnov-type decomposition for one-dimensional currents. This decomposition expresses such currents as a superposition, without mass cancellation, of currents associated with curves of bounded variation that have a vanishing Cantor part.

\medskip
\noindent \textsc{Keywords:} Arens--Eells space, bounded variation, Càdlàg function, current, cycle, flat chain conjecture, Lipschitz function, metric space, quasiconvex space, Smirnov's theorem, Kantorovich norm.

\medskip
\noindent \textsc{MSC (2020):} 32C30, 49Q15, 51F30 (primary), 28A75 (secondary).
\end{abstract}

\maketitle

\setcounter{tocdepth}{1}

\tableofcontents

\section{Introduction}

The theory of normal and integral currents, pioneered by Federer and Fleming in \cite{FedererFleming1960,Federer1969}, has been instrumental in the mathematical study of geometric objects like curves, surfaces, and volumes in $\R^d$. Formally, a $k$-current $T$ is a continuous linear functional (distribution) on the space of smooth, compactly supported $k$-forms:
\begin{equation}\label{eq:form}
\omega = f \, d\pi_1\wedge \cdots \wedge d\pi_k.
\end{equation}Those $k$-currents that extend continuously to all compactly supported continuous $k$-forms are called Federer–Fleming (FF) currents. By the Riesz representation theorem, such currents correspond to $k$-vector-valued Radon measures.
In this way, an FF $k$-current generalizes the notion of a $k$-dimensional oriented surface, capturing concepts such as size (mass) \gls{Mass1} $= \text{\gls{mass_measure_1}}(\R^d)$, orientation and boundary. We write \gls{FFk} to denote the space of FF $k$-currents in $\R^d$.

A central operation in this theory is the boundary operator $\partial$, defined via duality in analogy with Stokes’ theorem: if $\omega = f \, d\pi_1 \wedge \cdots \wedge d\pi_{k-1}$, then
\begin{equation}\label{eq:diff_intro}
\partial T(\omega) = T(d\omega) = T(df \wedge d\pi_1 \wedge \cdots \wedge d\pi_{k-1}),
\end{equation}
where $d\omega$ denotes the exterior derivative of $\omega$. The subclass of \emph{normal} $k$-currents \gls{FFNk} consists of FF currents whose boundaries also have finite mass. Altogether, these notions proved pivotal in solving classical problems, such as Plateau's problem \cite{Plateau1873a,Plateau1873b}, and laid the foundation for the development of modern geometric measure theory (see, e.g., \cite{Federer1969,Morgan2000GMT}).

Yet, this distributional framework has its limitations: many non Euclidean spaces of interest, such as fractal sets, path spaces in optimal transport, or general metric spaces, lack a smooth differentiable structure. In such settings, the classical definitions of differential forms and their derivatives, as in \eqref{eq:form}–\eqref{eq:diff_intro}, no longer apply. To address this challenge, Ambrosio and Kirchheim developed in \cite{AK_Acta} a theory of \emph{metric currents} on complete metric spaces $(X,d)$. Their approach, which draws inspiration from De Giorgi's ideas, elegantly bypasses the need of a differentiability structure by defining metric currents as functionals on $(k+1)$-tuples (metric $k$-forms):
\begin{equation}\label{eq:identification_intro}
\text{«} \, f \, d\pi_1 \wedge \cdots \wedge d\pi_k = (f, \pi_1, \dots, \pi_k) \, \text{»}
\end{equation}
where $f : X \to \R$ is a bounded Lipschitz function and the $\pi_i : X \to \R$ are Lipschitz functions acting as coordinate functions of $X$. A metric $k$-current is then a multilinear functional $T : \Lip_b(X) \times \Lip(X)^k \to \R$, satisfying axioms of continuity, locality, and finite mass (see Definition~\ref{def:MC}), which ensure that $(f,\pi_1,\dots,\pi_k)$ behaves formally like a differential form. The boundary operator is also defined via duality, using an exterior differential on metric forms: $d(f, \pi_1, \dots, \pi_k) = (1, f, \pi_1, \dots, \pi_k)$. As in the classical setting, a metric current is called \emph{normal} if both the current and its boundary have finite mass. We denote the spaces of metric $k$-currents and normal metric $k$-currents by \gls{metric_k_currents} and \gls{metric_k_normal}, respectively.

\subsection{The flat chain conjecture: a source of motivation}

A central question in this theory is how metric currents relate to FF currents in Euclidean spaces. In their foundational work, Ambrosio and Kirchheim observed that each metric current $T \in \Mbf_k(\R^d)$ can be associated with an FF current $\tilde T \in \mathscr M_k(\R^d)$, and this correspondence restricts to a linear isomorphism $\Nbf_k(\R^d) \cong \mathscr N_k(\R^d)$ between normal metric and normal FF currents.

Despite this important link, the full scope of this correspondence is not fully understood. In particular, it is unclear which FF currents can be realized as metric currents. A key intermediate notion in this context is that of \emph{flat chains}, introduced by Federer and Fleming in \cite{FedererFleming1960} (see also \cite{Federer1975flat}). The space of flat $k$-chains with finite mass, denoted \gls{Flatk}, is defined as the closure of $\mathscr N_k(\R^d)$ with respect to the \emph{flat norm} (defined equally for FF and metric currents):
\[
	\text{\gls{Flat_norm}} := \inf \left\{ \Mbf(R) + \Mbf(S) \,\middle|\, R + \partial S = T \right\}.
\]
Ambrosio and Kirchheim observed that the correspondence between normal metric and FF currents extends to flat chains: every flat chain $\tilde T \in \mathscr F_k(\R^d)$ induces a metric current $T \in \Mbf_k(\R^d)$. This led them to conjecture (see \cite[p.68]{AK_Acta}) the identification $T \mapsto \tilde T$ extends to an isomorphism between $\mathscr F_k(\R^d)$ and $\Mbf_k(\R^d)$ —a statement now known as the \emph{flat chain conjecture} (FCC).

The FCC has been resolved positively for $k=1$ and $k=d$. Schioppa~\cite{Schioppa_2016} provided the proof of the one-dimensional case through deep algebraic and geometric methods, with alternative and simpler proofs appearing later \cite{ABM,MM2024,demasi2024}. For the top-dimensional case $k=d$, the only known proof relies on Schioppa’s framework and a groundbreaking rigidity result for PDE-constrained measures by De Philippis and Rindler \cite{De_Philippis_2016}. For the intermediate dimensions $1 < k < d$, results (derived from Schioppa's theory and work by Arroyo-Rabasa et al.~\cite{APHR2020, arroyo-rabasa_elementary_2020}) are currently limited to rectifiability and dimensional estimates of the form $\|T\| \ll \mathcal I^k \ll \text{\gls{Haus_k}}$, where $\mathcal I^k$ and $\mathcal H^k$ denote the integral-geometric and Hausdorff $k$-dimensional measures, respectively.

Recent work by Alberti and Marchese \cite{AM_flat} (see also~\cite{DP}) gives additional structural insight, showing that flat chains can be thought of as normal currents “punctured” by Borel holes. Thus, resolving the FCC would not only bridge a theoretical gap but also yield a characterization of metric currents in Euclidean space.

From a practical perspective, the FCC can be viewed as a \emph{density} statement for metric currents: for any $T \in \Mbf_k(\R^d)$, there exists a sequence $(T_j) \subset \Nbf_k(\R^d)$ such that $\mathbf F(T - T_j) \to 0$. When the limit $T$ has finite mass, it is well-known (see \cite[\S4.1.17, p. 374]{Federer1969}or \cite[{Lem.~A.7}]{brasco2014continuous}) that this convergence is equivalent to convergence in mass $\Mbf(T - N_j) \to 0$ for some (possibly different) sequence $(N_j) \subset \Nbf_k(\R^d)$. Figure~\ref{fig:flat_normal} illustrates this idea:

\begin{figure}[h]
\centering
\fbox{\includegraphics[width=0.7\textwidth]{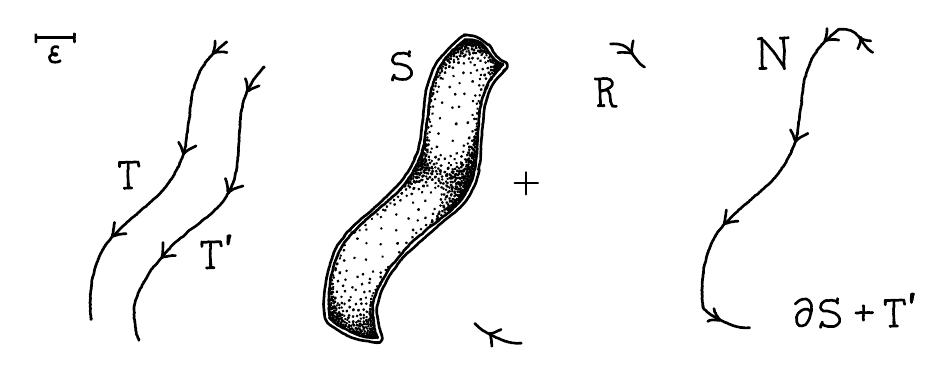}}
\caption{$T$ and $T'$ (left) are $\eps$-close in the flat norm: their difference $T - T'$ can be expressed as $\partial S + R$, where $S$ is a current with area of order $\eps$ and $R$ is a $1$-current with length of order $\eps$ (middle). On the right: $T$ is $\eps$-close in mass to the normal current $N = \partial S + T'$.}
\label{fig:flat_normal}
\end{figure}
It is easy to see that the same reasoning carries over to metric spaces. We will therefore adopt the following formulation of the FCC in a general metric space:

\begin{definition}[Flat chain conjecture]
Let $k \ge 0$ and let $X$ be a metric space. We say that $X$ satisfies the \gls{kFCC} if $\Nbf_k(X)$ is dense in $\Mbf_k(X)$ with respect to the mass norm. That is, for every $T \in \Mbf_k(X)$, there exists a sequence $N_j \in \Nbf_k(X)$ such that
\[
	\lim_{j \to \infty} \Mbf(T - N_j) = 0.
\]
\end{definition}

\subsection*{Main contributions}

The goal of this paper is to develop a comprehensive study of one-dimensional currents in metric spaces, primarily employing functional and optimal transport ideas. Our main results address both their structure and the validity of the {$1$-FCC}. What follows is a brief overview of the principal contributions; we refer the reader to the next section for precise definitions and terminology.

\subsection*{a) Characterizing the $1$-FCC via curve rectifiability}

Our first main result (contained in Theorem~\ref{thm:main1}) provides a sufficient and necessary condition for the validity of the $1$-FCC in a general separable and complete metric space $X$. This condition, which we term as \textbf{curve rectifiability}, requires that every $1$-rectifiable subset of $X$ can be covered, up to a set of $\mathcal H^1$-measure zero, by a countable union of Lipschitz curves. Therefore, we identify the minimal "connectedness" condition required for the $1$-FCC to hold.

\subsection*{b) Approximations by rectifiable and polyhedral currents} We strengthen this approximation property under a more quantitative assumption of connectedness, which we formalize through the notion of \textbf{piecewise quasiconvex} metric spaces (see Definition~\ref{ass:con}). This class includes many familiar settings, such as quasiconvex spaces, geodesic spaces, and vector spaces. In such spaces, we demonstrate (see Theorem~\ref{thm:approx_qc}) that every metric current can be approximated (modulo a cycle) by a sum of integral curve currents, i.e., currents associated with Lipschitz curves. Moreover, in Banach spaces, we further refine the approximation using polyhedral currents —finite linear combinations of oriented segments.

A key aspect of our methods in piecewise quasiconvex spaces is a \emph{novel} construction of primitives associated with $0$-currents (see Lemma~\ref{lem:OT}). This construction is based on an understanding of the Arens–Eells space $\text{\AE}(X)$ of $X$ as a pre-dual of the quotient space $\Lip(X)/\mathbb{R}$, displaying interesting connections to optimal transport theory. 

\subsection*{c) Covering metric currents by cycles}

Crucially, establishing the $1$-FCC in Banach (and piecewise quasiconvex) spaces opens the door to a deeper structural analysis. In Theorem~\ref{cor:cycle_cover}, we show that every one-dimensional metric current $T \in \Mbf_1(E)$ in a Banach space (of dimension at least two) $E$ can be \emph{covered} by a cycle $C \in \Nbf_1(E)$ in an almost optimal way, with respect to a suitable generalization of the Monge transportation problem associated with $\partial T$. Specifically, we establish that $T$ can be written as $T = C \mres B$ for some Borel set $B \subset E$, with the mass of the "remaining" cycle $C \mres B^c$ controlled by the Kantorovich norm of $\partial T$. This construction extends very recent ones by De Pauw~\cite[Thm 2.12]{DP} and refines prior ones —such as those in \cite{AM_flat}— by providing an efficient method for ``filling the holes'' of a current with a cycle of almost \emph{minimal mass}. 

\subsection*{d) A Smirnov-type decomposition for metric currents}

Drawing inspiration from Smirnov's decomposition of solenoidal measures~\cite{smirnov1993decomposition,rodriguez2024smirnov} and its extension to normal $1$-currents in metric spaces by Paolini and Stepanov~\cite{paolini2012decomposition,paolini2012decomposition2}, we conclude our main results by establishing a complete structural description of one-dimensional metric currents. Our Smirnov-type decomposition applies to any complete metric space and notably requires \textbf{no assumptions on connectedness}. 

Theorem~\ref{thm:rep} shows that any $T \in \Mbf_1(X)$ admits a decomposition into an integral over currents associated with a class of \emph{$SBV$-curves}. These are curves $u : [0,1] \to X$ of bounded variation, possessing well-defined left and right limits at every point, whose "fractal-like" part (the Cantorian part) vanishes. This means they only exhibit absolutely continuous behavior or jump-type discontinuities. A version of this result (Corollary~\ref{cor:fragment}) can also be phrased in terms of \emph{path fragments}—a more familiar concept in the metric geometry community. However, the SBV framework offers additional advantages: it distinguishes between the jump and continuous parts of a curve and aligns naturally with tools from analysis and stochastic analysis (e.g., the study of flows, evolutions, or transport along discontinuous trajectories).

More precisely, we demonstrate that there exists a finite Borel measure $\eta$ on the space of SBV-curves, equipped with the topology of Càdlàg functions, such that the current and its mass can be represented as:
\begin{align*}
	T &= \int_{SBV([0,1],X)} \llbracket u \rrbracket^a \, d\eta(u), \\
	\Mbf(T) &= \int_{SBV([0,1],X)} \ell(u) \, d\eta(u),
\end{align*}
where $\llbracket u \rrbracket^a$ denotes the current associated with the absolutely continuous part of $u$ and $\ell(u)$ is its  length.

\subsection*{A Note on Related Work}

While writing this paper, we became aware that Bate et al.~\cite{other_group} have simultaneously and independently obtained results related to those presented here. To the best of our knowledge, their techniques draw on metric geometry and analysis on metric spaces, whereas ours are typically linked to optimal transport and functions of bounded variation. A notable point of convergence is the independent discovery of the utility of the Arens-Eells framework for constructing approximations of metric 1-currents, which was also found by De Pauw very recently \cite{DP}.


\section{Definitions and description of the main results}

Throughout this paper, we assume that $X$ is a \textbf{separable metric space}. This assumption is standard and simplifies the presentation, though it can be lifted under a specific set-theoretic axiom. The assumption of separability can be replaced by the condition that the cardinality of any set is an \textbf{Ulam number} (see \cite[p.~58]{Federer1969}). Under this axiom, the mass measure of any metric current $T \in \Mbf_k(X)$ is supported on a $\sigma$-compact set (see \cite[Lem.~2.9]{AK_Acta}).

\subsection*{Fundamental definitions}

Metric currents are fundamentally defined as linear functionals on \emph{metric forms}, which are constructed from Lipschitz functions on $X$. A function $f: X_1 \to X_2$ between two metric spaces $(X_1, d_1)$ and $(X_2, d_2)$ is \emph{Lipschitz} if its Lipschitz constant is finite:
\[
	\text{\gls{Lip}} \coloneqq \sup_{\substack{x,y \in X\\ x \neq y}} \frac{d_2(f(x),f(y))}{d_1(x,y)} < \infty.
\]
We denote the space of all such Lipschitz maps by $\Lip(X_1,X_2)$. When the target space is the real line, $X_2 = \mathbb{R}$, we simply write $\Lip(X)$, and \gls{Lip_b} to denote its subspace of bounded Lipschitz functions.

A \emph{metric $k$-form} on $X$ is a tuple $(f, \pi_1, \dots, \pi_k)$, where $f \in \Lip_b(X)$ and $\pi_1, \dots, \pi_k \in \Lip(X)$. The space of such forms is denoted by $\mathscr{D}^k(X) \coloneqq \Lip_b(X) \times \Lip(X)^k$, with $\mathscr{D}^0(X) := \Lip_b(X)$.

With this notion in hand, we can define metric currents.
\begin{definition}[Metric $k$-currents]\label{def:MC} Let $k \ge 0$.
A \emph{metric $k$-current} on $X$ is a multilinear functional $\text{\gls{metric_T}} : \mathscr{D}^k(X) \to \mathbb{R}$ that satisfies three key properties:
\begin{enumerate}[label=\roman*), leftmargin=1.8em, itemsep=0.7em]
    \item \textbf{Finite Mass:} There exists a finite Borel measure $\mu_T$ on $X$ such that
    \[
    |T(f, \pi_1, \dots, \pi_k)| \le \left( \prod_{i=1}^k \text{Lip}(\pi_i) \right) \int_X |f| \, d\mu_T.
    \]
    The minimal measure $\mu_T$ satisfying this inequality is called the \emph{mass measure}, and the \emph{mass} of $T$ is defined as $\text{\gls{mass2}} := \text{\gls{mass_measure}}(X)$.
    \item \textbf{Continuity:} If  $(f, \pi_1^j, \dots, \pi_k^j)$ converges pointwise to  $(f,\pi_1,\dots,\pi_k)$ in $X$ and $\Lip(\pi_i^j) \le C < \infty$ for all $j$ and all $i \in \{1,\dots,k\}$, then 
    \[
    T(f, \pi_1^j, \dots, \pi_k^j)  \to T(f, \pi_1, \dots, \pi_k).
    \]
    \item \textbf{Locality:} If for a tuple $(f,\pi_1,\dots,\pi_k)$, there exists $i \in \{1,\dots,k\}$ such that coordinate function $\pi_i$ is constant on a neighborhood of the support of $f$, then 
    \[
    T(f, \pi_1, \dots, \pi_k) = 0.
    \]
\end{enumerate}
\end{definition}

The space of metric $k$-currents, denoted \gls{metric_k_currents}, is a Banach space when endowed with the mass norm. Any current $T \in \Mbf_k(X)$ can be uniquely and continuously extended to act on $L^1(\|T\|) \times \text{Lip}(X)^k$. Notice that $\Mbf_0(X)$ coincides with the space of finite signed Borel measures. 

The standard operations of restriction and push-forward are defined in the natural way. For any Borel set $E \subset X$, the \emph{restriction} of a current $T$ on $E$ is given by 
\[
\text{\gls{mres}}(f,\pi_1,\dots,\pi_k) \coloneqq T(\mathbf{1}_E f,\pi_1,\dots,\pi_k).
\]
Similarly, if $\Phi: X \to X'$ is a Lipschitz map, the \emph{push-forward} of $T$ under $\Phi$ is defined as 
\begin{equation}\label{eq:push_0}
    \text{\gls{push}}(f,\pi_1,\dots,\pi_k) \coloneqq T(\Phi \circ f, \Phi \circ \pi_1, \dots, \Phi \circ \pi_k).
\end{equation}

The {boundary operator} is defined in the following way:

\begin{definition}[Boundary operator]\label{def:boundary_normal}
The \emph{boundary} of a multilinear functional $T : \mathscr D^k(X) \to \R$ is the multilinear functional $\partial T : \mathscr D^{k-1}(X) \to \R$ defined by
\[
\text{\gls{boundary}}(f, \pi_1, \dots, \pi_{k-1}) \coloneqq T(1, f, \pi_1, \dots, \pi_{k-1}).
\]
\end{definition}
Normal metric currents are then defined analogously to classical currents:
\begin{definition}[Normal currents]
    A current $T$ is called \emph{normal} if both $T$ and its boundary $\partial T$ have finite mass. The space of normal $k$-currents is denoted by \gls{metric_k_normal}. A current $T$ is a \emph{cycle} if its boundary is zero, that is, $\partial T = 0$.
\end{definition}

\subsection*{The Kantorovich–Rubinstein norm}

A key quantity associated with the boundary of a $1$-current is the \emph{Kantorovich–Rubinstein (KR) norm}, a foundational concept in optimal transport theory. For a boundary $\partial T$, its KR norm is defined as:
\[
	\text{\gls{KR_boundary}} := \sup\{ \partial T(f) : f \in \Lip_1(X) \},
\]
where $\Lip_1(X)$ is the space of $1$-Lipschitz functions on $X$. By the definition of a metric current, a crucial inequality holds:
\[
	\|\partial T\|_{\text{KR}} \le \Mbf(T), \qquad \text{for all } T \in \Mbf_1(X).
\]
When $X$ is a geodesic space, then equality is attained attained if $T$ is a transport minimizer for $\partial T$ in a suitable Monge--Kantorovich problem framework (see~\cite{BB-JEMS2001, BBD}). In this case, the KR norm acts as a "homogeneous flat norm" for boundaries of $1$-currents, since
\[
    \text{\gls{F_0}} \coloneqq \inf \set{\Mbf(S)}{\partial S = \partial T} = \|\partial T\|_{\KR}.
\]
In general spaces it holds $\|\partial T\|_{\KR} \le \Fbf_0(\partial T) $ for all $T \in \Mbf_1(X)$, but the inequality can be strict. 

\subsection*{Currents associated to Lipschitz curves}
We now recall a fundamental construction that allows us to associate a metric normal $1$-current to any Lipschitz curve $\theta \in \Lip([0,1],X)$. 

\begin{definition}[Integral curve current]\label{def:integral_curve}
Given a Lipschitz curve $\theta: [0,1] \to X$, we define its associated \emph{(integral) curve current} \gls{integral_curve} by
\[
\llbracket \theta \rrbracket(f, \pi) := \int_0^1 f(\theta(t)) \cdot (\pi \circ \theta)'(t) \, dt.
\]
\end{definition}
The derivative $(\pi \circ \theta)'$ exists almost everywhere by Rademacher's theorem, ensuring that the integral is well-defined. Furthermore, by the area formula in metric spaces (see~\cite{Kirchheim_Rectifiable}), the definition of $\llbracket \theta \rrbracket$ is independent of the specific Lipschitz parameterization of the curve. The mass of this current is bounded by the length of the curve: $\Mbf(\llbracket \theta \rrbracket) \le \ell(\theta)$, where 
\[
\text{\gls{length}} \coloneqq \int_0^1 |\dot \theta|(t) \, dt, \qquad \text{\gls{upper_gradient}}(t):= \limsup_{h\to 0_+} \frac{d(\theta(t+h), \theta(t))}{h}.
\]
 If $\theta$ is an injective curve, then $\Mbf(\llbracket \theta \rrbracket) = \ell(\theta)$. The fundamental theorem of calculus ensures that the boundary of this current is given by: 
 \[
 \partial \llbracket \theta \rrbracket = \delta_{\theta(1)} - \delta_{\theta(0)},
 \]
 where $\delta_x$ is the Dirac delta measure at point $x$.\medskip

Integral curve currents belong to the space of \textbf{rectifiable $1$-currents}, denoted \gls{rect_1}, which are currents whose mass is concentrated on a $1$-rectifiable set and whose mass measure is absolutely continuous with respect to the $1$-dimensional Hausdorff measure $\mathcal{H}^1$.\medskip

In a vector space, we can also define \textbf{polyhedral $1$-currents} as finite linear combinations of oriented segments. The current associated with a straight line segment $[x,y]$, from $x$ to $y$, is denoted by \gls{oriented_interval}. A polyhedral $1$-current is thus of the form:
\[
	T = \sum_{h=1}^r \eta_h \llbracket x_h, y_h \rrbracket, \qquad \eta_h \in \mathbb{R}, \; (x_h, y_h) \in X \times X.
\]
The space of such currents is denoted by \gls{poly_k}.

For a complete treatment of the theory of metric currents, we refer the reader to Ambrosio and Kirchheim's seminal work \cite{AK_Acta}, which contains all the basic theory of metric currents.

\subsection{The $1$-flat chain conjecture: a characterization in terms of connectedness}\label{sec:app}

Our first main result establishes sufficient and necessary conditions for the validity of the $1$-flat chain conjecture on complete metric spaces. This characterization is given in terms of a sufficient degree of "curve-like connectedness." To precisely formulate these findings, we first recall some basic concepts concerning one-dimensional rectifiable sets in metric spaces (see \cite[p. 251]{Federer1969}).

\begin{definition}[Fragments and $1$-rectifiable sets]\label{def:frag}
A subset $E \subset X$ is called \emph{$1$-rectifiable} if there exists a surjective Lipschitz map $\gamma: B \to E$, where $B$ is a bounded subset of $\R$. If, in addition, $B$ is compact, then $\gamma$ is called a \emph{Lipschitz fragment}. The canonical current associated with such a fragment is defined similarly to that of a Lipschitz curve:
\[
	\llbracket \gamma \rrbracket(f,\pi) \coloneqq \int_B f(\gamma) \cdot (\pi \circ \gamma)' \, dt.
\]
\end{definition}
With these definitions, we introduce a specific property for metric spaces that relates to how their $1$-rectifiable sets are covered by Lipschitz curves.

\begin{definition}[Curve rectifiable space]\label{def:cr_space}
A metric space $X$ is said to be \emph{curve rectifiable} if every $1$-rectifiable subset (or any fragment) in $X$ can be covered, up to an $\mathcal{H}^1$-negligible set, by a countable union of images of Lipschitz curves.
\end{definition}

This property is central to our first result:

\begin{thmx}\label{thm:main1}
Let $X$ be a complete and separable metric space. The following are equivalent:
\begin{enumerate}[label=\arabic*., leftmargin=2em, itemsep=0.7em,topsep=0.7em]
    \item $X$ is a curve-rectifiable space.
    \item For every $T \in \Mbf_1(X)$, there exists a sequence $(N_j) \subset \Nbf_1(X)$ such that 
    \[
    \Mbf(T - N_j) \longrightarrow 0 \quad \text{as $j \to \infty$}.
    \]
\end{enumerate}
\end{thmx}

Vector spaces, quasiconvex spaces (which we will define shortly), and geodesic spaces are all examples of curve rectifiable spaces. In particular, Theorem~\ref{thm:main1} extends prior work by Schioppa (\cite[Thm.~1.6]{Schioppa_2016}), who established a version of this result for Banach spaces equipped with a "finitely generated differential structure." Our result shows that the $1$-FCC holds on any Banach space, regardless of the dimension of its differential structure. A key distinction of our proof lies in its methodology: where Schioppa employed deep algebraic and measure-theoretic techniques, we rely solely on fundamental functional analysis principles, thus sidestepping a detailed geometric analysis of fragments or abstract differentiability properties.

Schioppa himself recognized the potential for his proof to adapt to spaces where fragments can be "filled-in" to produce Lipschitz curves (\cite[p. 3012]{Schioppa_2016}). While this property aligns with the intuitive notion of curve rectifiability, it is \emph{not} necessary for the 1-FCC's validity. For example, two disjoint closed intervals in $\R$ are curve rectifiable but do not allow for such "filling-in." The examples below further clarify this concept.

\subsubsection{Examples concerning curve rectifiability}

It is important to note that $1$-rectifiable sets and fragments can exist even if $X$ does not contain any non-trivial Lipschitz curves. The following simple example exhibits a metric space that is $1$-rectifiable but its only Lipschitz curves are constants.

\begin{example}[Fat Cantor set]\label{ex:fat}
Let $X \subset [0,1]$ be the Smith–Volterra–Cantor set of size $\alpha \in (0,1/3)$, also known as the "fat Cantor set." This is a closed and nowhere dense set on the real line with measure $\mathscr{L}^1(X) = 1-\alpha$. It is straightforward to see that every subset of $X$ is $1$-rectifiable as a subspace of $[0,1]$. However, as $X$ is totally disconnected, it contains no non-trivial Lipschitz curve $\theta: [0,1] \to X$, and thus it is not curve rectifiable. In fact, every Borel subset of $X$ defines a metric $1$-current in a canonical way, but the only normal $1$-current in $X$ is the zero current, that is, $\Nbf_1(X) = \{0\}$. 
\begin{figure}[h]
\centering
\fbox{\includegraphics[width=0.4\textwidth]{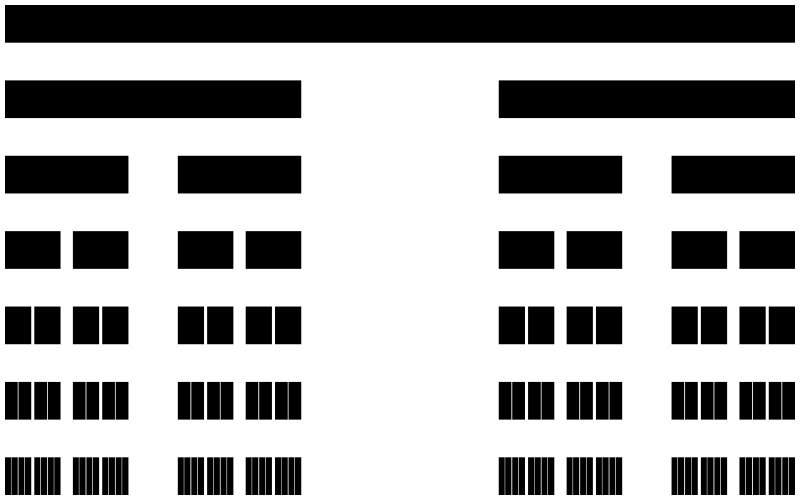}}
\caption{Iterative approximations of the Smith–Volterra–Cantor set.}
\end{figure}
\end{example}
\begin{remark}[Totally disconnected spaces]
    More generally, $\Nbf_1(X) = \{0\}$ for any totally disconnected complete space $X$. This is consequence of representation of normal $1$-currents as a superposition of Lipschitz curves, established by Stepanov and Paolni in~\cite{paolini2012decomposition,paolini2012decomposition2}.
\end{remark}
\newpage
Certain
geometries can also prevent curve rectifiability, even if the space is arc-connected:

\begin{example}[Almost transversal rectifiable set]\label{ex:transversal}
Let $C \subset [0,1]$ be the Cantor ternary set, and consider the space $X = (C \times (0,1)) \cup ([0,1] \times \{0\})$ as a subspace of $\R^2$. Notice that $X$ is complete and arc-connected. Let $u$ be the Cantor function and let $\Gamma_u$ be its graph. Since $u$ is monotone, the graph is $1$-rectifiable in $\R^2$. Moreover, it has length $\mathcal{H}^1(\Gamma_u) = 2$ (see, e.g., \cite[{p.143}]{ambrosioFunctionsBoundedVariation2000}). On the other hand, a projection argument shows that $E = X \cap \Gamma_u$ is $1$-rectifiable in $X$ with $\mathcal{H}^1(E) = 1$. For any $c \in C \cap (0,1)$, the segment $\{c\} \times [0,1]$ intersects $E$ precisely at the point $(c, u(c))$. Any Lipschitz curve on $X$ touching two such points must intersect each of their associated segments in a set of positive length; therefore, it can only intersect countably many of the segments. Since the Cantor set $C$ is uncountable, $E$ cannot be covered by a countable collection of Lipschitz curves. This proves that $X$ is not curve rectifiable.
\begin{figure}[h]
\centering
\fbox{\includegraphics[width=0.3\textwidth]{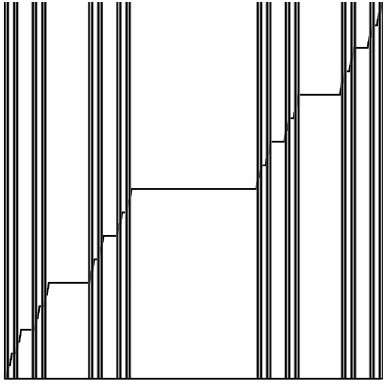}}	
\caption{The space $X$ in Example~\ref{ex:transversal} is depicted by the vertical lines, parametrized by the Cantor set, and the bottom segment connecting them. The $1$-rectifiable set $E$ is the intersection of the graph with the family of vertical lines.}
\label{fig:transversal}
\end{figure}
\end{example}

\subsection{Constructions in piecewise quasiconvex spaces} The proof of Theorem~\ref{thm:main1} relies on a sequence of structural results (Theorems~\ref{thm:approx_qc},~\ref{cor:cycle_cover} and~\ref{thm:rep}); see Fig.~\ref{fig:theorem-deps}. 
Our strategy begins by demonstrating the validity of the $1$-FCC when $X$ is \emph{piecewise quasiconvex} (but not necessarily complete), which is stronger than curve rectifiability, and will play a pivotal role in our constructions. In what follows, readers who wish to simplify the discussion may skip the definition of "piecewise quasiconvex" spaces and instead think in terms of the more familiar notions of "quasiconvex" (see below) or "Banach" spaces.

\begin{definition}[Piecewise quasiconvex space]\label{ass:con}
We say that $X$ is \emph{piecewise quasiconvex}\gls{piece_qc} if there exists a constant $c \ge 1$ with the following property: for any pair of points $x, y \in X$ and any $\eps > 0$, there exist pairs of points $(x_1, y_1), \dots, (x_{m}, y_m)$ in $X$ such that (see Fig.~\ref{fig:qc} below)
\[
		d(x,x_1) + d(y_{m}, y) + \sum_{i = 1}^{m-1} d(x_{i+1}, y_i) < \eps,
\]
together with Lipschitz curves $\theta_1, \dots, \theta_m : [0,1] \to X$ satisfying
\[
		\theta_i(0) = x_{i}, \qquad \theta_i(1) = y_i, \qquad \sum_{i=1}^m \ell(\theta_i) \le c d(x, y).
\]
\end{definition}

\begin{figure}[h]
\centering
\fbox{\includegraphics[width=0.4\textwidth]{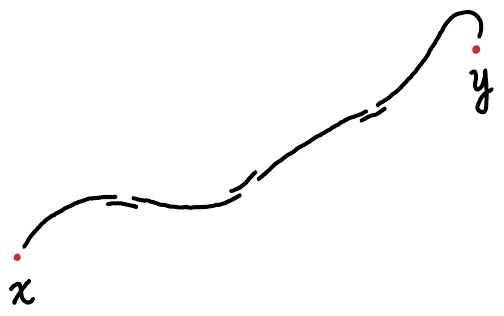}}
\caption{Collection of Lipschitz paths "almost" connecting the points $x$ and $y$.}
\label{fig:qc}
\end{figure}

\begin{remark}[Quasi-convex spaces]
 If $X$ is a complete metric space, then piecewise quasiconvexity implies classical \gls{qc}quasiconvexity: there exists $c \ge 1$ such that, for any $x,y \in X$, there exists a Lipschitz curve $\theta : [0,1] \to X$ with $\theta(1) = y$, $\theta(0) = x$, and $\ell(\theta) \le c d(x,y)$.
\end{remark}

Under this assumption we can construct rectifiable "primitives" of $\partial T$, almost minimizing a version of the Monge transport problem:

\begin{lemma}[Almost optimal primitives]\label{lem:OT} Let $X$ be a separable and piecewise quasiconvex space (with quasiconvexity constant $c\ge 1$) and let $T \in \Mbf_1(X)$. Then, for any $\delta \in (0,1)$, there exists a rectifiable current $R \in \mathcal R_1(X)$ satisfying
\[
	\partial R = \partial T \quad \text{and} \quad \Mbf(R) \le c(1 + \delta) \|\partial T\|_{\KR} .
\]
\end{lemma}

\begin{remark}
 The proof of Lemma~\ref{lem:discordia} is surprisingly simple, once one observes that the boundary $\partial T$ of any metric $1$-current belongs to the \emph{Arens--Eells space} $\text{\AE}(X)$ ---the Kantorovich--Rubinstein completion of dipole measures. These ideas, presented in Section~\ref{sec:prelims}, are pivotal for this work and  seem to form a novel approach when applied to the understanding of metric currents. 
 \end{remark}
 
 As a direct consequence of this construction, we obtain the following refinement of Theorem~\ref{thm:main1}:

\begin{thmx}\label{thm:approx_qc}
Let $X$ be separable and piecewise quasiconvex (with quasiconvexity constant $c \ge 1$) and let $T \in \Mbf_1(X)$ be a one-dimensional metric current. For any $\eps > 0$, there exists a cycle $C \in \Mbf_1(X)$ and and a sequence  $(P_j) \in \Nbf_1(X)$, where each $P_j$ is an integral current of the form
\begin{equation}\label{eq:quasi-poly}
	 	\sum_{i = 1}^{r} \eta_i\llbracket \theta_{i} \rrbracket, \qquad \eta_{i} \in \mathbb{R}, \quad \theta_{i} \in \Lip_1([0,1],X),
\end{equation} 
satisfying 
\[
\mathbf{M}(T - C - P_j) \longrightarrow  0 \quad \text{as } j \to \infty
\]
and
	\begin{align*}
		\mathbf{M}(P_j) &\le c \|\partial T\|_{\mathrm{KR}} + \eps.
	\end{align*}
Moreover,
\begin{enumerate}[topsep=0.5em,itemsep=0.5em, leftmargin=2em]
	\item[a)] If $X$ is a geodesic space, then the curves $\theta_i$ in~\eqref{eq:quasi-poly} can be chosen to be geodesics
	\item[b)] If $X$ is a Banach space, then the $P_j$ can be chosen to be polyhedral (cf.~\cite[Remark~3.4]{AM_flat}). \end{enumerate}
	 \end{thmx}
 \subsubsection{On the connectedness assumptions}\label{sec:comparisson} Many familiar spaces are piecewise quasiconvex: these include geodesic spaces, quasiconvex spaces, and Banach spaces.\medskip

We observe the following direct implication of Theorems~\ref{thm:main1} and~\ref{thm:approx_qc}:

\begin{corollary}\label{cor:con}
	 If $X$ is complete, separable and quasiconvex, then $X$ is curve rectifiable.
\end{corollary}

It is crucial to note that the converse of Corollary~\ref{cor:con} does not hold: not all curve rectifiable spaces are quasiconvex. 
To illustrate this, we provide an example of a complete, {arc-connected}, and separable metric space that is \textbf{curve rectifiable but not quasiconvex}.

\begin{example}\label{eq:cr_notqc}
Consider the piece-wise affine curves $\theta_n : [-2,2] \to \R^2$ defined by
\[
	\qquad \theta_n(t)
    \coloneqq \begin{cases}
		n(2 + t) & \text{if $-2 \le t < 0$} \\
		n(2 - t) & \text{if \phantom{-} $0 < t \le 2$}
	\end{cases}, \qquad n = 1,2,\dots
\]
Define the metric space $(X,|\cdot|_{\R^2})$, where $X = \bigcup_{n = 1}^\infty  \im \theta_n$.
Clearly, $X$ is complete, arc-connected, and separable as a subspace of $\R^2$. Moreover, $X$ is the union of images of countably many Lipschitz curves. However, $X$ is not quasiconvex since $|(-1,n) - (1,n)|_{\R^2} = 2$ while the shortest path-length from $(-1,n)$ to $(1,n)$ is $2\sqrt{1 + n^2}$.  
\end{example}

The following diagram summarizes the relations between arc-connectedness and the recently introduced connectedness assumptions on complete spaces:

\begin{figure}[h]
\centering
\begin{tikzcd}[column sep=1cm, row sep=1cm, shorten=4mm]
 & \textbf{\textup{quasiconvex}}
 \arrow[ddl,shift left=1ex,thick,Rightarrow]
 \arrow[ddr,shift right=1ex, shorten=2mm,thick,Rightarrow]
 \arrow[ddr,Leftarrow,"\;\;\;\textup{\scriptsize{Example~\ref{eq:cr_notqc}}}",shift left,"/" marking,start anchor={[xshift=1.4ex]},
start anchor={[yshift=-0.7ex]},end anchor={[xshift=1.4ex]},
end anchor={[yshift=-0.7ex]}, shorten=2mm,thick] & \\
 & & \\
\textbf{\textup{arc-connected}}
\arrow[rr, Rightarrow,"\hspace{-9ex} \textup{\scriptsize Example~\ref{ex:transversal}}","/" marking,shift left=1ex, shorten=2mm,thick]
\arrow[uur,Rightarrow,"\textup{\scriptsize{Example~\ref{ex:transversal}}}\;",shift left,"/" marking,start anchor={[xshift=-1.4ex]},
start anchor={[yshift=-0.7ex]},end anchor={[xshift=-1.4ex]},
end anchor={[yshift=-0.7ex]}, shorten=2mm,thick] &  & 
\textbf{\textup{curve rectifiable}}
\arrow[ll, Rightarrow,"\textup{\scriptsize \hspace{17ex} $X = [0,1] \cup [2,3]$}","/" marking,shift left=1ex, shorten=2mm,thick] 
\end{tikzcd} 
\end{figure}

\subsection{Filling currents into cycles}

Building upon the approximation results, we now explore how metric currents relate to cycles and rectifiable currents, drawing parallels with recent advancements in Euclidean spaces. Recently, Alberti and Marchese demonstrated~\cite{AM_flat} (see also~\cite{alberti1991lusin,arroyo2023lebesgue,DMG}) that every \emph{$k$-dimensional} FF flat chain with finite mass, $T \in \mathscr F_k(\R^d)$, can be ``completed'' into a cycle $C$ by adding a rectifiable current $R$, with the mass of $R$ being roughly the same as that of $T$. Simultaneously to this work (announced in~\cite{UniPi}), De Pauw~\cite[Thm. 2.12]{DP} improved on these approximations by showing that
\begin{equation}\label{eq:estimate1AM}
	 T = C + R, \qquad \mathbf{M}(R) \le \|\partial T\|_{\KR} + \eps \quad \text{for all $T \in \mathscr F_1(\R^d)$.}
\end{equation}
The validity of the $1$-flat chain conjecture in sufficiently well-connected spaces suggests that similar statements might hold for general metric currents. The next result demonstrates that~\eqref{eq:estimate1AM} holds for $1$-currents in piecewise quasiconvex spaces. 
Because our approximations rely on an optimal transport construction, we are also able to derive mass estimates in terms of the Kantorovich-Rubinstein norm of the boundary:

\begin{corollary}[Filling by rectifiable currents]\label{cor:filling}
Let $X$ be a separable and piecewise quasiconvex space  (with constant $c \ge 1$) and let $T \in \mathbf{M}_1(X)$ be a one-dimensional metric current. Then, for any given $\eps > 0$, there exists $C \in \Mbf_1(X)$ and $R \in \mathcal{R}_1(X)$ such that
\begin{enumerate}[itemsep=0.2em, leftmargin=2.5em,label=\arabic*.]
	\item $\partial C = 0$
	\item $T = C + R$
	\item $\mathbf{M}(C) \le \mathbf{M}(T) + c\|\partial T\|_{\mathrm{KR}} + \eps$
	\item $\mathbf{M}(R) \le c\|\partial T\|_{\mathrm{KR}} + \eps$
\end{enumerate}
\end{corollary}

We can establish an even stronger result if we consider Banach spaces of dimension at least two. Every one-dimensional metric current can be expressed as the Borel restriction of a cycle. This generalizes a result by Alberti and Marchese for flat chains in Euclidean space to a broader class of metric spaces.

\begin{thmx}[Covering by cycles in a Banach space]\label{cor:cycle_cover}
Let $X$ be a separable Banach space with $\dim(X) \ge 2$ and let $T \in \mathbf{M}_1(X)$ be a one-dimensional metric current. Then, for any given $\eps > 0$, there exists a cycle $C \in \Mbf_1(X)$ and a Borel set $B \in \mathcal B(X)$ such that
\begin{enumerate}[itemsep=0.2em, leftmargin=2.5em,label=\arabic*.]
	\item $\partial C = 0$
	\item $\spt C \subset \mathrm{conv}\{x : \dist_X(x, \spt \|\partial T\|) \le \eps\}$
	\item $\|T\|$ and $\|R\|$ are mutually singular; in particular, $T = C \mres B$ and $R = C\mres B^c$
	\item $\Mbf(C) \le \Mbf(T) + \|\partial T\|_{\KR} + \eps$ \label{eq:optimalCC}
\end{enumerate}
\end{thmx}

It's important to note that, even in compact subsets of $\R$, the only cycles are constants, meaning the previous characterization cannot hold (although, for one-dimensional spaces, a similar version holds with normal currents instead of cycles). 

The next example shows that the estimate in Corollary~\ref{cor:cycle_cover}.\ref{eq:optimalCC} is optimal:

\begin{example}[Optimality of the filling]
Consider $\theta$, a clockwise parametrization of the upper semi-circle in $\mathbb{R}^2$, and the current $T = \llbracket \theta \rrbracket \in \Mbf_1(\mathbb{R}^2)$. Observe that $\mathbf{M}(T) = \pi$ (length of the semi-circle) and $\|\partial T\|_{\mathrm{KR}} = 2$ (distance between its endpoints). The smallest cycle ``covering'' $T$ is the clockwise oriented upper half-circle $C = \llbracket \theta \rrbracket + \llbracket e_1, -e_1 \rrbracket$. Thus, $\mathbf{M}(C) = \pi + 2 = \mathbf{M}(T) + \|\partial T\|_{\mathrm{KR}}$.
\end{example}

\subsection{$\mathbf{SBV}$-representation of one-dimensional currents}

For the remainder of this section, we assume $X$ is a complete and separable metric space. To streamline our notation, we will denote the open interval $(0,1) \subset \R$ with the usual Euclidean metric as \gls{I}.

To fully describe the structure of metric currents in general complete metric spaces, we exploit a fundamental fact: any separable metric space can be isometrically embedded into a separable Banach subspace of $\ell^\infty$. Since metric currents and Lipschitz maps behave well under isometries, this allows us to transfer our problem to a more structured Banach space, where we can then apply the powerful tools developed in Corollary~\ref{cor:cycle_cover}.

This approach enables us to establish a Smirnov-type decomposition for metric $1$-currents in arbitrary complete metric spaces, critically, \textbf{without any connectedness assumptions}. This extends the work of Paolini and Stepanov~\cite{paolini2012decomposition,paolini2012decomposition2}, who previously provided such a decomposition for normal (cyclic and acyclic) metric one-currents on complete metric spaces. Their work represented normal $1$-currents as a superposition of currents associated with Lipschitz curves with constant speed, topologized as a subspace of $C(I,X)$. Our approach generalizes this by representing arbitrary metric $1$-currents as a superposition of currents associated with \textbf{compact} curves in $\SBV(I,X)$, topologized as a subspace of $\mathbb D(I,X)$, the space of \emph{Càdlàg functions} (right-continuous functions with existing left limits). 

Here, $\mathbb D(I,X)$ is equipped with the \emph{Skorokhod metric} 
\begin{equation}\label{eq:skorokhod}
    d_S(u_1,u_2) = \inf_{\lambda \in \Lambda} \max \left\{ \| \lambda - \mathrm{id}\|_\infty,\| d(u_1, u_2(\lambda))\|_\infty\right\}
\end{equation}
where $\Lambda$ is the set of strictly increasing, continuous bijections  $\lambda : [0, 1] \to [0, 1]$. 
This metric, is particularly well-suited for paths with jumps and coincides with the topology of uniform convergence in $C(I,X)$. 
Convergence in this space does not necessarily imply pointwise convergence, however, it does imply convergence in $L^1((0,1),X)$. 

\subsection*{Compact $BV$-curves} Let's first introduce the main properties and definitions of the curve spaces central to this representation. A Borel curve $u: I \to X$ belongs to $\BV(I,X)$ if $u \in L^1(I,X_c)$ for some compact metric subspace $X_c \subset X$ and its \emph{essential total variation} is finite:
\[
\inf_{v = u \, \text{a.e.}} \left\{ \sup_{0 < t_1 < \cdots < t_m < 1} \sum_{i = 1}^{m-1} d(v(t_{i+1}), v(t_i)) \right\} < \infty.
\]
This is equivalent to the existence of a least finite Borel measure $|Du|$, the \emph{total variation measure}, such that $|Du|(B) \ge |Df(u)|(B)$ for all Borel sets $B \subset I$ and $f \in \Lip_1(X)$. A key property of $u \in \BV(I,X)$ is that $u$ has well-defined left and right limits $u^-,u^+$ at every point. Moreover, its discontinuity set $S_u \subset I$ is at most countable, and the total variation measure can be decomposed into mutually singular measures as:
\[
|Du|(B) = \int_{B} |\dot u| \, dt + \int_{B \cap S_u} d(u^-,u^+) \, d\mathcal H^0 + |D^cu|(B)\,,
\]
where the \emph{approximate metric differential}
\begin{equation}\label{mderiv}
|\dot u|(t) \coloneqq \lim_{|h| \to 0} \frac{d(u^+(t+h),u^+(t))}{|h|}
\end{equation}
exists $\mathscr L^1$-a.e. and defines summable function on $I$. Here, $|D^c u|$ is the \emph{cantorian part}, a singular measure vanishing on countable sets, often representing the "fractal behavior" of $BV$ functions. A curve $u \in \SBV(I,X)$ is a $\BV$-curve for which the Cantor part $|D^c u|$ vanishes. 

\subsection*{Currents associated with $\SBV$-curves}
Before stating our representation result, we extend the fundamental correspondence $\theta \mapsto \llbracket \theta \rrbracket$ for $\theta \in \Lip([0,1],X)$ to a correspondence $\SBV(I,X) \to \Mbf_1(X)$ for $\SBV$-curves. Specifically, every $\SBV$-curve gives rise to a rectifiable metric $1$-current corresponding to its \emph{absolutely continuous part}:
\[
\llbracket u \rrbracket^a(f,\pi) \coloneqq \int_0^1 f(u)\, dD^a(\pi \circ u)\, dt, \qquad (f, \pi) \in \Lip_b(X) \times \Lip(X).
\]
As detailed in Section~\ref{sec:BV}, this functional defines a metric $1$-current in $X$. The mass of this current is related to the length of the curve: with a slight abuse of notation, we have $\Mbf(\llbracket u \rrbracket^a) \le \ell(u)$, where
\[
\ell(u) \coloneqq \int_0^1 |\dot u|(t)\, dt.
\]
As it is the case for Lipschitz curves, equality holds on the left-hand side if $u$ is injective. If moreover $u \in \SBV(I,X)$ has constant metric speed (approximate differential), then $\im u \subset X$ is $1$-rectifiable and 
\[
    \|\lpr{u}^a\| = \mathcal H^1 \mres \im u \quad \text{as measures on $X$}.
\]
The boundary of $\llbracket u \rrbracket^a$ is a collection of dipole measures, given by
\[
\partial \llbracket u \rrbracket^a = \delta_{u^-(1)} - \delta_{u^-(0)} \; - \; \sum_{t \in S_u} (\delta_{u^+(t)} - \delta_{u^-(t)}).
\]
We emphasize that $\llbracket u \rrbracket^a$ may fail to be a normal current. This is because its boundary mass $\Mbf(\partial \lpr{u}^a)$ becomes infinite whenever $u$ has infinitely many discontinuities. The Kantorovich norm, however, remains bounded:
\[
\|\partial \llbracket u \rrbracket^a\|_{\text{KR}} \le 2 + \int_{S_u} d(u^-,u^+) \, d\mathcal H^0 <\infty.
\]
All of these properties will be discussed in detail in Section~\ref{sec:BV}.


\subsection*{The representation}

Our representation result states that any metric $1$-current can be understood as a superposition of integral currents associated with \textbf{injective} $\SBV$-curves, without mass cancellations.  The contribution of each curve $u \in \SBV(I,X)$ with constant speed is restricted to the metric current $\llbracket u \rrbracket^a$ associated with its \emph{absolutely continuous part}. The "Borel holes" (cf. Corollary~\ref{cor:filling}) of the current are modeled by the \emph{jump discontinuities} in an almost optimal sense with respect to the Kantorovich--Rubinstein norm.

With these considerations, we finally state our decomposition result:

\begin{thmx}[$\SBV$-representation]\label{thm:rep} Let $X$ be a complete and separable metric space. Then,  $T \in \Mbf_1(X)$ if and only if there exists a finite Borel measure $\eta$ on the Càdlàg space $\mathbb D(I,X)$, concentrated on the set \gls{Theta_SBV} of injective curves in $\SBV(I,X)$ with constant metric speed, satisfying the following:
\begin{enumerate}[itemsep= 0.2em,leftmargin=2.5em,label=\arabic*.\,,topsep=1em]
	\item\label{item:rep_1} For any $(f,\pi) \in \Lip_b(X) \times \Lip(X)$, 
	\[
	T(f,\pi) = \int_{\mathbb D(I,X)} \llbracket u \rrbracket^a(f,\pi) \, d\eta(u).
	\]
	\item \label{item:rep_2}
The representation extends to a representations of the mass measures: for any Borel set $B \subset X$
\begin{align*}
	\|T\|(B) & = \int_{\mathbb D(I,X)} \|\llbracket u \rrbracket^a\|(B) \, d\eta(u)  \\
	& = \int_{\mathbb D(I,X)} \mathcal H^1(B \cap \im u) \, d\eta(u).
	\end{align*}
   In particular, it is free of mass cancellations:
	\begin{align*}
	\Mbf(T) & = \int_{\mathbb D(I,X)} \Mbf(\llbracket u \rrbracket^a) \, d\eta(u)  \\
	& = \int_{\mathbb D(I,X)} \ell(u) \, d\eta(u).
	\end{align*}
\item Moreover,  for any $\eps >0$, there exists $\eta$ satisfying~\ref{item:rep_1}-\ref{item:rep_2} and
	\[
	\int_{\mathbb D(I,X)}\left( \int_{S_u} d(u^-,u^+) \, d\mathcal H^0\right) \, d\eta(u) \le \|\partial T\|_{\KR} + \eps.
	\]
\end{enumerate}
\end{thmx}

When $X=\R^d$,  Theorem \ref{thm:rep}  specializes into the following representation result :

\begin{corollary}[Structure of one-dimensional flat chains] Let $T \in \mathscr F_1(\R^d)$ be an FF flat chain with finite mass. Then, there exists a finite Borel measure $\eta$ on $\mathbb D([0,1])^d$, concentrated on the injective constant speed curves in $SBV([0,1])^d$, and satisfies
 \begin{align*}
 T(\omega) & = \int_{\mathbb D([0,1])^d} \left( \int_0^1 \dpr{\omega \circ u, u'} \, dt \right) \, d\eta(u) \quad \text{for all $\omega \in C^1_c(\R^d;\R^d)$,}\\ 
 \Mbf(T) & = \int_{\mathbb D([0,1])^d} \ell(u)  \, d\eta(u).
 \end{align*}
 Moreover, for $\eta$-a.e. $u$, the pointwise derivative $u'$  exists a.e. in $(0,1)$ and satisfies $|u'| \equiv \ell(u)$.
\end{corollary}

Following Schioppa~\cite{Schioppa_2016}, We write \gls{frag} to denote the space of fragments $\gamma : B \subset [0,1] \to X$, topologized as a subspace of $\mathcal K([0,1]\times X)$. Here, for a metric space $Y$, \gls{K(Y)} denotes the class of compact subsets of $Y$ endowed with the Hausdorff distance $d_H$. 
We construct an $\eta$-measurable identification $\Theta_{\SBV}(X) \embed \Frag(X) : u \mapsto \gamma_u$ satisfying $\lpr{u}^a = \lpr{\gamma_u}$. Therefore, a direct consequence of the $\SBV$-representation theorem is the following fragment representation, which might appeal to a more fundamental metric space theory setting:

\begin{corollary}[Fragment representation]\label{cor:fragment} Let $X$ be a complete and separable metric space. Then,  $T \in \Mbf_1(X)$ if and only if there exists a finite Borel measure $\mu \in \Frag(X)$ satisfying
 \begin{align*}
 T(\omega) & = \int_{\Frag(X)} \lpr{\gamma}(\omega) \, d\mu(\gamma) \quad \text{for all $\omega \in \mathscr D^1(X)$,}\\ 
 \Mbf(T) & = \int_{\Frag(X)} \ell(\gamma)  \, d\mu(\gamma).
 \end{align*}
\end{corollary}

\subsubsection{Comments on the choice of the topology}
Our use of the \emph{Càdlàg topology} is a deliberate choice, as it offers a natural extension of the uniform metric topology from continuous functions. It is worth noting, however, that our methods are robust enough to extend Theorem~\ref{thm:rep} to other topologies on $\Theta_{\SBV}(X)$. Three interesting examples are:
\begin{enumerate}[itemsep= 0.5em,leftmargin=2.5em,label=\arabic*.\,,topsep=1em]
    \item The \emph{length-strict topology} on $\SBV([0,1],X)$, generated by the metric
    \[
    d_1(u,v) \coloneqq \int_0^1 d(u(t),v(t))\,dt + |\ell(u) - \ell(v)| + \left| |D^j u|((0,1)) - |D^j v|((0,1)) \right|.
    \]
    \item The \emph{extended graph topology} on $\SBV([0,1],X)$, generated by the metric
    \[
    d_2(u,v) \coloneqq \int_0^1 d(u(t),v(t))\,dt + d_H(\overline{\Gamma_u},\overline{\Gamma_v}),
    \]
    where $\overline{\Gamma_u} \coloneqq \overline{\set{(t,x)}{t \in (0,1), x \in [u^-(t),u^+(t)]}} \subset [0,1] \times X$ is the extended graph of $u$. 
    
    \item A \emph{stronger Skorokhod metric}, which is a refined version of the standard one. Instead of allowing any strictly increasing continuous reparameterization $\lambda$, it restricts $\lambda$ to be bi-Lipschitz with constants bounded by $2$. This choice imposes more regularity on the re-parameterizations, leading to a stronger topology.
\end{enumerate}
Each of these metrics ensures the continuity of the embedding $\Theta_{\SBV}(X) \embed \Frag(X)$, however, neither of them is comparable with the Càdlàg topology. 

\subsection*{Organization}The diagram below summarizes the logical dependencies among the main results.

\begin{figure}[h]
\centering

\begin{tikzpicture}[
    node distance=0.8cm and 0.8cm,
    theorem/.style={minimum width=2cm, minimum height=1.5cm, align=center},
    doublearrow/.style={->, double, double equal sign distance, line width=0.5pt, >=Stealth}
]

\node[theorem] (B) {\textbf{\small{Theorem B}}};
\node[theorem, above left=of B] (C) {\textbf{\small{Theorem C}}};
\node[theorem, above right=of B] (D) {\textbf{\small{Theorem D}}};
\node[theorem, above left=of D] (A) {\textbf{\small{Theorem A}}};

\draw[->, >=Stealth,double,line width=0.5pt] (B) -- (C);
\draw[->, >=Stealth,double,line width=0.5pt] (C) -- (D);
\draw[->, >=Stealth,double,line width=0.5pt] (D) -- (A);

\coordinate (Center) at ($(C)!0.5!(D)$);
\coordinate (CB) at ($(C)!0.5!(B)$);

\node[blue]  at ($(B) + (0,+0.4)$) {\small{$X$ piecewise quasiconvex}};
\node  at ($(B) + (0,-0.4)$) {\small{$1$-FCC}};
\node  at ($(A) + (0,-0.4)$) {\small{$1$-FCC characterization}};
\node[blue] at ($(A) + (0,0.4)$) {\small{$X$ curve-rectifiable}};
\node[blue]  at ($(C) + (0,0.4)$) {\small{$X$ Banach space}};
\node  at ($(C) + (0,-0.4)$) {\small{filling of currents}};
\node  at ($(D) + (0,-0.4)$) {\small{$SBV$-representation}};
\node[blue]  at ($(D) + (0,+0.4)$) {\small{$X$ complete}};
\end{tikzpicture}
\label{fig:theorem-deps}
\end{figure}

This paper is organized as follows. Section~\ref{sec:prelims} introduces the necessary background on Lipschitz spaces and their duality with the Arens–Eells space. Section~\ref{sec:OT} presents the key optimal transport construction used to define primitives for boundaries of metric $1$-currents (see Lemma~\ref{lem:OT}). Section~\ref{sec:qc} contains the proof of the approximation of one-dimensional metric currents by normal currents in piecewise quasiconvex spaces (Theorem~\ref{thm:approx_qc}), and establishes the hole-filling property for metric currents (Theorem~\ref{cor:cycle_cover}). Section~\ref{sec:BV} introduces compact $BV$- and $SBV$-curves and their associated currents. We also develop the $BV$-parametrization of sets in the real line, which plays a key role in assembling fragments into $SBV$-curves. Section~\ref{sec:rep} proves the $SBV$-curve decomposition for metric $1$-currents in arbitrary separable metric spaces. Section~\ref{sec:char} uses the $SBV$-representation to establish the equivalence between curve rectifiability and the validity of the $1$-flat chain conjecture. We also include a \textit{List of Symbols} to assist readers unfamiliar with some of the technical vocabulary introduced throughout. The \textit{Appendix} contains a few technical measurability lemmas used in the main proofs.

\subsection*{Acknowledgements} This research was supported by the European Research Council (ERC) Starting Grant "ConFine" (Grant No. 01078057). AR was supported as a grant holder and GB was funded in his visits to the University of Bonn and the University of Pisa, where this project began.

We are deeply grateful to many colleagues whose insights and discussions enhanced our understanding of the theory of currents in metric spaces. In particular, we thank Peter Gladbach for sharing ideas from stochastic analysis that suggested SBV curves might play a role in representing flat chains, helping guide us toward our conjecture in Theorem~\ref{thm:rep}. We also thank Giovanni Alberti and Giacomo del Nin for a stimulating conversation that led to Example~\ref{ex:transversal}. Finally, we are indebted to Ivan Violo for suggesting a simple proof of Lemma~\ref{lem:Hausdorff}, and to Andrea Merlo for pointing us to measurability techniques related to fragment approximation.


\section{Preliminaries}\label{sec:prelims}

\subsection{Functional analysis}
Given a normed vector space $E$, we denote its dual space, the set of all continuous linear functionals, by $E^*$. The following remark is crucial for our purposes (see~\cite[Proposition 3.14]{Brezis}):

\begin{proposition}[Weak-$*$ continuous maps]\label{rem:identification_continuity}
If $E$ is Banach and $\varphi : E^* \to \mathbb{R}$ is weak-$*$ continuous, then there exists $x_0 \in E$ such that $\varphi(f) = f(x_0)$ for all $f \in E^*$.
\end{proposition}

There is a canonical isometric embedding $J : E \to E^{* *}$ given by the evaluation map $x \mapsto e_x$. Therefore, we often identify $E$ with a subspace of its bidual $E^{* *}$. 
If $E$ is not complete, we denote its norm-completion by $\overline{E}$. It is a fundamental fact that dual spaces are stable under completion: $E^* = \overline{E}^*$. 
We also make use of the following classical results: 

\begin{theorem}[Banach--Alaoglu--Bourbaki, {\cite[Thm. 3.16]{Brezis}}]
	Let $E$ be a Banach space. The unit ball $B_{E^*} = \set{f \in E^*}{\|f\|_{E^*} \le 1}$
	is compact in the weak-$*$ topology.
\end{theorem}

\begin{remark}
	If $E$ is separable, then $B_{E^*}$ is metrizable in the weak-$*$ topology (see~\cite[Theorem 3.28]{Brezis}).
\end{remark}
 
\begin{theorem}[Banach--Dieudonné--Krein--Šmulian, {\cite[Theorem 3.33]{Brezis}}]
Let $V$ be a Banach space and let $C \subset E^*$ be convex set. Suppose that for every integer $n$ the set $C \cap (nB_{E^*})$ is weak-$*$ closed. Then, $C$ is closed in the weak-$*$ topology.\end{theorem}

This has the following immediate corollaries:

\begin{corollary}
Let $V$ and $W$ be Banach spaces and $T : V^* \to W^*$ be a linear map. Then $T$ is weak-$*$ continuous if and only if its restriction to $B_{V^*}$ is weak-$*$ continuous.
\end{corollary}

Paired with the Banach--Alaoglu theorem, it yields the following weak-$*$ continuity criterion: 

\begin{corollary}\label{lem:porfin} Let $V$ and $W$ be Banach spaces and $T : V^* \to W^*$ be a linear map.
If $V$ is a separable Banach space and $T$ is sequential weak-$*$ continuous, then $T$ is weak-$*$ continuous (even if the weak-$*$ topology is not metrizable).
\end{corollary}

\subsection{Lipschitz functions in a metric space}\label{sec:Lip}
Let $X = (X,d)$ be a metric space. The vector space of real-valued Lipschitz functions, denoted as Lip$(X)$, consists of all functions $f:X\to \R$ satisfying:
\[
\Lip(f) = \sup_{x \neq y} \frac{|f(x) - f(y)|}{d(x, y)} < +\infty
\]
We denote by $\Lip_1(X)$ the space of all Lipschitz maps with Lipschitz constant $\Lip(f) \le 1$. If $X$ is a pointed metric space, meaning it has a distinguished "base point" $x_0$ then we define \gls{Lip_0} to be the space of all Lipschitz functions $f:X \to \R$ vanishing at $x_0$. Endowed with the Lip$(\cdot)$ semi-norm, Lip$_0(X)$ is a Banach space. Note that the Banach space structure of Lip$_0(X)$ does not depend on the base point.

\subsection{${\Lip_0(X)}$ as a dual Banach space} In a pointed metric space $X$ (with base point $x_0$), the space Lip$_0(X)$ is a dual Banach space (in the sense that it has a pre-dual) of a particular space of measures $\text{\AE}(X)$, often referred to as the `Arens-Eells space' or the `Lipschitz-free space' of $X$. This space is constructed as follows: an atom in $X$ is a measure of the form $\delta_x - \delta_y$, where $x,y \in X$. A molecule in $X$ is a finite sum $m = \sum_j 
 \eta_j \delta_{x_j}$, where each $x_j$ belongs to $X$ and $\eta_j \in \R$ satisfy $\sum_j \eta_j = 0$.  Note that atoms form a linear basis for the molecules. Also note that a molecule $m$ may not have a unique representation in terms of the sums described above. The metric on $X$ defines a norm on a molecule $m$ by:
\begin{equation}\label{eq:AEnorm}
	\rho(m) \coloneqq \inf \set{\sum_j |\eta_j| d(x_j,y_j)}{m = \sum_j \eta_j (\delta_{x_j} - \delta_{y_j})}
\end{equation}
where the infimum is taken over all atomic representations of $m$. The Arens--Eells space \gls{AE} is then defined as the completion of all molecules with respect to the metric $\rho$ introduced above.  
\begin{remark}
$\textnormal{\AE}(X)$ is separable if and only if $X$ is separable.
\end{remark}

Next, we elucidate the relationship between the dual of the Arens-Eells space, $\text{\AE}(X)^*$, and $\Lip_0(X)$. 
 First, notice the fact that $d$ is a metric on $X$ implies that the topology induced by $\rho$ on $\text{\AE}(X)$ is Hausdorff. 
Arens and Eells established the following crucial identification: any function $f \in \Lip_0(X)$ defines a functional as
\begin{equation}\label{eq:AEmap}
	\Psi(f)(m) \coloneqq \dpr{m,f}.
\end{equation}
Through this identification, the following holds (see~\cite[{Thm. 3.3}]{Weaver_Book} for a modern exposition):
\begin{theorem}[Arens--Eells~{\cite{AE}}]\label{thm:AE} Let $X$ be a pointed metric space.
	The linear map $\Psi$ in~\eqref{eq:AEmap} defines a one-to-one correspondence between $\Lip_0(X)$ and the linear functionals $\varphi : \textnormal{\AE}(X) \to \R$ that are continuous with respect to the norm $\rho$. Moreover, this map defines a linear isometry since
	\[
		\Lip(f) =  \sup_{\substack{m \in \textnormal{\AE}(X),\rho(m) \le 1}} \Psi(f)(m)\,.
	\]
	In particular, $\textnormal{\AE}(X)^* \cong \Lip_0(X)$ is a dual Banach space with the $\Lip(\cdot)$ (semi-)norm. 
\end{theorem}

\begin{remark}\label{rem:weak_star}A direct consequence of this is that an equi-Lipschitz sequence $(f_j) \subset \Lip_0(X)$ converges weak-$*$ to $f$ (relative to the pre-dual $\textnormal{\AE}(X)$) if and only if it converges pointwise, i.e., 
\[
	f_j(x) \to f(x) \quad \forall x \in X.
\]
Therefore, in the following we will write $f_j \toweakstar f$ to denote that $(f_j) \subset \Lip_0(X)$ is uniformly equi-Lipschitz and converges pointwise to $f$. 
\end{remark}

\begin{remark}\label{remm}
	If $X$ is separable and $\lambda : \Lip_0(\R) \to \R$ is sequentially weak-$*$ continuous, then $\lambda$ is weak-$*$ continuous. In particular, $\lambda$ can be identified (under the canonical bidual isometry) with an element of   $\textnormal{\AE}(X)$. 
	\end{remark}
	
\begin{corollary}\label{cor:completion}
	Let $\hat X$ be the completion of $X$. Then,  $\Lip_0(X)$ is isometrically isomorphic with $\Lip_0(\hat X)$. In particular, $\textnormal{\AE}(X)$ is isometrically isomorphic with $ \textnormal{\AE}(\hat X)$.
\end{corollary}
\begin{proof}
By standard completion methods, there exists a unique canonical embedding $i : X \embed \hat X$. This embedding gives rise (with a possible abuse of notation) to an embedding $i : \Lip_0(X) \embed \Lip_0(\hat X)$ given by
\[
	if(\hat x) \coloneqq \lim_{x_j \to \hat x} f(x_j), \qquad \hat x \in \hat X,
\]
where the pointed Lipschitz spaces are pointed at some point $x_o \in X$ and are endowed with the strong Lipschitz semi-norm. Notice that $if$ is well-defined since, due to the Lipschitz continuity of $f$, the limit on the right-hand side does not depend on choice of the sequence $x_j \to \hat x$. This way, using the density of $X \embed \hat X$ it is straightforward to show that $i$ is injective. Similarly, one can show that the restriction operator $r : \Lip(\hat X) \to \Lip(X)$ given by $r \hat f = \hat f \circ i$ is also one-to-one and hence it follows that both $i$ and $r$ inverses of each other. Moreover, by definition and by density, it follows that both $i$ and $r$ are isometries. In particular, $\Lip_0(X)$ is isometrically isomorphic with $\Lip_0(\hat X)$ when equipped with the Lipschitz semi-norm. The isometry between the Arens--Eells spaces follows from Theorem~\ref{thm:AE}.\end{proof}
	
The following characterization will play a crucial role in our constructions:
\begin{corollary}[Sequential continuity characterization]\label{lem:porfin}
	Let $X$ be a separable metric space and let $\phi : \Lip(X) \to \R$ be a linear functional satisfying the following properties:
	\begin{enumerate}
		\item $\phi$ vanishes on constants: $\phi(1_X) = 0$
		\item $\phi$ is sequentially weak-$*$ continuous, i.e.,
		\[
		f_j \toweakstar f \quad \Rightarrow \quad \phi(f_j) \to \phi(f) \qquad \forall f \in \Lip(X).
	\]
	\end{enumerate} 
	Then, there exists $\lambda \in \textnormal{\AE}(X)$ satisfying
	\begin{equation}\label{eq:porfin2}
		\phi(f) = \dpr{\lambda, f} \qquad \forall f \in \Lip(\R).
	\end{equation}
\end{corollary}
\begin{proof}
	Since $\phi$ vanishes on the constants and there is a natural isometric isomorphism between $\Lip(X)/\R$ and $\Lip_0(X)$ with $X$ as a pointed space, we may without any loss of generality consider $\phi$ as a functional on $\Lip_0(X)$. We recall that if $X$ is separable, then $\text{\AE}(X)$ is separable, and therefore $\text{\AE}(X)$ is a separable Banach space. It then follows from Remark~\ref{remm} that $\lambda$ is weak-$*$ continuous and hence there exists an element in the Banach pre-dual of $\Lip_0(X)$ representing it, i.e., there exists $\lambda \in \text{\AE}(X)$ satisfying~\eqref{eq:porfin2} for all $f \in \Lip_0(X)$. The fact that this holds for all $f \in \Lip(X)$ follows from the fact that $\phi$ vanishes on constants.\end{proof}
	
	\begin{remark}
	If $X$ is separable, then the weak-$*$ topology of $\Lip(X)/\R$ is metrizable on bounded sets (sets of maps with bounded Lipschitz constant) and coincides with  the topology of pointwise convergence.
\end{remark}

\subsection{The Kantorovich completion of signed measures}\label{sec:completion}

This subsection will not be used directly in our proofs, but it provides context to some of the ``optimal transport'' ideas which underpin our constructions. Let us assume that $X$ is a real separable Banach space and denote by
 $\mathscr M(X)$  the Banach space of real-valued, finite Borel measures $\lambda$ on $X$, equipped with the total variation norm $|\lambda|(X)$. We then define $\mathbf X_{0,1}(X)$ to be  the subspace of measures with zero average and finite first moment, i.e., those satisfying $\lambda(X)= 0$ and $\int |x| \, d\lambda(x) < + \infty$. Notice that $\mathbf X_{0,1}(X) \subset \mathscr M(X)$.  This space is endowed with the \emph{Kantorovich--Rubinstein}  norm:
\[
	\text{\gls{KR}} \coloneqq \sup_{f \in \Lip_1(X)} \int f \, d\lambda.
\]
The well-known Kantorovich--Rubinstein Duality Theorem establishes a crucial link to optimal transport. It states that the \emph{Monge distance} $W_1(\mu,\nu)$, also known as the \emph{Wasserstein-1 distance},  between two non-negative measures $\mu$ and $\nu$ of equal mass is given by:
\[ W_1(\mu,\nu) \ =\ \|\lambda\|_{\KR} \quad \text{where}\ \lambda=\nu-\mu . \]
This identity forms a significant bridge to optimal transport theory. However it can be easily shown that the $\KR$ norm is not complete when $X$ is not discrete. In the particular case when $X = \R^d$, the closure of zero-average measures with respect to the $\KR$ norm does not include only measures of finite mass but also a subclass of  distributions of order one. 
 A simple illustration of this situation is the case of \emph{infinitely many dipoles} (see also \cite{Ponce2004}):
 \begin{example}[Infinite dipoles; {\cite[p. 368]{Federer1969}}]For each $j \in \mathbb N$, consider 
\[
	m_j \coloneqq \sum_{i =1}^j  \delta_{2^{1-2i}} - \delta_{2^{-2i}} \in \Mbf_0(\R), \qquad \lambda \coloneqq \sum_{i =1}^\infty  \delta_{2^{1-2i}} - \delta_{2^{-2i}}.
\]
By construction, the sequence of ``molecules'' $(m_j)$ is Cauchy since $\| m_j - m_k\|_{\KR} \le 3^{-1}4^{-j}$ whenever $k > j$. Moreover, $m_j \to \lambda$ in $\Lip(\R)^*$, implying $\lambda \in \text{\AE}(\R)$. However, a straightforward application of the fundamental theorem of calculus reveals that
\[
	\lambda(f) \coloneqq \sum_{n = 1}^\infty f(2^{-2i}) - f(2^{1 - 2i})  \qquad \Rightarrow \quad  |\lambda|(\R) = +\infty.
\]
\end{example}

\subsection*{Completion and tangent bundles to a measure}
The $\KR$-completion of $\mathbf X_{0,1}(X)$ has been explored from various perspectives \cite{Hanin1997, Haninext, BCJ}. In the case of a pointed metric space $X$, this completion coincides with the Arens-Eells space. All molecules on $X$ are elements of $\mathbf X_{0,1}(X)$, and, if $m$ is a molecule, then $\rho(m)$ given by \eqref{eq:AEnorm} coincides with the Monge transport distance between $m^+$ and $m^-$ (the positive and negative parts of $m$). In other words, we have the equalities: $\rho(m) = W_1(m^+,m^-)= \|m\|_{\KR}$. It follows that the $\KR$-completion of $\mathbf X_{0,1}(X)$ coincides with $\text{\AE}(X)$, i.e.,
\[
	 \text{\AE}(X) = \mathrm{clos}_{\KR} \{ \mathbf X_{0,1}(X)\}.
\]

\begin{notation}
In light of this identification, we shall henceforth write
\[
	\rho(\lambda) = \|\lambda\|_{\KR} \qquad \text{for all $\lambda \in \text{\AE}(X)$}.
\]
\end{notation}

When $X=\R^d$, this completion was identified as the subspace $\Mbf_{0,1}(\R^d)$, comprising distributions (of order at most $1$) that are continuous with respect to the weak-$\ast$ convergence in $\Lip(\R^d)$ \cite[Lemma~2.6]{BCJ} (see also~\cite{BBD}). There, the authors of \cite{BCJ} showed that $\Mbf_{0,1}(\R^d)$ is the space of divergences $\operatorname{div} \sigma$ associated with \emph{tangential transport measures} $\sigma = \tau \mu$, where $\mu$ is a positive Borel measure and $\tau \in L^1(\mu)^d$ is a vector field that is \emph{tangent} to $\mu$. This analysis led to the introduction of a \emph{tangent bundle} associated with a finite Borel measure, generalizing the concepts of differential geometry. More precisely, given a positive Borel measure on $\R^d$, the tangent bundle $T_\mu$ (as defined  in \cite{BCJ}) is the unique $\mu$-measurable map $x \mapsto T_\mu(x)$ from $\R^d$ to $\operatorname{Gr}(d)$, where $\operatorname{Gr}(d)$ is the Grassmannian of subspaces of $\R^d$, characterized by the following tangential regularity property (see \cite[Theorem 2.4, Theorem 3.6]{BCJ}):
\[
	\tau(x) \in T_\mu(x) \; \text{$\mu$-a.e.} \quad \Longleftrightarrow \quad \operatorname{div} (\tau \mu) \in \Mbf_{0,1}(\R^d).
\]
Building upon a general variational approach from earlier works \cite{BBF2001, BBS-Calva1997, BouFra-Hessian}, the second author and collaborators introduced in~\cite{BCJ} the construction of a weak-$\ast$ continuous functional $\mu$-tangential gradient operator $\nabla_\mu : \Lip(\R^d) \to L^\infty(\mu)^d$. This operator enables integration by parts against all vector fields $\tau\in L^1(\mu)^d$ such that $\tau(x)\in T_\mu(x)$ $\mu$-a.e., which in turn allows for a differential calculus over $\mu$. 

In a forthcoming paper, the authors of the present work will establish that in fact $T_\mu$ coincides with the differentiability bundle $\mathscr D(\mu,\cdot)$ introduced in the seminal work by G. Alberti and Marchese \cite{AM}, where a remarkable $\mu$-quantitative counterpart of the Rademacher Theorem was obtained.

An interesting application of the results from the next section (see Corollary~\ref{cor:boundaries} below), will be that $\Mbf_{0,1}(\R^d)$ coincides with the boundaries of $1$-dimensional flat chains: 
\[
	\textnormal{\AE}(\R^d) = \Mbf_{0,1}(\R^d) = \set{\partial T}{T \in \mathscr F_1(\R^d)}.
\]


\section{Primitives: a simple optimal transport construction}\label{sec:OT}

Next, we introduce a transport construction to produce primitives of any given $0$-metric current, which substantially improves upon the known estimates of the cone-construction. The construction relies on the density of molecules in the Arens--Eells space, and we believe it might be independent interest:

\begin{lemma}\label{lem:discordia} Let $X$ be a separable and piecewise quasiconvex space (with quasiconvexity constant $c\ge 1$) and let $m  \in \textnormal{\AE}(X)$ be a molecule. Then, for any $\delta \in (0,1)$, there exists a rectifiable normal current $R \in \mathcal R_1(X) \cap \Nbf_1(X)$ satisfying
\[
	\partial R = m \quad \text{and} \quad \Mbf(R) \le c(1 + \delta) \|m\|_{\KR}.
\]
\end{lemma}

\begin{remark}\label{rem:rect}
If $X$ is a geodesic space, then $R$  can be constructed of the form
\[
	R = \sum_{h = 1}^L \eta_h \llbracket \theta \rrbracket, \qquad \eta_h \in \R, \; \theta \in \Lip_1([0,1],X)
\]
and satisfying the sharp estimate $\Mbf(R) =  \|m\|_{\KR}$. 
If moreover $X$ is a Banach space, then $R$ can also be chosen to be a polyhedral current.
\end{remark}

\begin{proof} 
According to its optimal representation, we can write $m$ as a finite sum of dipoles
\[
m = \sum_{h=1}^L \eta_h (\delta_{y_h} - \delta_{x_h})
\]
attaining its minimal transportation cost, so that
\begin{equation}\label{optimo-m}
\|m\|_{\KR} = \sum_{h = 1}^L |\eta_h|d(x_h,y_h) < +\infty, \qquad |\eta_h| > 0 \quad \forall h = 1,\dots,L.
\end{equation}
For the convenience of the reader, we begin by providing a very simple proof of Lemma~\ref{lem:discordia}  when the underlying space X is either a geodesic space or a Banach space, thus arriving to the equality stated in Remark \ref{rem:rect}. 
 Second, we'll look at the case where X is a piecewise quasiconvex space, where completeness is not  anymore assumed.
 \\

\emph{\textbf{a) Proof for geodesic spaces.}}\vskip0.7em

For each pair of points $(x_h,y_h)$, in the optimal representation of $m$, let $\theta_h :[0,1] \to X$ be a geodesic connecting $x$ to $y$. This means:
\begin{equation}
\theta_h(0) = x_h, \quad \theta_h(1) = y_h, \quad \ell(\theta_h) = d(x_h,y_h) \quad \text{for all } h =1,\dots,L.
\end{equation}
Now, we define a $1$-dimensional normal rectifiable current $R$ as the weighted sum of these geodesics, oriented from $x_h$ to $y_h$:
\[
	R \coloneqq \sum_{h = 1}^L \eta_h \llbracket \theta_h \rrbracket  \in \mathcal R_1(X) \cap \Nbf_1(X).
\]
Then we have $\partial R= \sum_{h =1}^L \eta_h (\delta_{y_h} - \delta_{x_h}) = m$ , whereas by \eqref{optimo-m}:
\begin{align*}
\Mbf(R) & \le  \sum_{h =1}^L |\eta_h| \ell(\theta_h) =  \sum_{h =1}^L |\eta_h| \dist(x_h,y_h) = \|m\|_{\KR}.
\end{align*}
On the other hand, we know that  $\Mbf(R) \ge \|\partial R\|_{\KR}$ whence the equality $\Mbf(R) = \|m\|_{\KR}$ as announced in Remark \ref{rem:rect}.~\vskip0.7em

\noindent\emph{\textbf{b) Proof for piecewise quasiconvex spaces.}} Now, we consider the more general case where X is a piecewise quasiconvex space. We write $\hat X$ to denote the completion of $X$. Being piecewise quasiconvex allows us to approximate a path between two points with a sequence of shorter paths, albeit possibly with a larger total length up to a constant factor.\vskip0.7em

\noindent\emph{1. A first approximation.} The first step consists in showing that, for any $\eps > 0$, we can find a normal rectifiable-$1$ current $S$ such that 
\begin{equation}\label{eq:intermediate}
	m - \partial S \in \text{\AE}(X), \qquad \| m - \partial S\|_{\KR} \le \eps, \qquad \Mbf(S) \le c\|m\|_{\KR},
\end{equation}
where $c$ is the piecewise quasiconvex constant of $X$ given in Definition~\ref{ass:con}.
Appealing to the definition of piecewise quasiconvex space, we can, for each $h \in \{1,\dots,L\}$, find  $(p^h_1,q^h_1), \dots, (p^h_{r(h)},q^h_{r(h)}) \in X^2$ such that
\begin{equation}\label{eq:diste}
d(x_h,p_1^h) + d(y_h,q_{r(h)}^h) + \sum_{i = 1}^{{r(h)-1}} d({p_{i+1}^h},q_i^h) < \frac{ \eps }{L|{\eta_h}|} ,
\end{equation}
together with Lipschitz curves $\theta_1^h,\dots,\theta^h_{r(h)} \in \Lip([0,1],X)$ satisfying
\begin{equation}\label{eq:distC}
	{\theta_i^h(0) = p_i^h, \quad \theta_i^h(1) = q_i^h}, \qquad \sum_{{i=1}}^{r(h)} \ell (\theta_i^h) \le c d(x_h,y_h).
\end{equation}
We define a rectifiable metric $1$-current on $X$ by  setting
\[
	S \coloneqq \sum_{h=1}^L {\eta_h}  \left\{   \sum_{{i = 1}}^{r(h)} \llbracket \theta_i^h \rrbracket \right\}
\]
We claim that $S$ is a $1$-current satisfying~\eqref{eq:intermediate}: Indeed, by using the identity $\partial \llbracket \theta \rrbracket = \delta_{\theta(1)} - \delta_{\theta(0)}$ for $\theta \in \Lip([0,1],X)$, we get
	\begin{align*}
		m - \partial S =   \sum_{h=1}^L \eta_h  \left\{ { (\delta_{p_1^h} - \delta_{x^h})} +
        \sum_{i = 1}^{{r(h)-1}}  (\delta_{{p^h_{i+1}}} - \delta_{q^h_i}) + { (\delta_{y^h} - \delta_{q^h_{r(h)}})} \right\} .
	\end{align*}
	Thus, $\partial S - m$ is a molecule. On the other hand, the fact that $\|\delta_x - \delta_y\|_{\KR} = d(x,y)$ for any two $x,y \in X$ and the triangle inequality and  gives
	\[
		\| \partial S - m\|_{\KR} \le \sum_{h=1}^L |\eta_h| \left\{{d(x_h,p_1^h) + d(y_h,q_{r(h)}^h)} +  \sum_{i = 1}^{{r(h)-1}}  d({p^h_{i}},{q^h_i}) \right\} \stackrel{\eqref{eq:diste}}\le \eps.
	\]
Recalling that $\Mbf(\llbracket \theta \rrbracket) \le \ell(\theta)$ for all $\theta \in \Lip([0,1],X)$, we can estimate the mass of $S$ by computing
	\[
		\Mbf(S) \le \sum_{h=1}^L |\eta_h|    \sum_{i = 1}^{{r(h)}} \ell( \theta_i^h) \stackrel{\eqref{eq:distC}}\le c \sum_{h=1}^L |\eta_h|   d(x_h,y_h) = c \|m\|_{\KR}.
	\]  	
	This proves~\eqref{eq:intermediate}.~\vskip0.7em

\noindent\emph{2. Construction of the Rectifiable Current $R$.} The strategy here is to iteratively apply the approximation result from the previous step to construct a sequence of rectifiable currents whose boundaries progressively approach $m$. The sum of this sequence will then converge to the desired rectifiable current $R$ with $\partial R = m$ and satisfying a controlled mass estimate $\Mbf(S) \le c' \|m\|_{\KR}$.

\begin{enumerate}[itemsep= 0.7em,  leftmargin=1.8em,label=$\bullet$,topsep=1em]
\item The First Step ($j=1$): Starting with $m_0 \coloneqq m$, we apply the approximation from Equation~\eqref{eq:intermediate} with a specific error bound. Let $\delta > 0$ be a parameter (to be determined later). We can find a rectifiable current $S_1 \in \mathcal R_1(X) \cap \Nbf_1(X)$ such that the remainder $m_1 \coloneqq m_0 - \partial S_1$ is a molecule in $\text{\AE}(X)$ satisfying:
\[
\|m_1\|_{\KR} \le 2^{-1} \delta \|m\|_{\KR}, \qquad \Mbf(S_1) \le c \|m\|_{\KR}.
\]

\item The Inductive Step ($j \ge 2$): Now, we proceed inductively. Assume that for some $j \ge 2$, we have constructed a sequence of rectifiable currents $S_1, S_2, \dots, S_{j-1}$ and a sequence of molecules $m_1, m_2, \dots, m_{j-1}$ such that $m_i = m_{i-1} - \partial S_i$ for $i = 1, \dots, j-1$, and they satisfy the bounds:
\[
\|m_{i}\|_{\KR} \le 2^{-i} \delta \|m\|_{\KR}, \qquad \Mbf(S_i) \le c \|m_{i-1}\|_{\KR} \quad \text{for } i = 1, \dots, j-1.
\]
Applying the approximation from~\eqref{eq:intermediate} to the molecule $m_{j-1}$ (with an error bound scaled by $2^{-j}$), we can find a rectifiable current $S_{j} \in \mathcal R_1(X) \cap \Nbf_1(X)$ such that the next remainder $m_{j} \coloneqq m_{j-1} - \partial S_{j}$ is a molecule in $\text{\AE}(X)$ satisfying:
\begin{equation}\label{eq:KR_bound_revised}
\|m_{j}\|_{\KR} \le 2^{-j} \delta \|m\|_{\KR}, \qquad \Mbf(S_j) \le c \|m_{j-1}\|_{\KR}.
\end{equation}
Using the inductive hypothesis, we can bound the mass of $S_j$:
\begin{equation}\label{eq:mass_bound_revised}
\Mbf(S_j) \le {c} \|m_{j-1}\|_{\KR} \le c 2^{-(j-1)} \delta \|m\|_{\KR} = c 2^{-j+1} \delta \|m\|_{\KR}.
\end{equation}
\end{enumerate}

\noindent\emph{3. Construction of a Rectifiable Current.} 
We now define a sequence of partial sums of the rectifiable currents:
\[
P_j \coloneqq S_1 + S_2 + \cdots + S_j \in \mathcal R_1(X) \cap \Nbf_1(X)
\]
We can show that the sequence $(P_j)$ is Cauchy in the mass norm $\Mbf_1(X)$. For $k > j \ge 1$:
\[
\Mbf(P_k - P_j) \le \sum_{h = j+1}^k \Mbf(S_h) \stackrel{\eqref{eq:mass_bound_revised}}{\le} \sum_{h = j+1}^k c 2^{-h+1} \delta \|m\|_{\KR} \le c \delta \|m\|_{\KR} 2^{-j+1}.
\]
Since $\Mbf_1(\hat X)$ is a Banach space, the Cauchy sequence $(P_j)$ converges in mass to a limit current $R \in \Mbf_1(\hat X)$:
\[
\Mbf(R - P_j) \to 0 \quad \text{as } j \to \infty, \quad \text{where } R \coloneqq \sum_{h = 1}^\infty S_h \in \Mbf_1(X).
\]
From the mass bound for each $S_h$ (recalling that $\Mbf(S_1) \le c \|m\|_{\KR}$), we can estimate the mass of $R$:
\[
\Mbf(R) \le \sum_{h = 1}^\infty \Mbf(S_h) \stackrel{\eqref{eq:mass_bound_revised}}\le c \|m\|_{\KR} + {c}\delta \|m\|_{\KR} \sum_{h = 2}^\infty  2^{-h+1} \le c(1 +\delta) \|m\|_{\KR}.
\]
This proves the mass bound.\vskip0.7em

\noindent\emph{5. Proof that $R$ is a rectifiable current in $X$.} Observe that $\|R\|$ is concentrated on $\Gamma = \bigcup_j B_j \subset X$, where $B_j$ is the (finite) union of curves over which each $\|S_j\|$ is concentrated. This follows from the first inequality in the estimate above. In particular, $\Gamma$ is a rectifiable set contained in $X$. Now, let $U \in \mathcal B(\hat X)$ be a Borel set with vanishing $\mathcal H^1$ measure. Then, since each $P_j$ is $1$-rectifiable, it follows that
\[
	\|R\|(U) \le \|R - P_j\|(U) + \|P_j\|(U \cap X) \le \Mbf(R - P_j) \qquad \forall j \ge1.
\] 
Letting $j \to \infty$, we discover that $\|R\|(U) = 0$. This proves that $R \in \mathcal R_1(\hat X)$. Since $R$ is concentrated on $\Gamma \subset X$, it follows that $R \in \mathcal R_1(X)$.\vskip0.7em

\noindent \emph{6. Proof of the boundary identity.} It remains to verify that $\partial R = m$.
Notice that by construction,  $m - \partial P_j = m - (\partial S_1 + \cdots + \partial S_j) = m_j$. Hence, for $j \ge 2$ we get
\begin{align*}
	\|\partial R - m\| & \le \|m - \partial P_j\|_{\KR} + \|\partial P_j - \partial R\|_{\KR} \\
	& \le \|m_j\|_{\KR} + \Mbf(P_j - R)   
\end{align*} 
We know from Equation~\eqref{eq:KR_bound_revised} that $\|m_j\|_{\KR} \le 2^{-j} \delta \|m\|_{\KR} \to 0$ as $j \to \infty$. Therefore, letting $j \to \infty$, we conclude that $\partial R = m$. Since $|m|$ is a finite Radon measure, it follows that $R$ is a normal current, i.e., $R \in \Nbf_1(X)$.
\end{proof}

The following variant of the previous lemma will useful later on:

\begin{lemma}\label{lem:discordia2} Let $X$ be a separable Banach space and let $m  \in \textnormal{\AE}(X)$ be a molecule. Then, for any ${\eps} \in (0,1)$, there exists a rectifiable normal current $R \in \mathcal R_1(X \oplus \R) \cap \Nbf_1(X \oplus \R)$ satisfying
\[
	\partial R = m \oplus 0, \qquad \|R\| \mres (X \oplus \{0\}) \equiv 0 \quad \text{and} \quad \Mbf(R) \le \|m\|_{\KR} + {\eps}.
\]
\end{lemma}
\begin{proof}
	Let 
  \[
        m = \sum_{h = 1}^{r} \eta_h(\delta_{y_h} - \delta_{x_h}), \qquad x_h,y_h \in X,
    \]
    be an optimal decomposition of $m$ with respect to the Kantorovich--Rubinstein norm. 
    
    Letting ${M} \coloneqq \max\{|\eta_h|\}_{h =1}^r$ and ${t \coloneqq (2Mr)^{-1}\eps}$, we define $R \in \Nbf_1(X \oplus \R) \cap \mathcal R_1(X \oplus \R)$ by
    \[
        R \coloneqq \sum_{h=1}^{r} \eta_h \lpr{x_h \oplus 0, \frac{x_h + y_h}2\oplus  t} + \eta_h\lpr{\frac{x_h + y_h}2\oplus t, y_h\oplus 0}.
    \]
 By construction $\partial R = m \oplus 0$. Moreover, 
 \[
 	\|R\| = {\mathcal H^1 \mres \bigcup_{h=1}^{r} \left\{  (x_h \oplus 0, \frac{x_h + y_h}2\oplus  t) \cup  (\frac{x_h + y_h}2\oplus t, y_h\oplus 0) \right\}}.
 \]
 Since $X \oplus \{0\}$ intersects the support of $\|R\|$ on finitely many points (the ones associated to the dipoles of $m$), it follows that $\|R\| \mres X \oplus \{0\}$ is the zero measure. Finally, the triangular inequality and the isometry of the embedding $X \embed X \oplus \{0\}$ give
 \[
 	\Mbf(R) \le \sum_{h=1}^{r} |\eta_h| \|x_h - y_h\|_X +2 \sum_{h=1}^{r} |\eta_h| (2{M}r)^{-1}{\eps} \le  \|m\|_{\KR} + \eps.
 \]
 This finishes the proof.
\end{proof}

\section{Proof of the approximation and filling results}\label{sec:qc}
\subsection{Proof of Theorem~\ref{thm:approx_qc}}

Our goal is to show that any metric current $T \in \Mbf_1(X)$, can be approximated in mass by a sequence of normal currents $(T_j) \subset \Nbf_1(X)$, which, in fact, can be chosen to be rectifiable up to a cycle. 

\emph{Step 1. Approximation by Molecules.} By the axioms of currents, the boundary $\partial T$ acts as a linear functional on $\Lip(X)$ satisfying the assumption of Lemma~\ref{lem:porfin}, which guarantees that $\partial T \in \text{\AE}(X)$. Notice that this is a subtle but powerful advantage of the Arens--Eells approach. Indeed, $\partial T$ could "go out" of $X$, as we are not assuming that $X$ is complete. However, $\partial T$ is still represented as an element of $\text{\AE}(X)$. Another way of seeing this, is to work directly in the completion $\hat X$ of $X$ and recall from Corollary~\ref{cor:completion} that $\text{\AE}(X)$ is isometrically isomorphic to $\text{\AE}(\hat X)$. 

Since the space of molecules on $X$ is dense in $\text{\AE}(X)$ under the Kantorovich--Rubinstein norm ($\|\cdot\|_{\KR}$), this means we can find a sequence of molecules $(m_j)_j \subset \text{\AE}(X)$ satisfying
\begin{equation}\label{eq:firsteq-improved}
\|\partial T - m_j\|_{\KR} \to 0 \quad \text{as } j \to \infty.
\end{equation}

\emph{Step 2. Telescopic Estimates.} For a fixed small positive number $\eps \in (0,1)$ (whose specific value is not crucial for this part of the proof but will be used later for other computations), we can choose, without loss of generality, the sequence $(m_j)$ to satisfy the following conditions:
\begin{align}\label{eq:m1}
\|m_1\|_{\KR} \le \|\partial T\|_{\KR} + 2^{{-2}}\eps, \qquad \|m_j - m_{j-1}\|_{\KR} \le 2^{{-(j+1)}}\eps \quad \text{for each } j \ge 2.
\end{align}
These conditions ensure that the ``jumps'' between consecutive molecules in the sequence become progressively smaller.\vskip0.7em

\emph{Step 3. Constructing the approximating sequence.} For $\delta \in (0,1)$ and for each $j \ge 1$, considering the difference between consecutive molecules (with $m_0$ defined as the zero measure),  Lemma~\ref{lem:discordia} produces a rectifiable current $S_j \in \mathcal R_1(X) \cap \Nbf_1(X)$ such that:
\begin{equation}\label{eq:step2}
    \partial S_j = m_j - m_{j-1}, \qquad \Mbf(S_j) \le \begin{cases}
        c(1 + \delta) \|m_1\| \le c(1 + \delta)\|\partial T\|_{\KR} + c2^{-1}\eps & \text{for $j = 1$}\\
        c (1 + \delta) \|m_{j} - m_{j-1}\|_{\KR} \le  c2^{-j+1} \eps & \text{if $j \ge 2$}
    \end{cases}.
\end{equation}
We now define a sequence of rectifiable currents $(P_j)_{j \ge 1}$ by taking the cumulative sum  
\[
P_j \coloneqq S_1 + S_2 + \cdots + S_j  \in \mathcal R_1(X) \cap \Nbf_1(X).
\]
Notice that  the sum $\partial S_1 + \cdots + \partial S_j$ telescopes to $m_j$, giving $\partial P_j = m_j$. 
 We now observe that $(P_j)$ is a Cauchy sequence in $\Mbf_1(X)$. For any $k > j \ge 2$, the estimate in~\eqref{eq:step2} gives
 \begin{align*}\label{eq:estimatePj-improved}
\Mbf(P_k - P_j) & = \Mbf(S_{j+1} + \cdots + S_k) \le {c2^{-j} \eps}.
\end{align*}
Following the reasoning from Steps 4 and 5 in the proof of Lemma~\ref{lem:discordia}, we deduce that there exists a limit current $R \in \mathcal R_1(X)$ satisfying
\begin{equation}\label{eq:Ps-improved}
\Mbf(R - P_j) \to 0 \quad \text{as } j \to \infty.
\end{equation}
Notice that, even though $R$ is a priori only a limit of a Cauchy sequence in $\Mbf_1(X)$, it is itself concentrated in $X$ (and not in the completion $\hat X$).
 Moreover, notice that each $S_j$ is supported on some $\Gamma_j \subset X$, which is a finite union of images of Lipschitz curves.

For later use, we record that the mass convergence, the first bound in~\eqref{eq:m1} and the bounds in~\eqref{eq:step2} give
\begin{equation}\label{eq:eps1}
	\Mbf(R) \le \sup_j \Mbf(P_j) \le {c(1 + \delta)}\|\partial T\|_{\KR} +  {c\eps}.
\end{equation}

Now, we relate the boundary of $R$ to the boundary of $T$. Using the triangle inequality for the Kantorovich--Rubinstein norm and the estimate $\|\partial \cdot\|_{\KR} \le \Mbf(\cdot)$, we have:
\begin{align*}
\|\partial R - \partial T\|_{\KR} &\le \|\partial R - \partial P_j\|_{\KR} + \|\partial P_j - \partial T\|_{\KR} \\
&\le \Mbf(R - P_j) + \|m_j - \partial T\|_{\KR}.
\end{align*}
Therefore, taking the limit as $j \to \infty$, we conclude, from~\eqref{eq:firsteq-improved} and \eqref{eq:eps1}, that $\partial R = \partial T$ as $0$-functionals on $X$. Hence, $C \coloneqq T - R$ is a cycle in $\Mbf_1(X)$.  Finally, we define a sequence of currents $T_j$ by adding the cycle $C$ to the sequence of rectifiable currents $P_j$:
\[
T_j \coloneqq C + P_j \in \Nbf_1(X), \quad j \in \mathbb N.
\]
Since the difference $T - T_j$ is precisely $R - P_j$, from~\eqref{eq:Ps-improved} we know that 
\[
\Mbf(T - T_j) = \Mbf(R - P_j) \to 0 \quad \text{as } j \to \infty.
\]
This shows that the original current $T$ can be approximated in the mass norm by a sequence of currents $T_j$, where each $T_j$ is the sum of a cycle $C$ and a rectifiable current $P_j$. 

The fact that the elements of sequence $(P_j)$ constructed above can be chosen to be
\begin{itemize}
\item[a)] a sum of of geodesic integral curves (in geodesic spaces),
\item[b)] polyhedral $1$-currents (on Banach spaces) 
\end{itemize} 
follows directly from Remark~\ref{rem:rect}. This completes the proof.\qed\\

\subsection{Proof of Corollary~\ref{cor:filling}} The proof follows directly from the decomposition $C = T -  R$ and the estimate~\eqref{eq:eps1} with $\eps' = \frac \eps 2 (1+ c)^{-1}$ and $\delta = \frac \eps 2(1 + \|\partial T\|_{\KR})^{-1}$ in the proof of Theorem~\ref{thm:approx_qc}.

\subsection{Proof of Theorem~\ref{cor:cycle_cover}}
	If $X$ is a Banach space of dimension $\dim(X) \ge 2$, we can modify the construction to ensure that the original current $T$ and the rectifiable current $R$ from Step~2 are concentrated in disjoint sets. The argument follows by slightly modifying certain steps in the proof of Theorem~\ref{thm:approx_qc}, namely:
	
	 In Step~1, we can modify the approximating sequence $(m_j)$ in the following way: for fixed $j$, let $\sum_{h=1}^{L(j)} \eta_h(\delta_{y_h} - \delta_{x_h})$ be the optimal decomposition of $m_j$. Now, let $h \in \{1,\dots,L(j)\}$. The fact that $X$ is a normed space with dimension at least 2, implies that we can find $\nu_j \in \partial B_X$ that does not belong to $\spn\{y_h - x_h\}$ for all $h \{1,\dots,L(j)\}$. In particular, since $\|T\|$ is $\sigma$-finite, we can also find 
	 \begin{equation}\label{eq:tj}
	 0 < t_j  < L(j)^{-1}2^{-j+1}\eps' \max\{|\eta_1|,\dots,|\eta_{L(j)}|\}
	 \end{equation}
	  such that 
	 \[
	 \|T\|(\bigcup_{h = 1}^{L(j)} [x_h+t_j\nu_j,y_h+t_j\nu_j]) = 0.
	 \] We then define 
	 \[
	 m_j' \coloneqq \sum_{h=1}^{L(j)} \eta_h(\delta_{y_h'} - \delta_{x_h'}), \qquad x'_h \coloneqq x_h+t_j\nu_j, \quad y'_h \coloneqq y_h+t_j\nu_j.
	 \]
	 Notice that $\|m_j - m_j'\|_{\KR}\le 2^{-j} \eps'$ and therefore $(m_j')$ still satisfies~\eqref{eq:firsteq-improved}.

	In Step~2, we consider the estimates with $\eps' = \eps/2$ instead of $\eps$. 
	
	In Step~3, the associated primitive $S_j'$ to $m_j'$ (given by Lemma~\ref{lem:discordia}  and Remark~\ref{rem:rect}) is then of the form
	\[
		\sum_{h=1}^{L} \eta_h \llbracket x_h',y_h' \rrbracket, \quad \text{with} \quad \|T\|(\bigcup_{h = 1}^L [x_h',y_h']) = 0.
	\]
Following the proof by verbatim, replacing $(m_j)$ with $(m_j')$ and $(S_j)$ with $(S_j')$, we obtain that both $(P_j)$ and $R$ satisfy the same properties as in Corollary~\ref{cor:filling} for $\eps'$. Moreover, they are concentrated on a countable union $\Gamma \subset X$ of intervals where $\|T\|$ vanishes. In particular, $C = T - R$ is a cycle, and, letting $B = X \setminus \Gamma\in \mathcal B(X)$ we get 
\[
	T = C + R \quad \Rightarrow \quad T = T \mres B = C \mres B + R \mres (X \setminus \Gamma) = C \mres B.
\]
Recalling that $(P_j)$ and $R$ satisfy~\eqref{eq:eps1}, we obtain (recall that $C = 1$ on a geodesic space)
\[
	\Mbf(C) = \Mbf(T) + \Mbf(R) \le \Mbf(T) + \|\partial T\|_{\KR} + 2\eps' \le \Mbf(T) + \|\partial T\|_{\KR} + \eps.
\] 
This finishes the proof.\qed\\

Similar arguments building on Lemma~\ref{lem:discordia2} (instead of Lemma~\ref{lem:discordia}) yield the following version of Theorem~\ref{cor:cycle_cover}, which will be instrumental in the proof of Theorem~\ref{thm:rep}.

\begin{corollary}\label{cor:cycle_cover2}
Let $X$ be a separable Banach space and let $T \in \mathbf{M}_1(X)$. Then, for any given $\eps > 0$, there exists a cycle $C \in \Mbf_1(X \oplus \R)$ and a rectifiable current $R \in \mathcal R_1(X \oplus \R)$ such that
\begin{enumerate}[itemsep=0.2em, leftmargin=2.5em,label=\arabic*.]
	\item $C = iT + R$, where $i : X \embed X_0 \coloneqq X \oplus \{0\}$ is the canonical isometric embedding
	\item $\|iT\|$ is concentrated on $X_0$ and $\|R\| \mres X_0 \equiv 0$
	\item $\|T\|$ and $\|R\|$ are mutually singular; in particular, $T = C \mres X_0$ and $R = C\mres \{ X \oplus (0,\infty) \}$
	\item $\Mbf(R) \le  \|\partial T\|_{\KR} + \eps$ \label{eq:optimalCC}
\end{enumerate}
\end{corollary}

\begin{proof}
	The proof follows by verbatim the steps of Theorem~\ref{thm:approx_qc}, using the primitives provided by Lemma~\ref{lem:discordia2} (instead of the ones provided by Lemma~\ref{lem:discordia}).
\end{proof}

\subsection{Direct consequence of the proofs} Here we collect two side results results, which are connected to the discussion in Section~\ref{sec:completion}. A direct consequence of the proof of Theorem~\ref{thm:approx_qc} is the following construction of primitive currents for elements in $\text{\AE}(X)$:

\begin{corollary}[One-dimensional metric primitives]\label{cor:primitive}
	Let $X$ be a separable and piecewise quasiconvex space (with piecewise quasiconvexity constant $c$) and let $\lambda \in \textnormal{\AE}(X)$. For any $s > 0$, there exists a rectifiable metric current $R \in \Mbf_1(X) \cap \mathcal R_1(X)$ satisfying
	\[
	\partial R = \lambda \quad \text{and} \quad \Mbf(R) \le c \|\lambda\|_{\KR} + s.
	\]
\end{corollary}
\begin{proof}
The proof follows by verbatim the steps of the proof of Theorem~\ref{thm:approx_qc} under the following considerations: in Step~1, replace $\partial T$ by $\lambda$; in Step~2, let  $\eps > 0$ be small enough so that $\eps (1 + c) < s$ in~\eqref{eq:eps1}. \end{proof}

In particular, the Arens--Ells space coincides with the space of boundaries of metric currents:

\begin{corollary}\label{cor:boundaries}
	Let $X$ be a piecewise quasiconvex space. Then
	\begin{align*}
		\textnormal{\AE}(X) & = \set{\partial T}{T \in \Mbf_1(X)} \\
		& = \set{\partial T}{T \in \Mbf_1(X) \cap \mathcal R_1(X)}. 
	\end{align*}
	If moreover $X$ is a geodesic space and $\lambda \in \textnormal{\AE}(X)$, then
	\[
		\|\lambda\|_{\KR} = \inf\set{\Mbf(R)}{T \in \Mbf_1(X), \partial T = \lambda}. 
	\]
	In particular, for $X = \R^d$ with the standard Euclidean metric, we recover the results from~\cite{BCJ}:
	\begin{align*}
		\textnormal{\AE}(\R^d) & = \mathbf M_{0,1}(\R^d) = \set{\partial T}{T \in \mathscr F_1(\R^d)} 
	\end{align*}
	and also that
	\begin{align*}
		\|\lambda\|_{\KR} & = \inf \set{\Mbf(T)}{T \in \mathscr F_1(\R^d), \partial T = \lambda}.
	\end{align*}
\end{corollary}
\begin{proof}
	The first two assertions follow directly from Corollary~\ref{lem:porfin} and Corollary~\ref{cor:primitive}, taking into account that geodesic spaces are quasiconvex with constant $c=1$. 	
\end{proof}


   \section{Currents associated with $BV$-curves}\label{sec:BV}
This section introduces a specific type of one-dimensional metric current associated with certain curves of bounded variation ($BV$). These currents decompose naturally into absolutely continuous and cantorian parts, reflecting the structure of the $BV$ function's derivative. 

\subsection{The class of curves with bounded variation on a metric space} Let's first recall the definitions of curves that we shall consider. 

\begin{notation}
For the reminder of this section, we write $I$ to denote the open interval $(0,1)$ endowed with the Euclidean metric. We also assume that $X$ is \textbf{complete and separable}. 
\end{notation}

\subsection*{Integrable curves and the essential total variation} If $Y$ is a metric space, we define $L^1(I,Y)$ as the class of Borel curves $u : I \to Y$ such that $d(u,y) \in L^1(I)$ for some point $y \in Y$. The \emph{essential total variation} of $u \in L^1(Y)$ is the defined as the (possibly infinite-valued) infimum
\[
V(u) \coloneqq \inf_{v = u \, \text{a.e.}} \left\{ \sup_{0 < t_1 < \cdots < t_m < 1} \sum_{i = 1}^{m-1} d(v(t_{i+1}), v(t_i)) \right\} \in [0,\infty].
\]

We introduce the space  
\[
	\text{\gls{BV_curves}} \coloneqq \set{u \in L^1(I,X_c)}{X_c \subset \subset X, f(u) \in BV(I) \; \forall f \in \Lip_1(X)}
\]
of $BV$-curves with values on $X$ and compactly supported images.  We purposefully consider compact curves in order to exploit the theory developed in~\cite{Ambrosio1990BV}, where it is shown that $u \in \BV(I,X)$ if and only if the class of finite measures $\sigma$ on $I$ satisfying
\begin{equation}\label{eq:TV}
	\sigma(K) \ge |D f(u)|(B) \qquad \text{for all $K \subset \subset (0,1)$ and all $f \in \Lip_1(X)$}
\end{equation}
is non-empty. In this case, there exists a minimal measure \gls{total_variation} with this property, called the \emph{total variation measure} of $u$.\medskip 

Next, we recall some of the highlights of this theory for a curve $u \in \BV(I,X)$. The fact that $u \in \BV(I,X)$ is equivalent with the finiteness of the essential total variation $V(u)$; in this case it holds $V(u) = |Du|(I)$. Every $u \in \BV(I,X)$ possesses left and right continuous representatives. In other words, there exist left and right continuous curves $u^-,u^+ : I \to X$ satisfying
\[
	\text{\gls{u-}}(t) = \lim_{s \to t^-} u(s) \quad \text{and} \quad \text{\gls{u+}}(t) = \lim_{s \to t^+} u(s) \qquad \text{for every $t \in (0,1)$.}
\]
Moreover, $u^-$ and $u^+$ coincide outside a countable $S_u \subset I$, where $u$ is discontinuous. In particular, every $u \in \BV(I,X)$ curve has an $L^1$-representative in the space of Càdlàg curves. The \emph{approximate metric differential} 
\[
	\text{\gls{approx_metric}}(x) \coloneqq \lim_{|h| \to 0} \frac{d(u^-(t + h),u^-(t))}{|h|} = \lim_{|h| \to 0} \frac{d(u^+(t + h),u^+(t))}{|h|}
\]
exists $\mathcal L^1$-almost everywhere, defining a function in $L^1(I)$ satisfying $|\dot u| = |Du|/\mathcal L^1$. Lastly, the total variation measure can be decomposed into mutually singular measures as:
\begin{equation}
\begin{split}
    |Du|(B) & \coloneqq |D^a u|(B) + \, |D^j u|(B) \, + |D^c u|(B), \qquad \forall B \in \mathcal B(I),
\end{split}
\end{equation}
where
\[
\text{\gls{total_ac}}(B)\coloneqq \int_{B} |\dot u| \, dt \qquad \text{and} \qquad \text{\gls{total_j}}(B) \coloneqq   \int_{B \cap S_u} d(u^-,u^+) 
\]
 are the \emph{absoutely continuous} and \emph{jump} parts, and \gls{total_c} is the \emph{cantorian part}, a singular measure vanishing on countable sets, often representing the "fractal behavior" of $BV$ functions.\medskip
\subsection*{Compact curves with special bounded variation} A curve $u \in \text{\gls{SBV_curves}}$ is a compact $BV$-curve for which $|C\gamma|$ vanishes. It can be shown (from McShane's extension theorem~\cite[Thm 1.33]{Weaver_Book} and~\cite[Rmk.~2.4]{Ambrosio1990BV}) that $\SBV(I,X)$ also satisfies a weak formulation by composition with Lipschitz maps:
\begin{equation}\label{eq:SBV_weak}
	u \in \SBV(I,X) \quad \Leftrightarrow \quad f(u) \in SBV(I) \quad \forall f \in \Lip_1(X)
\end{equation}
where $SBV(I)$ is the classical space of functions with bounded variation on $I$ (see~\cite[p. 212]{ambrosioFunctionsBoundedVariation2000}).

\subsection{Currents associated with injective $BV$-curves} Now that we have introduced the concept of compact $BV$ curve, we will associate to any curve $u  \in \BV(I,X)$,   two particular integral currents:

\begin{definition}[Currents associated with $BV$ curves]
    Let $u \in \BV(I,X)$. We define the \emph{absolutely continuous metric current} $\lpr{u}^a$ and the \emph{cantorian metric current} $\lpr{u}^c$ associated with $\gamma$ as bi-linear functionals acting on pairs $(f, \pi) \in \mathscr{D}^1(X)$ by
    \[
         \text{\gls{abs_current}}(f,\pi) \coloneqq \int_0^1 f(u) \, d D^a(\pi(u))
    \]
    and
    \[
         \text{\gls{cantor_current}}(f,\pi) \coloneqq \int_0^1 f(u) \, d D^c(\pi(u)).
    \]
Here, $D^a$ and $D^c$ denote the absolutely continuous and cantorian parts of the classical $BV$ derivative. When $X$ is a Banach space, we also define the \emph{jump metric current} 
    \[
         \text{\gls{jump_current}} \coloneqq \bigcup_{t \in S_u} \lpr{u^-(t),u^+(t)}
    \]
    and (with a possible abuse of notation) 
    \[
     \lpr{u}  \coloneqq  \lpr{u}^a+  \lpr{u}^c+  \lpr{u}^j.
    \]
 \end{definition}   
 
 The following identity for post-composition with Lipschitz maps will be useful throughout the rest of the manuscript (for the definition of push-forward, see~\eqref{eq:push_0}): 
 \begin{remark}\label{rem:composition}
 	If $\Phi : X \to X'$ is Lipschitz and $u \in \SBV(I,X)$, then $\Phi(u) \in \SBV(I,X')$ and 
	\[
		\lpr{\Phi(u)}^a = \Phi_\#\lpr{u}^a.
	\]
 \end{remark}
 \begin{proof}
 	The fact that $\Phi(u) \in \SBV(I,X')$ follows from the fact that continuous images of compact sets are compact.  Let $(f,\pi) \in \mathscr D^1(X')$ be given. The associativity of the composition gives
	\begin{align*}
	\lpr{\Phi(u)}^a(f,\pi) & = \int_0^1 f(\Phi(u)) \, dD^a(\pi(\Phi(u))) \\
	& = \int_0^1 (f \circ \Phi)(u) \, d(D^a (\pi \circ \Phi)(u) ) = \lpr{u}^a(f \circ \Phi,\pi\circ \Phi) = \Phi_\#\lpr{u}^a(f,\pi).
	\end{align*}
	This proves the desired identity.
 \end{proof}

\begin{proposition}\label{prop:SBV} If $u \in \BV(I,X)$, then $\lpr{u}^a,\lpr{u}^c$ belong to $\Mbf_1(X)$. If moreover $X$ is a Banach space, then
$\lpr{u} - \lpr{u^-(0),u^+(1)}$ is a cycle.
\end{proposition}
\begin{proof}
First, we recall by the definition of $\BV(I,X)$ through composition with Lipschitz maps that $|D(\pi(u))| \le \Lip(\pi) |Du|$ for any $\pi \in \Lip(X)$. This, in turn, implies that
\begin{equation}\label{eq:Lip_chain_mass}
      |D^i (\pi(u))| \le \Lip(\pi) |D^i u| \quad \text{for all $i \in \{a,c,j\}$.}  
\end{equation}
From this, it is straightforward to compute that
\[
    \| \lpr{u}^a\| \le u_\# |D^a u| \quad \text{ana} \quad \|\lpr{u}^c\| \le u_\# |D^c u| 
\]
as measures. This implies that both $\lpr{u}^a$ and $\lpr{u}^c$ have finite mass provided that $u \in \BV(I)$. 

Now, suppose that $(f,\pi) \in \mathscr D^1(C)$ are such that $\pi \equiv c_0$ on a neighborhood $V$ of $\spt f$. Then, it follows that $\pi(u) \equiv C_0$ on a neighborhood $J$ of $\spt f(u)$. This follows from the fact that $u$ has left and right continuous representatives. In particular $D(\pi(u)) \equiv 0$ on $J$ and hence $\lpr{u}^a(f,\pi) = \lpr{u}^c(f,\pi) = 0$. This demonstrates $\lpr{u}^a$ and $\lpr{u}^c$ satisfy the locality property.

Now, we demonstrate that both currents satisfy the continuity axiom. For any $v \in BV(I)$, let $i \in \{a,c\}$ and define $v^i(t) \coloneqq D^i v(0,t]$ for $t \in (0,1)$. Observe that $v^i$ is continuous, belongs to $BV(I)$ and satisfies the identity $D^i v \equiv D v^i$ as measures. In particular, by the locality property,
\begin{equation}\label{eq:muero}
	\lpr{u}^i(f,\pi) =  \int_0^1 f(u) \, d D^i(\pi(u)) = \int_0^1 f(u) \, dD(\pi(u)^i), \quad i \in \{a,c\}.
\end{equation}
By the finite properties of $BV$-functions of one-variable, it follows that $\pi(u)^i$ is continuous and $\pi(u) \in BV(I)$. Hence, integration by parts holds and we get (recall that $v^i(0^+) = 0$)
\begin{equation}\label{eq:muero2}
	\lpr{u}^i(f,\pi) = - \int_0^1 \pi(u)^i \, dD(f(u)) + f(u^-(1)) D\pi(u)^i((0,1)), \quad i \in \{a,c\}.
\end{equation}
Now, let $\pi_j \toweakstar \pi$ in $\Lip(X)$. The definition of $\BV$ by composition with Lipschitz maps warrants that $|D^i (\pi_j \circ u)| \le \Lip(\pi_j) |D^iu|$ for all $i \in \{a,j,c\}$. 
Since moreover $\pi_j(u) \to \pi(u)$ pointwise, 
we gather that $D(\pi_j(u)^i) = D^i(\pi_j(u))$ converges on $I$ in the weak sense of measures  to $D(\pi(u)^i) = D^i(\pi(u))$ for each $i \in \{a,c,j\}$. In fact, the uniform mass bound in~\eqref{eq:Lip_chain_mass} warrants that no mass of $|D\pi_j(u)^i|$ can concentrate towards $\{1\}$ and hence
\[
    D\pi_j(u)^i((0,1)) \to  D\pi(u)^i((0,1)), \quad i \in \{a,c\}.
\]

By the standard $BV$-compactness, we further deduce that $\pi_j(u)^i$ converges to $\pi(u)^i$ in $L^1$, and hence also pointwise up to taking a subsequence (not relabelled). On the other hand~\eqref{eq:Lip_chain_mass} tells us that $|\pi_j(u)| \le \Lip(\pi_j)|Du|(I)$ for each $j$.  Since this is uniformly bounded on $j$, we can apply the dominated convergence theorem to the sequence $(\pi_j(u)^i)$ and pass to the limit in~\eqref{eq:muero2}:   
\begin{align*}
 \lpr{u}^i(f,\pi_j) & = - \int_0^1 \pi_j(u)^i \, dD(f(u)) + f(u^-(1)) D\pi_j(u)^i((0,1)) \\
 & \to \quad - \int_0^1 \pi(u)^i \, dD(f(u)) + f(u^-(1)) D\pi(u)^i((0,1)) =  \lpr{u}^i(f,\pi)
\end{align*}
for $i \in \{a,c\}$. Since this holds for all convergent subsequences, this proves that $\lpr{u}^a$ and $\lpr{u}^c$ satisfy the continuity property, as desired.

Let $\tilde D \coloneqq D^a + D^c$ and let $f \in \Lip(X)$.  Since $S_{f(u)} \subset S_u$ and $\partial\lpr{u}^j = \sum_{t \in S_u} \delta_{u^+(t)} - \delta_{u^-(t)}$, it follows from the fundamental theorem of calculus for $BV(I)$ functions that 
  \begin{align*}
        \partial \lpr{u}(f)
        & = \int_0^1  d\tilde D(f(u)) + \lpr{u}^j(1,f) \\ 
        &  = \int_0^1 d\tilde D(f(u)) + \sum_{t \in S_{f(u)}} f(u)^+(t) - f(u)^-(t) \\ 
        & = \int_0^1  \, d D(f(u)) = f(u^-(0)) - f(u^+(1)).
    \end{align*}
Since $f$ was chosen arbitrarily, this demonstrates that $\partial (\lpr{u} - \lpr{u^-(0),u^+(1)}) = 0$.
\end{proof}

We will also need the following semicontinuity result:

\begin{lemma}\label{lem:semicont}
	Suppose that $u_h \to u$ in $\SBV(I,X) \subset \mathbb D([0,1],X)$. Then
    \[
       |D^ju|(I) \le \liminf_{h} |D^j u_h|(I) \quad \text{and} \quad |D^ju|(I) \le \liminf_{h} |D^j u_h|(I).
    \]
\end{lemma}
\begin{proof}
    We may assume without loss of generality that both limit inferiors are limits. Since $d(u_h,u) \to 0$ in $L^1((0,1))$, we may pass to a subsequence $(u_{h_k})$ such that $u_{h_k} \to u$ pointwise almost everywhere. Then, by~\cite[Thm.~2.4]{Ambrosio1990BV}, we get 
    \[
        |Du|(I) \le \liminf_{k} |Du_{h_k}|(I) = \liminf_{h}|Du_{h}|(I).
    \]
    Let $\delta > 0$ and let $F_\delta\subset S_u$ consist of those points $t$ with $d(u^+(t),u^-(t)) > \delta$. Clearly, $F_\delta \uparrow (0,1)$ and therefore $|D^j u|(F_\delta) \to |D^j u|(I)$. On the other hand, the Càdlàg convergence warrants that $\delta_h = d(u,u_h(\lambda_h)) + \|\lambda_h - \mathrm{id}\|_{\infty} \to 0$ with $\lambda_h \in \Lambda$. This translates into the following key estimate (used throughout the text):
    \begin{equation}\label{eq:estimate}
        |d(u^+,u^-) - d(u_h^+(\lambda_h),u^-h(\lambda_h))| \le 2 \delta_h
    \end{equation}
    In particular, if $t \in F_\delta$, then $\lambda_h(S_u) \subset S_{u_h}$ provided that $h = h(\delta)$ is sufficiently large. From the estimate above and the finiteness of $F_\delta$ it follows that $|D^j u_h|(\lambda_h(F_\delta)) \to |D^j u|(S_u)$ as $j \to \infty$, for each $\delta > 0$. Gathering both limits, we conclude that
    \[
       \liminf_h |D^j u_h|(I) \ge \liminf_{\delta \to 0}\liminf_h |D^j u_h|(\lambda_h(F_\delta)) = \liminf_{\delta \to 0} |D^j u|(F_\delta) = |D^ju|(I). 
    \]
    This finishes the proof.
\end{proof}

\subsection{From $SBV$-curves to fragments}\label{sec:fragments} We write \gls{Theta_SBV} to denote the class of injective $\SBV$-curves $u : [0,1] \to X$ with constant metric speed. For a given $u \in \Theta_{\SBV}(X)$, we consider the cumulative distribution function of its total variation:
\[
g_u(t) \coloneqq \int_0^t d|Du|, \qquad t \in [0,1].
\]
Notice that $g_u$ is right-continuous and strictly monotone; the latter, due to the fact that $u$ has constant metric speed. In fact, $Dg_u = |Du| \ge \ell(u)\mathscr L^1 > 0$ as measures and hence $g_u^{-1}$ is Lipschitz in its domain of definition (which can be  extended to the closure $K_u$ of $\im g_u$). 

\begin{definition}[Fragment map] For any $u \in \Theta_{\SBV}(X)$, we define a fragment $\gamma_u \in \Frag(X)$ by
\[
    \gamma_u \coloneqq u \circ g_u^{-1} : K_u \to X, \qquad K_u \coloneqq \overline{\im u}.
\]
\end{definition}
To see that $\gamma_u$ is indeed a fragment, we use the definition of $g_u$ to observe 
\[
    |\gamma_u(s) - \gamma_u(t)| = |u(g^{-1}_u(s)) - u(g^{-1}_u(t))| = \left| \int_{(g^{-1}_u(s),g^{-1}_u(t)])} d|Du| \right| = |s - t|, \qquad s < t.
\]
This proves that $\gamma_u$ is Lipschitz with constant metric speed $1$, and hence it defines a fragment. Moreover the map is stable under the integral current associations: if $u \in \SBV$, then $\gamma_u \in \Frag(X)$ and by the chain rule for Lipschitz functions (see~\eqref{def:chain_rule} below) and the area formula it holds
\begin{equation}\label{eq:currents}
\begin{split}
    \lpr{u}^a(f,\pi) & = \int f(u) \, D^a(\pi(u)) \\
    & = \int (f \circ u)(g_u^{-1}) \, D^a(\pi \circ u)(g_u^{-1}) (g_u^{-1})' \\
    & = \int f(\gamma_u) \, (\pi (\gamma_u))' = \lpr{\gamma_u}(f,\pi).
\end{split}
\end{equation}
In particular $\lpr{u}^a \equiv \lpr{\gamma_u}$ as currents and the following holds:

\begin{remark}If $u \in \Theta_{\SBV}(X)$, then $\im u$ is $1$-rectifiable in $X$. Moreover, by the area formula for Lipschitz maps into metric spaces (see~\cite[Thm. 7]{Kirchheim_Rectifiable}),
\begin{equation}\label{eq:length_measure}
        \|\lpr{u}^a\| = \|\lpr{\gamma_u}\| = \mathcal H^1 \mres \im u \quad \text{as measures on $X$}.
\end{equation}
\end{remark}

\begin{notation} For a Borel map $u : K \subset [0,1] \to X$, we write  
 $\Gamma_u \subset [0,1] \times X$ to denote its graph. 
 \end{notation}
 The following result establishes sufficient conditions for the continuity of the fragment map.

\begin{lemma}\label{lem:YES} Let $(u_h) \subset \Theta_{\SBV}(X)$ be a sequence converging $u_h$ in $\mathbb D([0,1],X)$ to some $u \in \Theta_{\SBV}(X)$. If moreover $\ell(u_h) \to \ell(u)$ and $|D^j u_h|(I) \to |D^j u|(I)$, then
\[
	\Gamma_{\gamma_{u_h}} \longrightarrow \Gamma_{\gamma_{u}} \quad \text{with respect to the Hausdorff metric in $\mathcal K([0,1] \times X)$.}
\]
In particular, $\gamma_{u_h} \to \gamma_{u}$ in $\Frag(X)$.
\end{lemma}
\begin{proof}
By definition of the Skorokhod metric, there exist monotone bijections $\lambda_h : [0,1] \to [0,1]$ satisfying $\delta_h \coloneqq \sup_{t \in [0,1]} d(u_h(\lambda_h(t)),u(t)) + \|\lambda_h - \mathrm{id}\|_\infty \longrightarrow 0$.

For the ease of notation, we write $g = g_u$, $g_h = g_{u_h}$ for the cumulative distribution functions associated with $|Du$ and $|Du_h|$ respectively; we also write  $K = K_u$ and $K_h = K_{u_h}$ for the closures in $X$ of $\im u$ and $\im u_h$. Then, recalling that the inverse function $g^{-1}:K \to [0,1]$  is Lipschitz continuous,  we define a one-to-one monotone map $\rho_h : K \to K_h$ by the concatenation of the bijective maps
\[
	K \stackrel{g^{-1}}\longrightarrow [0,1] \stackrel{\lambda_h}\longrightarrow [0,1] \stackrel{g_{h}}\longrightarrow K_h.
\]

We \emph{claim} that $\|\rho_h - \mathbf 1_K \mathrm{id}\|_\infty \to 0$ as $h \to 0$. Letting $s = g(t) \in K$ and computing we get
\[
	|\rho_h(s) - s| =   |Du_h|((0,\lambda_h(t)]) - |Du|((0,t])| + |Du_{h}|.
\]
Decomposing $|D(\cdot)| = |D^a(\cdot) + |D^j(\cdot)|$ on the right hand side we obtain the following estimate (recall that $u$ and $u_h$ have constant metric speed at most $1$):
\[
	|\rho_h(s) - s| \le \underbrace{|\ell(u_h) - \ell(u)| +\delta_h}_{\eqqcolon \alpha_h} +  \underbrace{| |D^ju_h|((0,\lambda_h(t)]) - |D^ju|((0,t])|}_{\eqqcolon \beta_h(t)}. 
\]	
By assumption, $\alpha_h \to 0$ as $h \to 0$. In order to prove the claim we need to show that $\beta_h(t) \to 0$ uniformly for all $t  \in g^{-1}(K)$. Let $I_t \coloneqq (g^-(t),g^+(t))$ so that $|D^j u| = \sum_{t \in S_u} |I_t| \delta_t$. Here, we have used that $Dg = |Du|$ as measures. Now, for fixed $\delta >0$, let $\mathcal F_\delta = \set{t \in S_u}{|I_t| > 0}$ be collection of points where the jump discontinuity has size larger than than $\delta$. The cardinality of this set is trivially bounded as $\#(\mathcal F_\delta) \le \delta^{-1}|Du|(I) \le \delta^{-1}$. Now, we use the fact that $u_h$ converges to $u$ in the Skorokhod metric: if $t \in \mathcal F_\delta$, then writing $t_h \coloneqq \lambda_h$ it holds  
\begin{equation}\label{eq:please}
 |\, |I_t| - d(u_h^-(t_h), u^+_h(t_h)) \, | \le 2\delta_h.
\end{equation}
This shows that $s_h \coloneqq \lambda_h(s) \in S_{u_h}$ whenever $|I_t| > 2\delta_h$. 
The fact that $\lambda_h$ is monotone conveys $\lambda_h((0,t]) = (0,t_h]$. Gathering the observations above, we get 
\[
	\beta_h(t) \le 2\delta_h \delta^{-1} + |D^j u_h|(I \smallsetminus \lambda_h(\mathcal F_{\delta}))  +|D^j u|(I \smallsetminus \mathcal F_{{\delta}}).
\]
Now, we make crucial use of the convergence $\eps_h \coloneqq |Du_h|(I) - |Du|(I) \to 0$, as it will allow us to get a a more uniform estimate on the second term of the right hand side: 
\begin{align}\label{eq:porfin}
	|D^j u_h|(I \smallsetminus \lambda_h(\mathcal F_{\delta})) 
	& =  |D^ju_h|(I) -  |D^ju_h|(\lambda_h(\mathcal F_\delta))  \stackrel{\eqref{eq:please}}\le   2\delta_h\delta^{-1}  +|\eps_h| +  |D^j u|(I \smallsetminus \mathcal F_{{\delta}}).
\end{align}
Choosing $\delta = \sqrt{\delta_h}$ and $h$ sufficiently large (so that $ \sqrt{\delta_h} > 2\delta$) gives
\begin{align*}
\beta_h(t) & \le 2\sqrt{\delta_h} + |D^j u_h|(I \smallsetminus \lambda_h(\mathcal F_{\sqrt\delta}))  +|D^j u|(I \smallsetminus \mathcal F_{\sqrt{\delta}}) \\
& \le   4\sqrt{\delta_h} + |\eps_h| + 2|D^j u|(I \smallsetminus \mathcal F_{\sqrt{\delta}}) 
\end{align*}
Since $\mathcal F_\delta \uparrow S_u$ as sets and $|Du|$ is a Radon measure, we conclude that $\beta_h(t) \to 0$ as $h \to 0$. 

This proves the claim. We are now ready to prove the assertion. Recalling that 
\[
	\gamma_{u_h}(\rho_h) = u_h \circ g_h^{-1} \circ (g_h \circ \lambda_h \circ g^{-1}) = u_h \circ \lambda_h \circ g^{-1},
\]
we conclude $\|d(\gamma_{u_h}(\rho_h),\gamma_{u})\|_\infty = \|d(u_h \circ \lambda_h(g^{-1}),u(g^{-1}))\|_\infty\le \delta_h \to 0$. 

Since $\rho_h : K \to K_h$ is bijective, it follows that 
\[
d_H(\Gamma_{\gamma_{u_h}},\Gamma_{\gamma_{u}}) \le \delta_h + \|\rho_h - \mathbf 1_K\mathrm{id}\|_\infty \to 0 \ ,
\]
hence the wished convergence in $\mathcal K([0,1] \times X)$. This finishes the proof.
\end{proof}

\subsection{The area formula for real-valued monotone maps}\label{sec:area}
	Now, assume that $X = \R$.
    The chain rule (see, e.g.,~\cite[Thm.~3.99]{ambrosioFunctionsBoundedVariation2000}) for functions of bounded variation ensures that if $f \in \Lip(\R)$, then $f'$ exists $|\tilde Df|$-almost everywhere and
    \begin{equation}\label{def:chain_rule}
        D(f\circ u) = f'(u) \, \tilde D\gamma +  (f(u^+) - f(u^-)) \, \mathcal H^0 \mres S_\gamma, 
    \end{equation}
    where $\tilde Du \coloneqq D^a u + D^c u$. 
    
    \begin{notation}
    	If $A \subset \R$ be a Borel set, we write (with a possible abuse of notation): 
	\[
		\text{\gls{set}}(f,\pi) \coloneqq \int_A f \pi' \, dt, \qquad f \in \Lip_b(\R), \pi \in \Lip(\R).
	\]    
	\end{notation}
    We have the following characterization:

\begin{lemma}\label{lem:pushforward}
    Let $\phi : [0,1] \to \R$ be a Borel monotone function. 
    Then, the classical pointwise derivative $\phi'$ exists $\mathscr L^1$-almost everywhere and the push-forward of the absolutely continuous and cantorian parts of $D\phi$ under $\phi$ are given by:
    \begin{align*}
        \phi_\# D^a\phi &= \mathscr L^1 \mres \phi(A) \\
        \phi_\# D^c \phi &= \mathscr L^1 \mres \{\, {\mathrm{im} \, \phi} \smallsetminus \phi(A) \,\},
    \end{align*}
    where
    \[
        A \coloneqq \set{  t \in [0,1]}{   \frac{d\mathscr L^1}{d D^c \phi}(t) =  \infty }.
    \]
    In particular, the metric currents associated with $\phi$ can be represented as currents over the following sets:
    \begin{align*}
        \lpr{\phi}^a  &=\llbracket \, \phi(A)\,\rrbracket  \\
         \lpr{\phi}^c  &= \llbracket \, {\mathrm{im} \, \phi} \smallsetminus \phi(A) \,\rrbracket \\
         \lpr{\phi} &= \lpr{\phi^+(0),\phi^-(1)},
    \end{align*}
    
\end{lemma}

\begin{proof}
    Assume that $\phi$ is non-decreasing without loss of generality. Since $\phi \in BV(I)$, the approximate gradient $\nabla \phi$ exists and is non-negative almost everywhere, and satisfies $D^a\phi = \nabla\phi \mathscr L^1$. Notice also that $D^j \phi = \partial \lpr{\phi}^j$, as measures in $\Mbf_0(\R)$. Proposition~\ref{prop:SBV} and the fact that the only cycle on $\R$ is the zero current, imply that $\lpr{\phi} = \lpr{a,b}$ with $a = \phi^+(0)$ and $b = \phi^-(1)$.

    Since $\phi$ has left and right limits at every point, it follows that 
    $\Gamma_\phi  \coloneqq \bigcup_{t \in S_\phi} (\phi^-(t),\phi^+(t)) = [a,b] \smallsetminus \overline {\mathrm{im} \, \phi}$. This, the fact that $\lpr{\phi} = \lpr{a,b}$ and the definition of push-forward yields
    \begin{equation}\label{eq:tilde}
          \mu \coloneqq   \phi_\# \tilde D\phi  = \mathscr L^1 \mres {\mathrm{im} \, \phi} =  \mathscr L^1 \mres {\overline{\mathrm{im} \, \phi}}.
    \end{equation}
 
    By the Besicovitch Differentiation Theorem, the Borel set $A$  is $D^c\phi$-null and has full $\mathscr L^1$-measure. In particular $D^c\phi$ is concentrated on $B \coloneqq \im \phi  \smallsetminus A$ and $D^a \phi$ is concentrated on $A$. If $\phi$ is injective, then $\mu^a \coloneqq  \mu \mres \phi(A)$ and $\mu^c \coloneqq  \mu \mres \phi(B)$ are mutually singular and hence the first assertion holds
    when $\phi$ is injective. 
    
    To address the case when $\phi$ is only monotone, we argue as follows: for any given $\eps > 0$, let $\phi_\eps(t) \coloneqq \phi(t) + \eps t$. Since $\phi_\eps$ is strictly monotone, it is injective. Denoting by $A_u$ and $B_u$ the $u$-dependence of the sets $A$ and $B$ introduced above, we have  $A_{\phi_\eps} = A_\phi$ and $B_{\phi_\eps} = B_\phi$ since $D^c \phi_\eps = D^c \phi$. By the previous step, for every $g \in C_c(\R)$ it holds
    \[
       \int \ g(\phi_\eps) \, dD^a \phi + \eps \int g(\phi_\eps) \, dt = \int \ g(\phi_\eps) \, dD^a \phi_\eps = \int_{\phi_\eps(A_{\phi})} g \, dt.
    \]
    Since $\phi_\eps \to \phi$ uniformly and $1_{\phi_\eps(A_{\phi})} \to 1_{\phi(A_\phi)}$ in $L^1$ as $\eps \to 0$, we conclude from the identity above that 
    \[
        \int f(\phi) \, dD^a \phi =   \int_{\phi(A_\phi)} f \, dt.
    \]
    This, together with~\eqref{eq:tilde} gives
    \[
        \int f(\phi) \, D^c \phi = \int_{\phi(B_\phi)} f \, dt,
    \]
    as desired. The identities for $\lpr{\phi}^a$ and $\lpr{\phi}^c$ are consequences of the push-forward identities. 
\end{proof}

The following version of the area formula is a direct consequence of the previous result:

\begin{corollary}[Area formula]\label{cor:area}
    Let $\gamma : \R \to \R$ be a monotone non-decreasing function and let $g \in L^1(\R)$. Then,
    \[
        \int_{{U \cap \im \gamma}} \,  g(x) 
        \, dx  = \int_{\gamma^{-1}(U)} g(t)\, d\tilde D\gamma(t) \qquad \text{for all Borel sets \, $U \subset \R$.}
    \]
    Moreover, these integrals do not depend on the choice of Lebesgue representative of $\gamma$.
\end{corollary}

\subsection{Parametrization of sets in the real line}

For a Borel set $K \subset \R$, we write $\mathbf 1_K \in L^1(\R)$ to denote its Lebesgue representative indicator function. In this way, we shall indentify the Borelians $\mathcal B([0,1])$ modulo negligible sets as a metric subspace of $L^1((0,1))$. We will write \gls{|K|} to denote the one-dimensional Lebesgue measure of $K$. The following lemma demonstrates a fundamental property: any Borel set (up to Lebesgue equivalence classes) in the real line can be parametrized by arc-length using a $BV$ curve.

\begin{figure}[h]
\centering 
\fbox{\includegraphics[width=.5\textwidth]{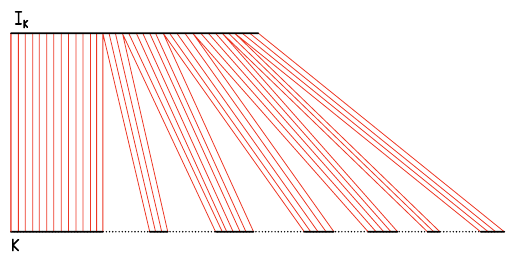}}
\caption{Optimal transport maps (paths in red) $T : I_K \coloneqq [0,|K|] \to K$ associated with the transportation of
 $\mathcal L^1 \mres I_K$ into ${\mathcal L}^1 \mres K$. The monotone map $u_K$ constructed in Lemma~\ref{lem:arc-length} can be expressed in terms of $T$ as $u_K(t) = |K| \cdot T(|K|^{-1}t)$.}  
 \end{figure}
 
\begin{lemma}[Transport parametrization]\label{lem:arc-length}
Let $K$ be a Borel set in $[0,1]$.  There exists a monotone non-decreasing Càdlàg function $\text{\gls{K_parameter}} : [0,1] \to  [0,1]$ satisfying:
\begin{align*}
    \lpr{u_K}^{a}  & = \llbracket K \rrbracket\\
    \lpr{u_K}^{c}  & = \lpr{\, \im u_K \smallsetminus K\,} = \llbracket \, \overline K \smallsetminus K \, \rrbracket \\
    \lpr{u_K}^j & = \lpr{\mathrm{ess\, inf} \, K, \mathrm{ess\, sup} \, K} \smallsetminus \lpr{\, \overline K \,}.
\end{align*}
Moreover, 
\begin{enumerate}[itemsep=0.2em, leftmargin=2em, label=\roman*), topsep=0.5em]
\item \label{item:i_arc} For $\mathscr L^1$-almost every $t \in (0,1)$, it holds
\[
   u'_K(t) \coloneqq \lim_{|h| \to 0} \frac{u_K(t + h) - u_K(t)}{|h|} = |K|.
\]
\item \label{item:ii_cont} The map $\mathcal B([0,1]) \subset L^1([0,1]) \to L^1((0,1)) : K \mapsto u_K$ \label{item:iii_arc}
is continuous. 
\end{enumerate}
\end{lemma}

\begin{proof} The case when $|K| = 0$ is trivial and therefore we may assume that $|K| > 0$.

Define $f : [0,1] \to [0,1]$ by $f(x) = |K|^{-1}\int_0^x 1_K = |K \cap [0,x]||K|^{-1} $. The function $f$ is non-negative, monotone non-decreasing,  $1-$Lipschitz and satisfies $f' = |K|^{-1}$ almost everywhere on $K$. Let $a = \mathrm{ess \, inf} \, K$ and $b = \mathrm{ess \, sup} \, K$. Now, let us consider its pseudoinverse $\gamma : [0,1] \to [0,1]$ defined by $\gamma(t) =  
\inf \set{ x \in [0,1]}{f(x) > t}$. Here, we have adopted the convention that the infimum above is $b$ when $\{f  > t\}$ is empty.
Notice that this function is defined by the condition $\gamma(t) \leqslant x$ if and only if $t \leqslant f(x)$. Since $f$ is monotone, so is $\gamma$ and in particular $\gamma \in BV(I)$. In addition, by definition and the Lipschitzianity of $f$, it follows that $\gamma$ is right-continuous (independently of its $L^1$-representative).   Notice that $\gamma(0^+) = a$ and $\gamma(1^-) = b$, and
more generally $f(\gamma)(t) = t$ on $I$.
This follows from the left-continuity of $\gamma$ and the definition of $f$ and $\gamma$. 

The chain rule for BV functions (see \cite[Theorem 3.99]{ambrosioFunctionsBoundedVariation2000}) guarantees that $1_K(\gamma)$ is well-defined $|\tilde D \gamma|$-a.e. ($\gamma$ is continuous $|\tilde D \gamma|$-a.e.) and conveys the following identity of measures:
\[
\mathscr L^1 \mres I \equiv  |K|^{-1} 1_K (\gamma) \, \dot \gamma \, \mathscr{L}^1 \, + \,  |K|^{-1}  1_K(\gamma) D^c \gamma \, + \, (f(\gamma^+) - f(\gamma^-)) \, \mathcal H^1 \mres S_\gamma
\] 
We deduce that $\gamma(t) \in K$ and $\dot \gamma(t) = |K|$ for $\mathscr L^1$-almost every $t \in I$ and also that the last two terms on the right-hand side above vanish, which gives $|K| \mathscr L^1 \mres I \equiv   1_K (\gamma) \, \dot \gamma \, \mathscr{L}^1$. Integrating both sides and recalling the pushforward identity $\mathscr L^1 \mres \gamma(A_\gamma) = \gamma_\# D^a\gamma$ (cf. Lemma~\ref{lem:pushforward}), we get
\[
	|K|  = \int_I 1_K(\gamma) \dot \gamma \, dt=  \int_{K}   \, d \gamma_\#D^a \gamma = \mathscr L^1 \mres \gamma(A_\gamma) (K) = |\gamma(A_\gamma) \cap K|.
\]
From these two observations, we deduce that $|\gamma(A_\gamma) \triangle K| = 0$. Recalling Proposition~\ref{prop:SBV} and the identities from Lemma~\ref{lem:pushforward} we get the desired expressions for $\lpr{\gamma}^a$, $\lpr{\gamma}^c$ and $\lpr{\gamma}^j$.
The first part of the assertion follows by setting $u_K \coloneqq \gamma$.

\emph{Proof of~\ref{item:i_arc}.} To see that the pointwise derivative $\gamma'$ exists almost everywhere in $I_K$, we recall that the fundamental theorem of calculus holds with $D\gamma$ outside of the discontinuity set (see~\cite{ambrosioFunctionsBoundedVariation2000}): 
\[
    \gamma(t + h) - \gamma(t) =  D\gamma(t, t+h) \qquad \forall t,h \in I_K \smallsetminus S_\gamma.
\]
From this, the monotonicity of $\gamma$ and the Lebesgue--Besicovitch differentiation theorem (which on the real line holds with half-intervals) it follows that
\[
    \lim_{|h| \to 0}\frac{\gamma(t + h) - \gamma(t)}{|h|} = \dot \gamma(t) = |K| \qquad \text{for $\mathscr L^1$-almost every $t \in I_K$.}
\]

\emph{Proof of~\ref{item:ii_cont}.} The proof of this part can be found in~\cite[Prop. 2.17]{santambrogioOptimalTransportApplied2015}. For the convenience of the reader, we have furbished a proof here. For an integrable function $f : [0,1] \to [0,1]$ and $t \ge 0$, we write $E_{f,t} = \{f \ge t \}$ denote the $t$ super-level set of $f$. We recall the layer cake representation
\begin{equation}\label{eq:layer}
    f(x) = \int_0^1 1_{E_{f,t}}(x) \, dt. 
\end{equation}
for later use.  Now, let $K,L \subset [0,1]$ be any two given Borel sets. 
If $x \in [0,1]$, then by definition it holds $u_K(t) = \mathrm{ess}\,\inf E_{f_K,t}$ and similarly for $\gamma_L(t)$. Moreover, by the monotonicity of $f_K,f_L$, we can express the pointwise distance from $u_K$ to $\gamma_L$ as the Lebesgue size difference of their associated super level sets, i.e., 
\begin{equation}\label{eq:super_identity}
|u_K(t) - \gamma_L(t)| = | \mathscr L^1(E_{f_K,t}) - \mathscr L^1(E_{f_L,t})|  \qquad \forall t \in [0,1].
\end{equation}
Expressing $f_K,f_L$ in terms of the layer cake formula~\eqref{eq:layer} and changing the order of integration by means of the Fubini-Tonelli Theorem we deduce 
\begin{align*}
    \int_0^1 |f_K(x) - f_L(x)| \, dx & = \int_0^1 | \mathscr L^1(E_{f_K,t}) - \mathscr L^1(E_{f_L,t})| \, dt \stackrel{\eqref{eq:super_identity}}=  \|u_K - \gamma_L\|_{L^1(0,1)}.  
\end{align*}
This proves that the assignment $f_K \mapsto u_K$ is continuous from $L^1(0,1)$ to $L^1(0,1)$. In fact, it is a continuous isometry. Finally, we make the simple observation that (by the very definition of $f_K$)
\begin{equation}\label{eq:f_gamma}
    |f_K(x) - f_L(x)| \le \int_0^x |1_K(s) - 1_L(s)| \, ds \le \|1_K - 1_L\|_{L^1(0,1)}.
\end{equation}
This proves that $1_K \mapsto f_K$ is continuous. In particular, $1_K \mapsto u_K$ is continuous as it is the concatenation of a continuous map and a continuous isometry. 
This finishes the proof.
\end{proof}

    \begin{remark}\label{rem:notice}
        Notice that 
        \begin{itemize}
        \item $D^j \gamma$ vanishes  if and only if $K$ is essentially connected,
        \item $\gamma \in SBV(I)$ if and only if $K$ is essentially closed. 
        \end{itemize}
    \end{remark}

\begin{example}[Fat Cantor set II] Let $X \subset [0,1]$ be the fat Cantor set from Example~\ref{ex:fat}. Recall that $X$  is closed, is nowhere dense and has measure $|X| = 1-\alpha$ for some $\alpha \in (0,1/3)$. Consider its open complement $A = [0,1] \smallsetminus X$ and notice that $\partial A = X$.
Therefore, 
    \begin{align*}
        \lpr{u_X}^a = \lpr{X}, \qquad \lpr{u_X}^c = 0, \qquad \lpr{u_X}^j \equiv 0,
    \end{align*}
    and
        \begin{align*}
        \lpr{u_A}^a = \lpr{A}, \qquad \lpr{u_A}^c = \lpr{X}, \qquad \lpr{u_X}^j \equiv 0.
    \end{align*}
   \end{example}


\section{Proof of the $SBV$-representation}\label{sec:rep}
This section presents the construction of the $\SBV(I,X)$ representation for $1$-dimensional metric currents in Banach spaces. We begin by setting up the necessary preliminary tools, including the crucial hole filling of metric currents (cf. Theorem~\ref{cor:cycle_cover}) and the decomposition of cyclic currents by Paolini and Stepanov. We then use the structure of the hole filling results to define a map that pushes continuous curves into $\SBV$ curves under the Càdlàg topology. In order to keep the presentation as clear as possible, we have decided to cover the rigorous (and somewhat nontrivial) measurability properties of the push-forward map in the Appendix.  

\subsection{Proof on separable Banach spaces}We first establish Theorem~\ref{thm:rep} for Banach spaces \textbf{of dimension at least $2$}.  

\subsection{Preliminaries}
To begin, fix an arbitrary $\eps > 0$. As a first step, we augment our Banach space $X$ to $X' = X \oplus \mathbb{R}$, where $X_0 = X \oplus \{0\}$ is an isometric embedding of $X$, and $X^+ = X \oplus (0,\infty)$ denotes the upper half-space. As shown in Corollary~\ref{cor:cycle_cover2}, for any given current $T \in \mathbf{M}_1(X)$, we can find a cycle $C \in \mathbf{M}_1(X')$ and a rectifiable current $R \in \mathcal{R}_1(X')$ such that:
\begin{enumerate}
    \item $C = i_\# T + R$, where $i : X \embed X_0$ is the canonical isometric isomorphism.
    \item $i_\#T = C \mres X_0$ and $R = C \mres X^+$. \label{item:restriction}
    \item $\mathbf{M}(R) \le \|\partial T\|_{\KR} + \eps$. \label{item:4}
    \item $R$ is of the form $R = \sum_{k = 1}^\infty \eta_k \llbracket x_k,y_k\rrbracket$, with $x_k,y_k \in X^+$ and $\eta_k \in \mathbb{R}$, such that the union of open intervals $\bigcup_{k=1}^\infty(x_k,y_k)$ is disjoint.
\end{enumerate}
\begin{figure}[h]
\centering
\includegraphics[width=0.7\textwidth]{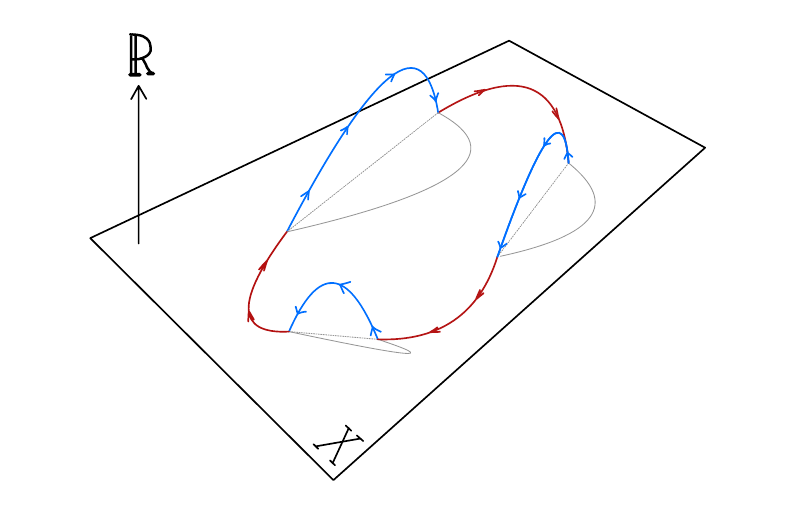}
\caption{Depiction of $i_\#T$ (in red) and $R$ (in blue).}
\label{fig:X_0}
\end{figure}
As usual, we denote $C([0,1],X')$ as the space of continuous curves equipped with the uniform convergence topology. By Theorem 3.1 and Corollary 4.1 in~\cite{paolini2012decomposition2}, there exists a finite Borel measure $\bar\eta$ on $C([0,1],X')$ concentrated on $\mathrm{Lip}_1([0,1],X')$, such that
\begin{align*}
    C(\omega') & = \int_{C([0,1],X')} \llbracket \theta \rrbracket(\omega') \, d \bar\eta(\theta)
\end{align*}
for all metric $1$-forms $\omega'$ on $X'$. This superposition is free of mass cancelations:
\begin{align}\label{eq:mass}
    \mathbf{M}(C) & = \int_{C([0,1],X')} \mathbf{M}(\llbracket \theta\rrbracket) \, d \bar\eta(\theta) = \int_{C([0,1],X')} \ell(\theta) \, d \bar\eta(\theta).
\end{align}
Moreover, $\bar\eta$-almost every curve $\theta$ is an injective arc (hence $\mathbf{M}(\llbracket\theta\rrbracket) = \ell(\theta) = 1$) with constant speed $|\dot \theta| \equiv 1$. The evaluation maps $e_t : C([0,1],X') \to X'$ (defined by $e_t(\theta) = \theta(t)$ for $t \in [0,1]$) are Borel, and $(e_0)_\# \bar\eta = (e_1)_\# \bar\eta = \|C\|$. From property~\ref{item:restriction}, we deduce $T,R \le C$ and their associated measures $\|T\|, \|R\|$ are mutually singular. In particular, from~\eqref{eq:mass} it follows that for any $\omega' \in \mathscr D^1(X')$ we have
\begin{align}
 i_\# T(\omega') & = \int_{C([0,1],X')} \llbracket  \theta \rrbracket \mres X_0(\omega') \, d \bar\eta(\theta) \label{eq:T}\\
\mathbf{M}(i_\#T) & = \int_{C([0,1],X')} \mathbf{M}(\llbracket \theta \rrbracket \mres X_0) \, d \bar\eta(\theta) \label{eq:T_mass}
\end{align}
and
\begin{align}
R(\omega') & = \int_{C([0,1],X')} \llbracket \theta \rrbracket \mres X^+(\omega') \, d \bar\eta(\theta) \label{eq:R}\\
\mathbf{M}(R) & = \int_{C([0,1],X')} \mathbf{M}(\llbracket \theta \rrbracket \mres X^+) \, d \bar\eta(\theta). \label{eq:R_mass}
\end{align}
\begin{remark}\label{rem:isometry}
Notice that, since $i$ is an isomorphism of Banach spaces, 
\[
	T(\omega) = T(\omega \circ i^{-1} \circ i) = i_\#T(\omega \circ i^{-1}) \qquad \text{for all $\omega \in \mathscr D^1(X)$}.
\]	
The fact that $i$ is an isometry further gives $\Mbf(T) = \Mbf(i_\#T)$.
\end{remark}

\begin{remark}
    In all that follows we will make use of the following property (see~\cite[Sec.~6]{Schioppa_2016}): for any $\omega \in \mathscr D^1(X)$, the maps $\gamma \mapsto \lpr{\gamma}(\omega)$ and $\gamma \mapsto \ell(\gamma)$ are Borel from $\Frag(X)$ to $\R$, where the space of fragments is topologized as subspace of $\mathbf K([0,1] \times X)$. 
\end{remark}

\subsection{Constructing $\mathbf{SBV}$-curves via push-forward} In all that follows, we write $I \coloneqq (0,1)$.
The next step is to construct a $\bar \eta$-measurable map from $C([0,1],X')$ to $\SBV$-curves on $\mathbb D([0,1],X)$. This push-forward mechanism is central to translating the decomposition of cyclic currents into a representation in terms of $\SBV$ paths.

Let $\theta \in C([0,1],X')$ and let $K_\theta \coloneqq \theta^{-1}(X_0) \in \mathcal{K}([0,1])$. By Lemma~\ref{lem:arc-length}, there exists a monotone non-decreasing Càdlàg curve $\gamma : [0,1] \to [0,1]$ belonging to $BV(I)$, with $\dot \gamma_\theta \equiv \gamma_\theta'$ in $L^1(I)$, and satisfying (cf. Remark~\ref{rem:notice}):
\begin{equation}\label{eq:sets}
\begin{split}
    \llbracket\gamma_\theta\rrbracket^a & = \llbracket\mathrm{im} \, \gamma_\theta\rrbracket = \llbracket K_\theta\rrbracket \\
    \llbracket\gamma_\theta\rrbracket^c & = \llbracket\mathrm{im} \, \gamma_\theta \smallsetminus K_\theta\rrbracket \equiv 0 \\
    \llbracket\gamma_\theta\rrbracket^j & = \llbracket a_\theta,b_\theta\rrbracket - \llbracket \mathrm{im} \, \gamma_\theta \rrbracket = \llbracket (a_\theta,b_\theta) \smallsetminus \im \gamma_\theta\rrbracket,
    \end{split}
\end{equation}
where $a_\theta = \mathrm{ess \, inf} \, K_\theta$ and $b_\theta = \mathrm{ess \, sup} \, K_\theta$. Therefore $\gamma_\theta \in SBV(I)$ for all $\theta \in C([0,1],X')$.

Define $\beta : C([0,1],X') \to \mathbb D([0,1],X) \cap \SBV(I,X)$ by $\beta(\theta) \coloneqq \aleph \circ \theta \circ \gamma_\theta$, where $\aleph : X' \to X$ is the composition of the isometry $i^{-1} : X_0 \to X$ with the canonical projection $p : X' \twoheadrightarrow X_0$. Here, the space $\mathbb D([0,1],X)$ is  topologized with the Skorokhod metric~\eqref{eq:skorokhod}.
Notice that $\beta$ is well-defined since $\mathbb D([0,1])$ and $SBV(I) = \SBV(I)$ are both stable under the composition with metric-valued Lipschitz maps.

Proposition~\ref{prop:measurability_beta} in the Appendix and the lipschitzianity of $\aleph : X_0 \to X$ warrant that $\beta$ is $\bar \eta$-measurable. We have thus far defined an $\bar \eta$-measurable map that assigns an $\SBV$-curve to every continuous curve with values on $X$. The crucial property for $\beta$ is contained in the following chain of identities: First, observe that if $(f,\pi) \in \Lip_b(X) \times \Lip(X)$, then  
\begin{equation}\label{eq:crucial_beta}
\begin{split}
	\lpr{\beta(\theta)}^a(f,\pi) &  = \lpr{\gamma_\theta}^a (f(\aleph \circ \theta),\pi(\aleph \circ \theta)) = \lpr{K_\theta} (f(\aleph \circ\theta) ,\pi(\aleph \circ\theta) )
\end{split}
\end{equation}
Here, in passing to the first equality we have simply used the definition of $\lpr{\cdot}^a$. The second equality follows from the identity $\lpr{\gamma_\theta}^a = \lpr{K_\theta}$ from~\eqref{eq:sets}. 
Using that $\mathbf 1_{K_\theta} \equiv \mathbf 1_{X_0} \circ \theta$ and computing on the right-hand side of~\eqref{eq:crucial_beta},  we deduce that 
\begin{equation}\label{eq:crucial_beta2}
\begin{split}
	\lpr{\beta(\theta)}^a(f,\pi) =  \int_0^1 (\mathbf 1_{X_0} \circ \theta) (f(\aleph) \circ\theta) (\pi(\aleph) \circ\theta))'  
	= \lpr{\theta}\mres X_0 (f(i^{-1})), \pi(i^{-1})), 
	\end{split}
\end{equation}
where the last inequality follows from the fact that $\aleph|_{X_0} = i^{-1}$.

\subsection*{The push-forward}We define $\Theta_{\SBV}(X) \subset \mathbb D([0,1],X)$ as the metric subspace of injective Càdlàg functions, whose restriction to $I$ are $\SBV$, and possess constant metric speed (i.e., $|\dot \gamma| \equiv \ell(\gamma)$). 

We claim that $\beta(\theta) \in \Theta_{\SBV}(X)$ for $\bar \eta$-almost every $\theta$. 
Since $\bar \eta$ is concentrated on \gls{Theta_X}, the subspace of $\Lip_1([0,1],X)$ of all injective curves with constant metric speed one, we may assume without loss of generality that $\theta \in \Theta(X)$.  
The injectivity of $\beta(\theta)$ is clear. We are left to demonstrate the constant-speed property. 
To verify this, let us first recall from Lemma~\ref{lem:arc-length} that $\gamma'_\theta(t) = \ell(\gamma_\theta)$ for $\mathscr L^1$-a.e. $t \in (0,1)$. 
Since $\theta \in \Theta(X)$, we have $|\dot \theta|(s) = 1$ for $\mathscr L^1$-almost every $s \in [0,1]$. Then, resorting to properties of the metric differential established in~\cite[Thm.~2]{Kirchheim_Rectifiable}, we find that for $\mathscr L^1$-almost every $x \in \mathrm{im}(\gamma_\theta)$, it holds $| \theta(z) - \theta(y) |_X - |z-y| = \mathrm{o}(|z - x| + |y - x|)$. 
Letting $z = \gamma_\theta(t + h)$ and $y = \gamma_\theta(t)$, we conclude that:
\[
\frac{| \theta(\gamma_\theta(t + h)) - \theta(\gamma_\theta(t)) |_X}{|\gamma_\theta(t+h)-\gamma_\theta(t)|} -1 = \mathrm{o}(|\gamma_\theta(t + h) - x| + |\gamma_\theta(t) - x|).
\]
Letting $x = \gamma_\theta(t)$ in the identity above yields 
\[
    |\nabla (\theta \circ \gamma_\theta)|(t) = \lim_{|h| \to 0} \frac{| \theta(\gamma_\theta(t + h)) - \theta(\gamma_\theta(t)) |_X}{|h|} = \ell(\gamma_\theta).
\]
This proves the claim. 

With the measurability of $\beta$ and the constant speed of $\beta(\theta)$ established, we can now define a finite Borel measure $\eta$ on $\mathbb D([0,1],X)$ by pushing forward $\bar \eta$ under $\beta$, i.e., $\eta = \beta_\# \bar \eta$. This measure will serve as our "$SBV$ representation measure." Notice from the previous step that $\eta$ is concentrated on $\Theta_{\SBV}(X)$. 

\subsection{The representation}\label{sec:rep_meas}

By definition, for any $\eta$-summable Borel map $g : \mathbb D([0,1],X) \to \mathbb{R}$, its integral with respect to $\eta$ is:
\begin{equation}\label{eq:push}
\int_{\mathbb D([0,1],X)} g(u) \, d\eta(u) \coloneqq \int_{C([0,1],X')} g(\beta(\theta)) \, d\bar \eta(\theta).
\end{equation}
Let $\omega \in \mathscr D^1(X)$ be an arbitrary metric $1$-form. We define a map $g : \mathbb D([0,1],X) \to \mathbb{R}$ by
\[
g(u) =
\begin{cases}
\lpr{u}^a(\omega) & \text{if $u \in \Theta_{\SBV}(X)$} \\
0 & \text{else}
\end{cases}
\]
To be able to represent $T$ as a current, the key step is to prove that $g$ is Borel measurable. We recall from Lemma~\ref{lem:semicont} that the maps $u \mapsto |Du|(I)$, $|D^au|(I)$, and $|D^ju|(I)$  are measurable with respect to the $\mathbb{D}([0,1],X)$ topology. By Lusin's theorem for finite Borel measures in metric space (see~\cite[Thm. 17.12]{Kechris}), it suffices to show that $g$ is Borel on any set where these maps are continuous. 

Consider any Borel set $F \subset \mathbb{D}([0,1],X) \cap \Theta_{\SBV}(X)$ on which these maps are continuous. Lemma~\ref{lem:porfin} provides the following convergence property: for any sequence $u_j \to u$ in $F$, their associated fragments converge, i.e., $\gamma_{u_j} \to \gamma_u$ in the fragment topology on $\Frag(X)$. Since $\llbracket u\rrbracket^a = \llbracket \gamma_u\rrbracket$ as currents and $g$ is known to be Borel measurable with respect to the fragment topology (see~\cite[Sec.~6]{Schioppa_2016}), we conclude that $g$ is Borel measurable on $F$. This establishes the overall $\eta$-measurability of $g$.

Since also $B \mapsto \|\lpr{\gamma}\|(B)$ is Borel for all $\gamma \in \Frag(X)$ and all Borel $B \subset X$, the Lusin-type property argument above can also be applied to show that the map
\[
    u \mapsto
\begin{cases}
\|\lpr{u}^a\|(B) = \mathcal H^1(\im u \cap B) & \text{if $\gamma \in \Theta_{\SBV}(X)$} \\
0 & \text{else}
\end{cases}
\]
is $\eta$-measurable. 

We can hence now make use of the push-forward identity~\eqref{eq:push} with these two maps.

By construction, we obtain (cf. Remark~\ref{rem:isometry}):
\begin{align*}
   \int_{\mathbb D([0,1],X)}  \llbracket u\rrbracket^a(\omega) \, d\eta(u) & \stackrel{\eqref{eq:push}}= \int_{C([0,1],X')} \llbracket\beta(\theta)\rrbracket^a(\omega) \, d\bar \eta(\theta) \\
    & \stackrel{\eqref{eq:crucial_beta2}}= \int_{C([0,1],X')} \llbracket \theta \rrbracket \mres X_0 (\omega\circ i^{-1}) \, d\bar \eta(\theta) \stackrel{\eqref{eq:T}}= i_\#T(\omega\circ i^{-1}) = T(\omega).
\end{align*}
This proves that $\eta$ represents $T$ as a superposition over $\Theta_{\SBV}(\Theta)$. Testing with metric differentials of the form $\omega = (1,f)$, we deduce that the boundary of $T$ also satisfies a similar representation: for every $f \in \Lip(X)$ it holds
\begin{equation}\label{eq:KRe}
\begin{split}
	\partial T(f) & = \int_{\mathbb D([0,1],X)}  \partial \llbracket u \rrbracket^a(f) \, d\eta(u) \\
	& = \int_{\mathbb D([0,1],X)}  \{ f(u^+(0)) - f(u^-(1)) - D^ju(f) \} \, d\eta(u). 
	\end{split}
\end{equation}

\subsection{Absence of mass cancellations}
Let $B \subset X$ be a Borel set. Similarly as above, we get
\begin{align*}
   \int_{\mathbb D([0,1],X)}  \|\llbracket u\rrbracket^a\|(B) \, d\eta(u) & \stackrel{\eqref{eq:push}}= \int_{C([0,1],X')} \|\llbracket\beta(\theta)\rrbracket^a\|(B) \, d\bar \eta(\theta) \\
    & \stackrel{\eqref{eq:crucial_beta2}}= \int_{C([0,1],X')} \|\llbracket \theta \rrbracket\|(X_0 \cap iB) \, d\bar \eta(\theta) \stackrel{\eqref{eq:T}} \ge \|i_\#T\|(iB) = \|T\|(B).
\end{align*}
This tells us that
\begin{equation}\label{eq:measures}
    \int_{\mathbb D([0,1],X)}  \|\llbracket u\rrbracket^a\| \ge \|T\| \quad \text{as measures on $X$.}
\end{equation}

Recalling~\eqref{eq:T_mass}, we get:
\begin{align*}
   \Mbf(T) = \mathbf{M}(i_\#T) & = \int_{C([0,1],X')} \mathbf{M}( \lpr{\theta} \mres X_0) \, d\bar\eta(\theta) \\
    & \stackrel{\eqref{eq:crucial_beta}}= \int_{C([0,1],X')} \mathbf{M}(\llbracket\beta(\theta)\rrbracket^a) \, d\bar\eta(\theta) = \int_{\mathbb D([0,1],X)}  \mathbf{M}(\llbracket\gamma\rrbracket^a) \, d\eta(\gamma).
\end{align*}
Since $\Mbf(T) = \|T\|(X)$, we conclude that the inequality of measures in~\eqref{eq:measures} must be an equality and hence 
\[
     \int_{\mathbb D([0,1],X)}  \|\llbracket\gamma\rrbracket^a\|(B) = \|T\|(B) \quad \text{for all Borel sets $B \subset X$.}
\]
This proves the mass decomposition for $T$, in terms of $\eta$. 

\subsection{Estimate for the jump part} Let $\theta \in \Lip([0,1],X)$. Recalling Remark~\ref{rem:composition} we compute
\begin{align*}
    \llbracket\beta(\theta)\rrbracket^j & = \sum_{t \in S_{\gamma_\theta}} \llbracket \aleph \circ \theta(\gamma^-_\theta(t)),\aleph \circ \theta(\gamma^+_\theta(t))\rrbracket \\
    & = \aleph_\# \lpr{\theta \circ \gamma_\theta}^j = \aleph_\#\theta_\# \llbracket [a_\theta,b_\theta] \smallsetminus K_\theta \rrbracket = \aleph_\# \lpr{\theta|_{a_\theta,b_\theta}} \mres X^+.
\end{align*}
Since $\Lip(\aleph) \le 1$, we deduce $\Mbf(\llbracket\beta(\theta)\rrbracket^j) \le \Mbf(\lpr{\theta}\mres X^+)$. From this bound for Lipschitz curves on $X'$ and the fact that $\bar \eta$ is concentrated in $\Lip_1([0,1],X')$, we deduce:
\begin{align*}
    \int_{\mathbb D([0,1],X)} |D^j \gamma|(0,1)\, d\eta(\gamma) & = \int_{\mathbb D([0,1],X)}  \mathbf{M}(\llbracket\gamma\rrbracket^j)\, d\eta(\gamma) \\
    & = \int_{C([0,1],X')} \mathbf{M}(\llbracket\beta(\theta)\rrbracket^j) \, d\bar \eta(\theta) \\
    & \le \int_{C([0,1],X')} \mathbf{M}(\llbracket\theta\rrbracket\mres X^+) \, d\bar \eta(\theta) \stackrel{\eqref{eq:R_mass}}= \mathbf{M}(R) \stackrel{\eqref{item:4}}\le \|\partial T\|_{\KR} + \eps.
\end{align*}
This establishes the Kantorovich estimates for the superposition of the approximate jump derivatives. 

This finishes the proof of the $SBV$-representation in Banach spaces (of dimension at least $2$). \qed

\subsection{Proof of the $\mathbf{SBV}$ representation in separable metric spaces}
This section extends the $\mathbf{SBV}$ representation from Banach spaces to general separable metric spaces. The key idea is to leverage isometric embeddings of separable metric spaces into $\ell^\infty$ and then apply the result established in the previous section.

\subsection{Push-forwards between ssometric metric spaces}\label{sec:push}
Here, we establish how metric currents and $\SBV$-curves behave under isometric push-forwards, which is essential for transporting our representation result.

Let $X,X'$ be metric spaces and let $\Phi : X \to X'$ be an isometric isomorphism. Recall that every current $T \in \mathbf{M}_1(X)$ induces a current $T' \in \mathbf{M}_1(X')$ by the push-forward $T' = \Phi_\# T$, defined by:
\[
    \text{\gls{push}}(f,\pi) = T(f\circ \Phi,\pi\circ \Phi), \qquad (f,\pi) \in \mathscr D^1(X').
\]
The locality and continuity properties are straightforward to verify, as the composition of Lipschitz maps is Lipschitz. Since $\Phi$ is an isometry, $\mathrm{Lip}(\pi \circ \Phi) = \mathrm{Lip}(\pi)$ for any $\pi \in \mathrm{Lip}(X')$. Thus,
\[
    |T'(f,\pi)| \le \mathrm{Lip}(\pi) \int_X |f \circ \Phi| \, d\|T\| = \int_X |f| \, d(\Phi_\#\|T\|),
\]
where $\Phi_\# \|T\|$ is the standard push-forward of a Borel measure. This implies $\|T'\| \le \|\Phi_\# T\|$. Because $\Phi^{-1}$ is also an isometry, an analogous argument shows $\|T\| \le \|(\Phi^{-1})_\# T'\|$. Therefore, $\Phi_\# : \mathbf{M}_1(X) \to \mathbf{M}_1(X') : T \mapsto T'$ is a bi-Lipschitz isomorphism.

Notice that if $\Phi : X \to X'$ is an isometry, then the map $\Phi_{\mathbb D}: \gamma \mapsto \Phi\circ \gamma$ also induces an isometry of Cadlag spaces $\mathbb D([0,1],X) \to \mathbb D([0,1],X')$. Indeed, since $\Phi$ is bi-Lipschitz from $X$ to $X'$, it follows that $\Phi_{\mathbb D}$ is bi-Lipschitz. The fact that it is an isometry follows directly from the fact that $\Phi$ is itself an isometry. This isometry also restricts to an isometry between $\Theta_{\SBV}(X)$ and $\Theta_{\SBV}(X')$.

\subsection{Embedding a separable metric space $X$ into a separable Banach space}
The classical Kuratowski embedding theorem states that every (pointed) separable metric space can be isometrically embedded into $\ell^\infty$. This section briefly reviews the construction of this embedding.
Let $(X,d)$ be a separable metric space. Choose a countable dense set $(x_h)_{h = 1}^\infty \subset X$ and an arbitrary reference point $x_o \in X$. Define the map $\Phi : X \to \ell^\infty$ by:
\[
    \Phi(x) \coloneqq (d(x,x_h) - d(x_o,x_h))_{h =1}^\infty, \qquad x \in X.
\]
The triangle inequality ensures that $\|\Phi(x)\|_{\ell^\infty} = \sup_{h \in \mathbb{N}} |d(x,x_h) - d(x_o,x_h)| \le d(x,x_o)$, so the map is well-defined.
Since $X$ is Hausdorff, $\Phi$ is injective. For any $x,y \in X$, the triangle inequality also gives:
\[
    \| \Phi(x) - \Phi(y)\|_{\ell^\infty} = \sup_{h \in \mathbb N} |d(x,x_h) - d(y,x_h)| \le d(x,y).
\]
Conversely, the map $z \mapsto |d(x,z) - d(y,z)|$ is $1$-Lipschitz. By the density of $(x_h)$, for any $\varepsilon > 0$, we can find $x_h$ such that $|d(x,x_h) - d(y,x_h)| \ge d(x,y) - \varepsilon$. Taking the supremum and letting $\varepsilon \to 0^+$, we conclude that $\| \Phi(x) - \Phi(y)\|_{\ell^\infty} = d(x,y)$.
Thus, $\Phi$ is an isometry. This implies that $\Phi$ restricts to a bijective isometric isomorphism from $X$ into its image $X' \coloneqq \Phi(X)$, which is itself contained in a separable subspace of $\ell^\infty$.

\subsection{Proof of the representation Theorem~\ref{thm:rep}}
We now combine the previous sections to prove the $SBV$ representation for $1$-dimensional metric currents in any separable metric space.
Let $X$ be a separable metric space and $T \in \mathbf{M}_1(X)$ be a $1$-dimensional metric current.
Let $\Phi : X \to \ell^\infty$ be the isometric embedding constructed in the previous section. Let $E \subset \ell^\infty$ be a separable Banach space of \textbf{dimension at least $2$} containing $X' = \Phi(X)$ (this can always be found by considering a larger space).
Define $T' = \Phi_\# T \in \mathbf{M}_1(X') \subset \mathbf{M}_1(E)$. Since $E$ is a Banach space of dimension at least two, by the results of the previous section, there exists a finite Borel measure $\mu$ on $\mathbb D([0,1],E)$, concentrated on $\Theta_{\SBV}(X') \cap \mathbb D([0,1],X')$ satisfying:
\begin{align}
    T'(\omega) & = \int_{ \mathbb D([0,1],E)} \llbracket\zeta\rrbracket^a(\omega) \, d\mu(\zeta), \\
    \|T'\|(B) & = \int_{ \mathbb D([0,1],E)} \|\llbracket\zeta\rrbracket^a\|(B) \, d\mu(\zeta). \label{eq:iso_mass}
\end{align}
Now, let $\Psi = \Phi^{-1} : X' \to X$ be the inverse isometry. 
Recalling that $\mu$ is concentrated on $\mathbb D([0,1],X') \cap \Theta_{\SBV}(X')$ we get
\begin{align*}
    T(f,\pi) & = T'(f\circ \Psi,\pi\circ \Psi) \\
       & = \int_{ \mathbb D([0,1],X')} \llbracket\zeta\rrbracket^a(f \circ \Psi,\pi\circ \Psi) \, d\mu(\zeta) 
         = \int_{ \mathbb D([0,1],X')} \llbracket\Psi(\zeta)\rrbracket^a(f,\pi) \, d\mu(\zeta).
\end{align*}
Let $\eta = \Psi_\# \mu$. By construction and the stability of Càdlàg spaces under Lipschitz isometries, $\eta$ is a finite Borel measure on $\mathbb D([0,1],X)$, which by the previous identity yields the representation:
\[
    T(f,\pi) = \int_{\mathbb D([0,1],X)} \llbracket u\rrbracket^a(f,\pi) \, d\eta(u).
\]
Moreover, by construction $\eta$ is concentrated on the injective curves of $\Theta_{\SBV}(X)$ with constant speed. 

The mass identity:
\[
    \|T\| = \int_{\Theta_{\SBV}(X)} \|\lpr{u}^a\| \, d\eta(u)
\]
follows directly from the identity $\Psi_\# \|T'\| = \|T\|$ and~\eqref{eq:iso_mass}.

For the boundary estimate, recall that for any $\varepsilon > 0$, the representation in the Banach space $E$ satisfies:
\[
    \int_{\Theta_{\SBV}(E)} |D^j \zeta| \, d\mu(\zeta) \le \|\partial T'\|_{\KR} + \varepsilon.
\]
Since $\Psi$ is an isometry, it preserves the jump lengths, that is, $|D^j \Psi(\zeta)| \equiv |D^j \zeta|$ as measures. Also, $\|\partial T'\|_{\KR} = \|\partial T\|_{\KR}$ because push-forwards by isometries also map Arens--Eells spaces isometrically. Thus, pushing forward $\mu$ to $\eta$ yields:
\[
    \int_{\Theta_{\SBV}(X)} |D^j u| \, d\eta(u) = \int_{\Theta_{\SBV}(E)} |D^j \Psi(\zeta)| \, d\mu(\zeta) \le \|\partial T\|_{\KR} + \varepsilon.
\]
The full statement of Theorem~\ref{thm:rep} is therefore proven.

\subsection{Proof of Corollary~\ref{cor:fragment}}

By Theorem~\ref{thm:rep}, there exists a finite Borel measure $\eta$ on $\mathbb D([0,1],X)$, concentrated on $\Theta_{\SBV}(X)$ and satisfying
\begin{align}
    T(\omega) & = \int_{ \mathbb D([0,1],E)} \llbracket u\rrbracket^a(\omega) \, d\eta(u), \\
    \|T\|(B) & = \int_{ \mathbb D([0,1],E)} \|\llbracket u\rrbracket^a\|(B) \, d\eta(u), \label{eq:iso_mass}
\end{align}
for any $\omega \in \mathscr D^1(X)$ and $B \in \mathcal B(X)$. 

Recalling from Lemma~\ref{lem:semicont} that the functions from $\Theta_{\SBV}(X)$  into $\R$ defined by $u \mapsto \ell(u)$ and $u \mapsto |D^j u|(I)$ are Borel, and hence $\eta$-measurable. It then follows from Lusin's theorem and Lemma~\ref{lem:YES} that the fragment map $\Gamma : \Theta_{\SBV}(X) \to \Frag(X) : u \mapsto \gamma_u$ introduced in Section~\ref{sec:fragments} is $\eta$-measurable. This allows us to define the push-forward $\mu \coloneqq \Gamma_\# \eta$, which defines a Borel measure on $\Frag(X)$. 

Since the functions $\gamma \mapsto \lpr{\gamma}(\omega)$ and $\gamma \mapsto \ell(\gamma)$, from $\Frag(X)$ into $\R$, are Borel (see~\cite[Sec.~6]{Schioppa_2016}) and $\mu$-summable, we can integrate them against $\mu$. Finally, since $\lpr{u}^a = \lpr{\gamma_u}$ as currents, it follows from the definition of push-forward that
\begin{align*}
    T(\omega) & = \int_{ \mathbb D([0,1],E)} \llbracket u\rrbracket^a(\omega) \, d\eta(u) \\
    & = \int_{ \mathbb D([0,1],E)} \llbracket \gamma_u\rrbracket(\omega) \, d\eta(u) \\
    & = \int_{\Frag(X)} \llbracket \gamma \rrbracket(\omega) \, d\mu(\gamma).
\end{align*}
The superposition for the mass measures and cancellation-free property follow analogously using that $\ell(u) = \ell(\gamma_u)$ and arguing as in the previous proofs. \qed


\section{Proof of the $1$-FCC characterization}\label{sec:char}
This section is dedicated to the proof of Theorem~\ref{thm:main1}), which states that a metric space $X$ is curve rectifiable if and only if every metric current $T \in \Mbf_1(X)$ can be approximated in the mass norm by normal currents.

Before diving into the proof, let us recall the definition of curve rectifiability: A metric space $X$ is said to be \emph{curve rectifiable} if every $1$-rectifiable set can be covered, up to an $\mathcal{H}^1$-null set, by countably many images of Lipschitz curves.

We divide the proof into two parts:

\subsection*{Sufficiency}
First, we show that if every metric current $T \in \Mbf_1(X)$ can be approximated in the mass norm by normal currents, then $X$ must be curve rectifiable.

Since $1$-rectifiable sets can be expressed as the union of at most countably many images of fragments, it suffices to show that the image of any fragment can be covered by at most countably many Lipschitz curves on $X$. More precisely, we want to prove the following: if $\gamma \in \Frag(X)$ is a nontrivial fragment and $\Gamma \coloneqq \gamma([0,1]) \subset X$, then there exists a countable set $\mathcal{J} \subset \Lip([0,1],X)$ such that
\begin{equation}\label{eq:main_1}
	\mathcal{H}^1 \left(\Gamma \smallsetminus \bigcup_{\theta \in \mathcal{J}} \Gamma_\theta \right) = 0 \qquad \text{where}\quad
	\Gamma_\theta := \Gamma \cap \theta([0,1]).
\end{equation}
By assumption, we can find a sequence of normal currents $(N_j) \subset \Nbf_1(X)$ satisfying
\begin{equation}\label{convNj}
\Mbf(\lpr{\gamma} - N_j) \to 0.
\end{equation}
Recalling Paolini and Stepanov's decomposition of normal metric currents \cite{paolini2012decomposition,paolini2012decomposition2}, for each $j \in \mathbb{N}$, there exists a finite Borel measure $\eta_j$ on $C([0,1],X)$, concentrated on a family $\mathcal{J}_j$ of $1$-Lipschitz curves in $\Lip_1([0,1],X)$ with constant speed one such that
\begin{align}\label{eq:no_mass}
	N_j(\omega) & = \int_{C([0,1],X)} \lpr{\theta}(\omega) \, d\eta_j(\theta) \qquad \text{for all $\omega \in \mathscr{D}^1(X)$} \nonumber \\
	\|N_j\| & = \int_{C([0,1],X)} \|\lpr{\theta}\| \, d\eta_j(\theta)
\end{align}
We observe that, by the triangular inequality, we have:
\[
	\Mbf(\lpr{\gamma} - N_j) \ge \Mbf(\lpr{\gamma} \mres B - N_j \mres B) \ge |\Mbf(\lpr{\gamma} \mres B) - \|N_j\|(B)| ,
\]
where $B$ is an arbitrary Borel set in $X$.
If we additionally assume that $B\subset \Gamma$, then $\Mbf(\lpr{\gamma} \mres B) =\|\lpr{\gamma}\|(B) = \mathcal{H}^1 (B)$ due to the injectivity of $\gamma$ and, by exploiting the mass convergence \eqref{convNj}, we infer that
\[ \|N_j\|(B) \to \mathcal{H}^1(B) \qquad \text{for all $B\subset \Gamma$}\]
As a consequence, we obtain the following key implication that will be needed later:
\begin{equation}\label{eq:missing}
B \subset \Gamma\ ,\ \mathcal{H}^1(B) >0 \quad \Rightarrow \quad \liminf_{j\to \infty} \|N_j\|(B) > 0.
\end{equation}

Next, for every $j \in \mathbb{N}$, we consider the following length maximization problem:
\[
	 \alpha_j \coloneqq \sup \set{H(\mathcal{F})}{\mathcal{F}\ \text{countable} \ \subset \mathcal{J}_j \; }, \qquad H(\mathcal{F}) \coloneqq \mathcal{H}^1\left(\bigcup_{\theta\in \mathcal{F}} \Gamma_\theta \right).
\]

Clearly $\alpha_j$ is finite, because it is always majorized by $\mathcal{H}^1(\Gamma)$. In fact, it is a maximum (i.e., it is attained). Indeed, let $(\mathcal{F}_k)_{k \in \mathbb{N}} \subset \mathcal{J}_j$ be a maximizing sequence such that $H(\mathcal{F}_k) \to \alpha_j$. Now, simply take $\mathcal{F}^j \coloneqq \bigcup_{k} \mathcal{F}_k$, which is again a countable subset of $\mathcal{J}_j$. Then, since the functional $H$ is non-decreasing with respect to set inclusion of countable families of curves, it holds
\[
	H(\mathcal{F}^j) \ge \lim_{k \to \infty} H(\mathcal{F}_k) = \alpha_j.
\]
This proves that  the supremum above is indeed a maximum. 

 Now, consider the family $\mathcal{J} \coloneqq \bigcup_j \mathcal{F}^j$, and notice that $\mathcal{J}$ is countable by construction. By the monotonicity property of the functional $H$ and since $\bigcup_{\mathcal{J}} \Gamma_\theta \subset \Gamma$, for every $j$ it holds:
\[
\alpha_j = H(\mathcal{F}^j) \le H(\mathcal{J}) = \mathcal{H}^1\left(\bigcup_{\theta\in \mathcal{J}} \Gamma_\theta\right) \le \mathcal{H}^1(\Gamma) ,
\]
whence, by the sub-additivity of outer measures:
\[
\mathcal{H}^1\left(\Gamma \smallsetminus \bigcup_{\mathcal{J}} \Gamma_\theta \right) \le \mathcal{H}^1(\Gamma)- \alpha_j\ .
\]
We can conclude with the desired equality \eqref{eq:main_1} (proving the sufficiency) once we show the following claim:
\begin{equation}\label{Claim_alphaj}
\alpha_j \to \mathcal{H}^1(\Gamma) \qquad \text{as $j\to \infty$}\ .
\end{equation}
Suppose that \eqref{Claim_alphaj} is false. Then there exists $\eps > 0$ such that
$\eps + \limsup_{j \to \infty} \alpha_j \le \mathcal{H}^1(\Gamma)$.
In particular, for sufficiently large $j$, there exists a Borel set $B_j \subset X$ such that
	\[
		B_j \subset \Gamma \smallsetminus \bigcup_{\theta\in \mathcal{F}^j} \Gamma_\theta, \qquad \mathcal{H}^1(B_j) \ge \eps,
	\]
where $\mathcal{F}^j$ is a maximizer of $H$ within the set of countable subsets of $\mathcal{J}_j$.

However, from \eqref{eq:missing}, for sufficiently large $j$, it holds $\|N_j\|(B_j) > 0$. In particular, by \eqref{eq:no_mass} we have
\[
	0 < \|N_j\|(B_j) = \int \|\lpr{\theta}\|(B_j) \,d \eta_j(\theta) = \int \mathcal{H}^1(\theta \cap B_j) \,d \eta_j(\theta).
\]

Therefore, there must exist at least one curve $\theta_j \in \mathcal{J}_j$ such that $\mathcal{H}^1(\theta_j \cap B_j)>0$. Then, consider $\mathcal{G}^j \coloneqq \mathcal{F}^j \cup \{\theta_j\}$. Observe that
\[
	\mathcal{H}^1 \left(\bigcup_{ \theta \in \mathcal{G}^j} \Gamma_\theta \right) \ge \mathcal{H}^1 \left(\bigcup_{ \theta \in \mathcal{F}^j} \Gamma_\theta \right) + \mathcal{H}^1(\theta_j \cap B_j) > \mathcal{H}^1 \left(\bigcup_{ \theta \in \mathcal{F}^j} \Gamma_\theta \right).
\]
This contradicts the maximality of $\mathcal{F}^j$ for $j$ sufficiently large. Hence, \eqref{Claim_alphaj} is true. \medskip

\subsection*{Necessity}
We show that if $X$ is curve rectifiable, then every metric current $T \in \Mbf_1(X)$ can be approximated in the mass norm by normal currents.

\subsubsection*{1. Approximation of fragments by normal currents}
The first objective is to demonstrate that any fragment within a curve rectifiable space can be approximated in the mass norm by normal currents.

To this end, consider an arbitrary fragment $\gamma \in \Frag(X)$. By the definition of a curve rectifiable metric space, there exists a sequence $(\theta_j)_{j =1}^\infty \subset \Lip_1([0,1],X)$ such that
\begin{equation}\label{eq:cr}
	\im \gamma \subset \left(\bigcup_{j=1}^\infty \im \theta_j \right) \cup N, \qquad \Hcal^1(N) = 0.
\end{equation}
This sequence can be refined inductively so that the images of distinct curves are essentially disjoint:
\begin{equation}\label{eq:no_over}
i \neq j \quad \Rightarrow \quad \Hcal^1(\im \theta_i \cap \im \theta_j) = 0.
\end{equation}
The idea is to approximate each "piece" $S_j \coloneqq \lpr{\theta_j} \mres \im \gamma$ separately with normal currents $N_{j,\eps} \in \Nbf_1(X)$ for each $\eps > 0$. The critical aspect, enabled by \eqref{eq:no_over}, is that the length of $\gamma$ will not be over-counted when combining the approximations of these pieces. Moreover, in order to ensure the approximation by normal currents, we must manage the potential for infinite mass boundaries arising from the full series, which requires a controlled approximation by partial sums. 

Let us then fix $j \in \mathbb{N}$ and focus on the approximation of $S_j = \lpr{\theta_j} \mres \im \gamma$. Since we are working with Lipschitz parameterizations, the underlying idea is similar to the problem of approximating a compact set $C \subset [0,1]$ in length measure by a sequence of finitely many intervals. The following construction mimics the ideas from~\cite[Prop. A.2]{Merlo}:

Fix $\eps > 0$. Let $\Gamma \coloneqq \im \gamma$ and consider the pre-image $K_j \coloneqq \theta_j^{-1}(\Gamma)$, which is a compact set in $[0,1]$. Since we do not want to add unnecessary length in the approximation, we consider $I_j \coloneqq [a_j,b_j] \subset [0,1]$, the smallest closed interval containing $K_j$ (whose endpoints are the infimum and supremum of $K_j$). Since $I_j \smallsetminus K_j \subset (0,1)$ is open, it can be expressed as a union of a sequence $(I_{j,h})$ of disjoint open intervals (possibly empty) of $(0,1)$. Therefore, there exists a sufficiently large number $h_j \in \mathbb{N}$ such that
\begin{equation}\label{eq:tail}
	\sum_{h \ge h_j} |I_{j,h}| < 2^{-(j+1)}\eps.
\end{equation}
Now, consider the compact set
\[
	C_j \coloneqq I_j \smallsetminus \bigcup_{h < h_j} I_{j,h}.
\]
By construction $K_j \subset C_j$, and $C_j$ is an $\eps$-approximation from the outside of $K_j$ in terms of length, consisting of finitely many closed intervals. We may now define a normal current approximation of $S_j$, by letting
\[
	N_{j,\eps} \coloneqq \lpr{\theta_j} \mres C_j = (\theta_j)_\# \lpr{C_j} \in \mathcal{R}_1(X).
\]
As the sum of finitely many integral currents associated with Lipschitz curves, it follows that $N_{j,\eps} \in \Nbf_1(X)$. Moreover, by construction $S_j = (\theta_j)_\# \lpr{K_j}$ and hence
\begin{equation}\label{eq:approx_eps}
	\Mbf(S_j - N_{j,\eps}) \le \Mbf( (\theta_j)_\# \lpr{I_j \smallsetminus C_j}) \le \Hcal^1(I_j \smallsetminus C_j) \stackrel{\eqref{eq:tail}}{<} 2^{-(j+1)}\eps.
\end{equation}
Observe that
\[
	\sum_{j = 1}^\infty (\Mbf(N_{j,\eps}) + \Mbf(S_{j})) \le \sum_{j = 1}^\infty (2\Mbf(S_j) + 2^{-(j+1)}\eps) = 2\mathcal{H}^1(\Gamma) + \eps.
\]
From this, we also deduce the existence of a sufficiently large natural number $h_\eps$ such that
\begin{equation}\label{eq:most}
	\sum_{j > h_\eps} \Mbf(N_{j,\eps}) + \Mbf(S_{j}) < 2^{-1}\eps.
\end{equation}

We can now "glue" most of our approximations together, by defining
\[
	N_\eps \coloneqq \sum_{j = 1}^{h_\eps} N_{j,\eps} \in \Nbf_1(X) \cap \mathcal{R}_1(X).
\]
Thus, we have:
\[
	\Mbf(\lpr{\gamma} - N_\eps) \stackrel{\eqref{eq:no_over},\eqref{eq:most}} \le \sum_{j = 1}^{h_\eps} \Mbf(S_j - N_{j,\eps}) + 2^{-1}\eps \stackrel{\eqref{eq:approx_eps}}\le \eps.
\]
Letting $\eps \to 0$ yields the desired approximation.\qed~\\

\subsubsection*{2. A superposition of normal approximations.}
In light of Corollary~\ref{cor:fragment}, every metric current $T \in \Mbf_1(X)$ can be represented, regardless of the connectedness assumptions on $X$, as a superposition of currents associated with fragments such that there are no mass cancellations. Specifically, we demonstrated the existence of a finite Borel measure $\eta$ on $\Frag(X)$ such that
\begin{align}
	T(\omega) & = \int_{\Frag(X)} \lpr{\gamma})(\omega) \, d\eta(\gamma) \\
	\|T\|(B) & = \int_{\Frag(X)} \mathcal{H}^1 (\im \gamma \cap B) \, d\eta(\gamma). \label{eq:nomas}
\end{align}

In the previous step we have shown that for any $\eps$ and any fragment $\gamma \in \Frag(X)$ there exists (at least one) normal current $N_{\gamma,\eps} \in \Nbf_1(X)$ satisfying $\Mbf(\lpr{\gamma} - N_{\gamma,\eps}) < \eps$. From a quantitative viewpoint, the superposition
\begin{equation}\label{eq:candidate}
	R_\eps(\omega) \coloneqq \int_{\Frag(X)} N_{\gamma,\eps}(\omega) \, d\eta(\gamma)
\end{equation}
of such approximations is a natural candidate for the mass approximation of $T$. Indeed, supposing that $R_\eps$ is well-defined \emph{and} defines a normal current, it is straightforward to verify, using \eqref{eq:nomas}, that
\[
	\Mbf(T - R_\eps) \le \int_{\Frag(X)} \Mbf(\lpr{\gamma} - N_{\gamma,\eps}) \, d\eta(\gamma) \le \eta(\Frag(X)) \eps.
\]

However, defining \eqref{eq:candidate} introduces two key challenges:\medskip
\begin{itemize}
\item[(a)] The first is the \textbf{well-definedness} of $R_\varepsilon$ as a metric current. For the integral in \eqref{eq:candidate} to be well-defined, the map $\gamma \mapsto N_{\gamma,\varepsilon}$ must be $\eta$-measurable, and $\gamma \mapsto \Mbf(N_{\gamma,\eps})$ must be $\eta$-summable. Summability follows directly from \eqref{eq:nomas}, the finiteness of $\eta$ and the bound $\Mbf(\lpr{\gamma} - N_{\gamma,\eps}) < \eps$. While our construction in Step~1 might be intuitively measurable, rigorously establishing this directly can be cumbersome. To address this, we will appeal to \textbf{Von Neumann's measurable selection theorem}. \\

\item[(b)] The second is whether $R_\varepsilon$ defines a \textbf{normal current}. The mass of the boundary, $\Mbf(\partial N_{\gamma,\varepsilon})$, generally increases with the irregularity of the fragment $\gamma$. Consequently, $R_\varepsilon$ as defined in \eqref{eq:candidate} will not necessarily be a normal current. This is consistent with the fact that $\Nbf_1(X)$ is not a closed subset of $\Mbf_1(X)$. Our solution involves constructing a monotone sequence of Borel sets $\mathcal{X}_j \subset \Frag(X)$ that converge to a full $\eta$-measurable set. On these sets, we will ensure that $\Mbf(\partial N_{\gamma,\varepsilon}) \le j$ for $\eta$-almost every $\gamma$, and then we will take the limit as $j \to \infty$.
\end{itemize}\medskip

With these considerations in mind, we can now proceed to develop the rigorous mathematical arguments:\\

\subsubsection*{2.(a) A measurable selection of normal approximations.}
The primary goal of this section is to construct an \textbf{$\eta$-measurable map} $\Phi_\eps : \Frag(X) \to \Nbf_1(X)$ using a measurable selection theorem. This map must satisfy the following condition:
\begin{equation}\label{eq:eps_meas}
	\Mbf(\lpr{\gamma} - \Phi_\eps(\gamma)) \le \eps.
\end{equation}
\textbf{Notation.} For simplicity of notation, and to avoid carrying a heavy $\eps$-dependence throughout the mathematical objects, we suppress $\eps$ from the notation of $\Phi$ and other related sets and functions where the dependence on $\eps$ is implicit from the context of this section.

To begin, we can assume that \textbf{$X$ is a separable space}. This assumption is crucial because it ensures that:
\begin{itemize}
    \item The \textbf{space of fragments $\Frag(X)$}, equipped with the Hausdorff metric of their graphs (as subsets of $\Kcal([0,1]\times X)$), is a \textbf{Polish space}.
    \item The space $C([0,1],X)$ is a \textbf{Polish space} under the topology induced by uniform convergence. Consequently, its subspace $\Lip_m(X) \coloneqq \set{\theta \in \Lip([0,1],X)}{\Lip(\theta) \le m}$ is also a Polish space for all $m \in \mathbb{N}$, when endowed with the uniform distance between functions.
\end{itemize}

Now, let $\mathcal{X} \coloneqq \Frag(X)$. Consider the \textbf{product spaces $\mathcal{Z}_m \coloneqq \Lip_m(X)^m$}. This space is also Polish, inheriting its topology from the uniform metric on each coordinate (which induces the product topology). Next, we define a map $\Delta_m: \mathcal{X} \times \mathcal{Z}_m \to \mathbb{R}$ given by
\begin{align*}
	\Delta_m(\gamma, (\theta_1,\dots,\theta_m)) & \coloneqq \Mbf( \lpr{\gamma} - \sum_{k = 1}^m \lpr{\theta_k})
\end{align*}
These maps are \textbf{Borel measurable}. This follows from the Borel measurability of operations on currents (e.g., sum, pushforward, mass) and the fact that an integral current associated with a Lipschitz curve is Borel measurable.

From the preceding discussion, it follows that \textbf{$\mathcal{Y}_m \coloneqq \Delta_m^{-1}([0,\eps)) \subset \mathcal{X} \times \mathcal{Z}_m$} is a standard \textbf{Borel set}. As a result, both $\mathcal{Y}_m$ and its projection,
\[
\mathcal{X}_m \coloneqq \pi_{\mathcal{X}}(\mathcal{Y}_m) = \set{\gamma \in \mathcal{X}}{\exists (\theta_1,\dots,\theta_m) \in \mathcal{Z}_m, \Delta_m(\gamma, (\theta_1,\dots,\theta_m)) < \eps},
\]
are \textbf{analytic sets} in $\mathcal{X} \times \mathcal{Z}_m$ and $\mathcal{X}$, respectively. Given that $\eta$ is a finite Borel measure (see \cite[Thm. 4.3.1]{Borel_Sets_Book}), $\mathcal{X}_m$ is \textbf{$\eta$-measurable}. 

Since the assumptions of Von Neumann's measurable selection theorem (see \cite[Thm. 5.5.2]{Borel_Sets_Book}) are satisfied, there exists an \textbf{$\eta$-measurable section $\Psi_m : \mathcal{X}_m \to \mathcal{Z}_m$} that satisfies:



\[
	\Mbf(\lpr{\gamma} - \sum_{k=1}^m \lpr{(\Psi_m(\gamma))_k}) \le \eps \qquad \text{for $\eta$-almost every $\gamma \in \mathcal{X}_m$.}
\]
To be precise, $\Psi_m(\gamma)$ is a tuple $(\theta_1, \dots, \theta_m)$. We then define $\Phi_m(\gamma) \coloneqq \sum_{k=1}^m \lpr{(\Psi_m(\gamma))_k}$, which also defines a measurable map. 

Our construction in Step 1 ensures that any $\gamma$ belongs to $\mathcal{X}_m$ for some sufficiently large $m \in \mathbb{N}$. Specifically, since a fragment can be approximated arbitrarily well by a finite sum of integral Lipschitz curve currents (as shown in Step 1), and since $\Lip_m(X)$ contains curves with Lipschitz constant up to $m$, we have that $\mathcal{X}_m \subseteq \mathcal{X}_{m+1}$ for all $m \in \Nbb$. Thus, we can write:
\[
	\mathcal{X} = \bigcup_{m \in \mathbb{N}} \mathcal{X}_m = \bigsqcup_{m \in \mathbb{N}} \mathcal{W}_m, \quad \text{where} \quad \mathcal{W}_m \coloneqq \mathcal{X}_m \smallsetminus \bigcup_{0\le k < m} \mathcal{X}_k.
\]
Since each $\mathcal{W}_m$ is an analytic set (and thus $\eta$-measurable), the \textbf{$\eta$-measurable section $\Phi_\eps : \mathcal{X} \to \Nbf_1(X)$} defined by
\[
	 \Phi_\eps(\gamma) \coloneqq \sum_{m \in \mathbb{N}} \mathbf{1}_{\mathcal{W}_m}(\gamma) \Phi_m(\gamma),
\]
satisfies condition \eqref{eq:eps_meas}, thereby proving the desired assertion.\\

\subsubsection*{2.(b) Construction of the normal currents.}
Let $\eps > 0$. Recalling the notation introduced in the previous step, we have $\mathcal{X} = \Frag(X)$ and
\[
	\mathcal{X}_m = \set{\gamma \in \mathcal{X}}{\exists (\theta_1,\dots,\theta_m) \in \Lip_m([0,1],X)^m, \Mbf(\lpr{\gamma} - \sum_{k=1}^m \lpr{\theta_k}) < \eps}.
\]
Notice that the sequence $(\mathcal{V}_k)_{k = 1}^\infty$ of $\eta$-measurable subsets of $\mathcal{X}$ given by
\[
	\mathcal{V}_k \coloneqq \bigcup_{m \le k} \mathcal{X}_m, \qquad k \in \mathbb{N},
\]
is monotone with respect to set-inclusion. Moreover, by construction $\bigcup_k \mathcal{V}_k = \bigcup_m \mathcal{X}_m = \mathcal{X}$. Hence, from the dominated convergence theorem we get
$\mathbf{1}_{\mathcal{V}_k} \to 1$ in $L^1(\eta)$,
where $\mathbf{1}_A$ denotes the indicator function of an $\eta$-measurable set $A \subset \mathcal{X}$. Lastly, since for every $\gamma \in \mathcal{V}_k$, the current $\Phi_\eps(\gamma)$ consists of at most $k$ integral currents associated with Lipschitz curves, it follows that
\begin{equation}\label{eq:truncation}
	\Mbf(\partial \Phi_\eps(\gamma)) \le 2k \qquad \text{for all $\gamma \in \mathcal{V}_k$}.
\end{equation}

Let $\eta_\eps \coloneqq (\Phi_\eps)_\# \eta$, where $\Phi_\eps$ is the selection from the previous section satisfying \eqref{eq:eps_meas}. By construction, $\eta_\eps$ is a finite Borel measure on $\Mbf_1(X)$, concentrated on $\Nbf_1(X)$.

For $k \in \mathbb{N}_0$, we define $\mathcal Z = \bigcup_{m \in \Nbb} \mathcal Z_m$ and set
\begin{align*}
	N_{\eps,k}(\omega) \coloneqq \begin{cases} \displaystyle
	\int_{\mathcal Z} N(\omega) \, d\eta_\eps(N) & \text{if $k = 0$}\\[1em]
	\displaystyle \int_{\Phi_\eps(\mathcal{V}_k)} N(\omega) \, d\eta_\eps(N) & \text{if $k > 0$}
	\end{cases}
\end{align*}
Notice that $N_{\eps,k}$ is well-defined, and satisfies all the properties of currents because each $\mathcal Z_m$ inherits the measurability properties of $\Frag(X)$ and $\Lip([0,1],X)$ as parametrization of metric currents. The only axiom that requires perhaps an explanation is the continuity property: if $\pi_j \toweakstar \pi$ in $X$, then $N(f,\pi_j) \to N(f,\pi_j)$ for all $N \in \mathcal Z$ and the desired result then follows from the dominated convergence theorem. 

In particular, $N_{\eps,k} \in \Mbf_1(X)$ for every $k \in \Nbb_0$.  If, moreover, $k \in \mathbb{N}$, then for every $f \in \Lip(X)$ it holds
\[
	|N_{\eps,k}(1,f)| = \left|\int_{\mathcal{V}_k} \Phi_\eps(\gamma)(1,f) \, d\eta(\gamma)\right| \le \int_{\mathcal{V}_k} \Mbf(\partial \Phi_\eps(\gamma)) \, d\eta(\gamma) \stackrel{\eqref{eq:truncation}}{\le} 2k \eta(\mathcal{V}_k) \le 2k \eta(\mathcal{X}).
\]
Taking the supremum over $f$ in the left-hand side we hence obtain $\Mbf(\partial N_{\eps,k}) \le 2k \eta(\mathcal{X}) < \infty$, which demonstrates that $N_{\eps,k} \in \Nbf_1(X)$ for every $k \in \mathbb{N}$.\\

\subsubsection*{3. Conclusion.}
Let $k \in \mathbb{N}$. From the triangle inequality we get
\begin{align*}
	\Mbf(N_{\eps,0} - N_{\eps,k}) & \le \int_{\Phi_\eps(\mathcal{X} \smallsetminus \mathcal{V}_k)} \Mbf(N) \, d\eta_\eps(N) \\
	& = \int_{\mathcal{X} \smallsetminus \mathcal{V}_k} \Mbf(\Phi_\eps(\gamma)) \, d\eta(\gamma) \\
	& \le \int_{\mathcal{X} \smallsetminus \mathcal{V}_k} (\Mbf(\lpr{\gamma}) + \eps) \, d\eta(\gamma).
\end{align*}
Recalling that $\eta$ is a finite measure, that $\int \ell(\gamma) \, d\eta(\gamma) = \Mbf(T) < \infty$ and that $\mathbf{1}_{\mathcal{V}_k} \to 1$ in $L^1(\eta)$, we deduce that $\Mbf(N_{\eps,0} - N_{\eps,k}) \to 0$ as $k \to \infty$.
In particular, there exists $k_\eps \in \mathbb{N}$ large enough so that
\begin{equation}\label{eq:almost}
\Mbf(N_{\eps,0} - N_{\eps}) \le \eps, \qquad N_\eps \coloneqq N_{\eps,k_\eps} \in \Nbf_1(X).
\end{equation}

Now, appealing to the properties of the push-forward and the definition of $\eta_\eps$ we obtain
\begin{align*}
	\Mbf(T - N_{\eps,0}) = \Mbf\left(\int_{\mathcal{X}} (\lpr{\gamma} - \Phi_\eps(\gamma)) \, d\eta(\gamma)\right) \le \int_{\mathcal{X}} \Mbf(\lpr{\gamma} - \Phi_\eps(\gamma)) \, d\eta(\gamma) \le \eps \eta(\mathcal{X})
\end{align*}
for all $\omega \in \mathscr{D}^1(X)$, which proves
\begin{equation}\label{eq:almost2}
\Mbf(T - N_{\eps,0}) \le \eps \eta(\mathcal{X})
\end{equation}
Gathering \eqref{eq:almost}-\eqref{eq:almost2} gives
\[
	\Mbf(T - N_\eps) \le \Mbf(T - N_{\eps,0}) + \Mbf(N_{\eps,0} - N_\eps) \le \eps (\eta(\mathcal{X}) + 1).
\]

We have thus proved the following: for a given metric current $T \in \Mbf_1(X)$, there exists a family of normal currents $\{N_\eps\}_{\eps > 0} \subset \Nbf_1(X)$ satisfying
\[
	\lim_{\eps \to 0} \Mbf(T - N_\eps) = 0.
\]
This finishes the proof.\qed

  \printbibliography 
  \pagenumbering{gobble}
  
  \printnoidxglossary[sort=use,type=symbol,style=symbolsstyle,title=List of symbols]


\appendix

\section{Measurability results}
\label{sec:appendix_measurability}
\small{
In this appendix, we address several technical measurability questions that underpin the theoretical framework of this paper. While the results presented here are largely expected, providing their detailed proofs ensures completeness and allows the main arguments of the paper to remain focused on core concepts. 

\begin{lemma}\label{lem:Haus_implies_L1}
Let $(K_j)_{j \in \mathbb{N}} \subset \mathcal K([0,1])$ be a sequence of closed sets converging to a closed set $K \subset [0,1]$ in the Hausdorff distance, i.e., $d_H(K_j,K) \to 0$ as $j \to \infty$. If, in addition, their Lebesgue measures converge, $\mathscr L^1(K_j) \to \mathscr L^1(K)$, then their characteristic functions converge in $L^1([0,1])$:
\[
\mathbf 1_{K_j} \longrightarrow \mathbf 1_{K} \quad \text{as } j \to \infty \text{ in } L^1([0,1]).
\]
\end{lemma}

\begin{remark}
    The  proof given below can be easily extended to the case of  a sequence of closed subsets $(K_j)$ in $\R^d$.
\end{remark} 

\begin{proof} Actually we will show that $\| \mathbf 1_{K_j}-\mathbf 1_{K}\|_{L^2(0,1)} \to 0$. By the assumption of convergence  $\mathscr L^1(K_j) \to \mathscr L^1(K)$, we have that $\| \mathbf 1_{K_j}\|_{L^2(0,1)} \to \|\mathbf 1_{K}\|_{L^2(0,1)}$.
Therore it is enough to prove that $\mathbf 1_{K_j} $ converges weakly to $\mathbf 1_{K}$ in $L^2(0,1)$.  
As $\mathbf 1_{K_j}$ is bounded in $L^2(0,1))$, we can assume, after extracting a subsequence,  that  $\mathbf 1_{K_j}$  converges weakly  to a function $\theta \in L^2([0,1])$. Then $\theta$ satisfies $\theta\in  [0,1]$ a.e. while 
$$\int_{[0,1]}  \theta \, dx = \lim_{j\to\infty} \int_{[0,1]} \mathbf 1_{K_j} = \mathscr L^1(K) .$$
 Let $x\notin K$ and a closed interval $B$ centered at $x$ such that $K\cap B$ is empty.
Then by the Hausdorff convergence  $K_j\to K$, we have that $K_j\cap B$ is empty for large $j$, whence 
$\int_B \theta \, dx =\lim_j \mathscr L^1(K_j\cap B)=0$.  It follows that  $\theta= 0$ a.e. on $K^c$.
Therefore, since $\theta\le 1$, the equality  $\int_{[0,1]}  \theta \, dx= \int_{K}  \theta \, dx  = \mathscr L^1(K)$ implies that   $\theta=\mathbf 1_K$ - a.e.. We conclude  that $\mathbf 1_K$ is the unique weak cluster point, thus the weak limit of  $(\mathbf 1_{K_j})$ as we claimed.
\end{proof}

The following standard result is included with a proof for completeness, as we could not find a readily available demonstration in the literature:

\begin{lemma}\label{lem:Hausdorff}
Let $B \subset X$ be a closed set. The inverse image map $L: C([0,1],X) \to \mathcal{K}([0,1])$ defined by $L(\theta) = \theta^{-1}(B)$ is Borel measurable. Here, $\mathcal{K}([0,1])$ denotes the space of non-empty closed subsets of $[0,1]$ equipped with the Hausdorff distance $d_H$. 
\end{lemma}

\begin{remark}
    Notice that the continuity  of the map $L$ fails in general. Indeed, consider for instance the case where $X=\R^2$, $B=\{(0,0)\}$
and  $\theta_j(t)= j^{-1} (\cos 2\pi t, \sin 2\pi t)$. Then  $L(\theta_j)$ is empty while $\theta_j$ converges to a constant map such that $L(\theta)=[0,1]$.
\end{remark}

\begin{proof} In  this proof,   for every $C \in\mathcal K([0,1])$ and  
 $\delta>0$, we denote   $C^\delta$ the enlarged set $C^\delta:= \{ s\in [0,1] : \dist(s,C)< \delta\}$ which is relatively open in $[0,1]$.
The Borel  tribe of $\mathcal{K}([0,1])$ is generated by the open balls. The inverse image by the map $L$ of a ball centered at $K$ and of radius $\delta$ can written as  $A_K^\delta \cup B_K^\delta$ where:
\[   A_K^\delta =\set{\theta \in C([0,1], X)} { \theta^{-1}(B) \subset K^{\delta}},\quad 
 B_K^\delta =\set{\theta \in C([0,1], X)} { K\subset (\theta^{-1}(B))^\delta }.
 \]
Therefore we are done if we show that both $A_K^\delta$ and $ B_K^\delta$ are Borel subsets of $C([0,1], X)$.  

First we claim that $ A_K^\delta$ is an open subset. If it is not the case, there exists $\theta \in  A_K^\delta$ and a sequence $(\theta_j)$
in  the complement of $ A_K^\delta$  such that $\theta_j \to \theta$ uniformly.
This means that , for every $j$, there exits $t_j\in [0,1]\smallsetminus K^{\delta}$ such that $\theta_j(t_j)\in B$.
As $[0,1]\smallsetminus K^\delta$ is compact, after passing to a subsequence, we can assume that $t_j\to t$ where $t\notin K^\delta$. Then, by the uniform convergence of $\theta_j$ and the fact that $B$ is closed, we deduce that $\theta(t)\in B$. This would imply that
$\theta\notin  A_K^\delta$ a contradiction. 

Next we claim that $ B_K^\delta$ is a countable union of closed subsets thus a Borel of $C([0,1],X)$. Indeed we observe that, for every monotone sequence $\delta_n \nearrow \delta$, it holds the equality $B_K^\delta = \bigcup  B_K^{\delta_n}$ where
\[ C_K^\delta :=\set{\theta \in C([0,1], X)} { K\subset \overline {(\theta^{-1}(B))^\delta }}.
\]
Therefore we are reduced to   prove that the subset  $C_K^\delta$  defined above is closed.
Let  $\theta_j \in  C_K^\delta$ be such that $\theta_j \to \theta$ uniformly. By definition, for  each $j$, there exists a point  $t_j \in [0,1] \cap \theta_j^{-1}(B)$ such that $\dist(t_j,K) \le \delta$. As before,  we may assume that $t_j\to t\in [0,1]$ and  , by passing to the limit, we deduce that $\theta(t)\in B$ while $\dist(t,K)\le t$. This means precisely that  $\theta\in C_K^\delta$.
 \end{proof}

Finally, we prove the  crucial measurability result which we use for proving the $SBV$-representation Theorem \ref{thm:rep}.\\

\begin{proposition}\label{prop:measurability_beta}
Let $B \subset X$ be a closed set and let $\eta$ be a finite Borel measure on $C([0,1],X)$, which is concentrated on $\Theta(X)$, the subspace of $\Lip_1([0,1],X)$ of injective curves with constant metric speed $1$. Then, the map $\beta: C([0,1],X) \to \mathbb D([0,1],X)$ defined by
\[
\beta(\theta) = \theta \circ \gamma_{\theta^{-1}(B)},
\]
 is $\eta$-measurable. Here, $\gamma_{K}$ is the monotone map from Lemma~\ref{lem:arc-length} associated with a closed set $K \in \mathcal K([0,1])$,
\end{proposition}
\begin{proof} 
Our goal is to demonstrate that the map $\beta: C([0,1],X) \to \mathbb D([0,1],X)$ is $\eta$-measurable. By Lusin's theorem, it suffices to show that for any $\varepsilon > 0$, there exists a closed set $F \subset \Lip_1([0,1],X)$ with $\eta(F^c) < \varepsilon$ such that the restriction $\beta|_F$ is continuous. The map $\beta(\theta)$ is the composition of three functions, and we will prove the continuity of each component in a sequence of steps.

\begin{enumerate}
    \item The inverse image map $L: C([0,1],X) \to (\mathcal{K}([0,1]), d_H)$, defined by $L(\theta) = \theta^{-1}(B)$.
    \item The transport map $\Gamma: \mathcal{K}([0,1]) \to \mathbb{D}([0,1])$, defined by $\Gamma(K) = u_K$.
    \item The composition map $C: \Theta(X) \times \mathbb D([0,1],X) \to \mathbb D([0,1],X)$, defined by $C(\theta, \gamma) = \theta \circ \gamma$.
\end{enumerate}

\subsubsection*{Step 1: continuity of the inverse image}

From Lemma~\ref{lem:Hausdorff}, the map $L$ is Borel measurable. A related map, $G: \Theta(X) \to [0,\infty)$ defined by $G(\theta) = \Mbf(\llbracket\theta\rrbracket \mres B)$, is upper semicontinuous. By Lusin's theorem, for any $\varepsilon > 0$, there exists a closed set $F \subset \Lip_1([0,1],X)$ with $\eta(F^c) < \varepsilon$ such that the restrictions of both $L$ and $G$ to $F$ are continuous. This implies that if $\theta_j \to \theta$ in $F$, then $d_H(L(\theta_j), L(\theta)) \to 0$.

Now, let $K_j = L(\theta_j)$ and $K = L(\theta)$. Since curves in $\Theta(X)$ are injective and have constant metric speed 1, their length corresponds to the measure of the preimage: $|K_j| = \mathcal{H}^1(\im \theta_j \cap B) = \Mbf(\llbracket\theta_j\rrbracket \mres B)$. Because $G$ is continuous on $F$, we have $|K_j| \to |K|$. The Hausdorff convergence $d_H(K_j, K) \to 0$ combined with the convergence of lengths $|K_j| \to |K|$ implies, by Lemma~\ref{lem:Haus_implies_L1}, that the characteristic functions converge in $L^1$:
\[
\|\mathbf{1}_{K_j} - \mathbf{1}_K\|_{L^1} \to 0.
\]
Therefore, the restriction $L|_F$ is continuous from $F$ to the space $\mathcal{K}([0,1])$ equipped with the metric $d_H + \|\cdot\|_{L^1}$.

\subsubsection*{Step 2: continuity of the transport map}

We now show that the map $\Gamma: K \mapsto u_K$ is continuous from $(\mathcal{K}([0,1]), d_H + \|\cdot\|_{L^1})$ to the Skorokhod space $(\mathbb{D}([0,1]), d_S)$. Specifically, if $d_H(K_j, K_0) + \|\mathbf{1}_{K_j} - \mathbf{1}_{K_0}\|_{L^1} \to 0$, then $\gamma_j = \gamma_{K_j}$ converges to $\gamma_0 = \gamma_{K_0}$ in the Skorokhod topology. This requires finding a sequence of strictly increasing homeomorphisms $\lambda_j: [0,1] \to [0,1]$ such that
\begin{equation}\label{eq:lambda}
\max\{ \|\gamma_0 - \gamma_j \circ \lambda_j\|_\infty, \|\mathrm{id} - \lambda_j\|_\infty \} \to 0.
\end{equation}

We begin by defining the sets of open intervals that correspond to the jumps of our transport maps. For each $j \in \mathbb{N}_0$, let $\mathcal{F}^j$ be the family of disjoint open intervals whose union is $\Gamma_j \coloneqq [0,1] \smallsetminus K_j$. For any $\delta > 0$, we define $\mathcal{F}^j_\delta \subset \mathcal{F}^j$ as the subfamily of intervals with length strictly larger than $\delta$, and we set $\Gamma^j_\delta \coloneqq \bigcup_{I \in \mathcal{F}^j_\delta} I$.

With these definitions in place, the convergence $d_H(K_0,K_j) + \|\mathbf{1}_{K_j} - \mathbf{1}_{K_0}\|_{L^1} \to 0$ gives rise to the following key observations:

\begin{enumerate}
    \item \emph{Convergence of total length:} The $L^1$ convergence of the characteristic functions of $K_j$ implies that the measure of the symmetric difference of the complements vanishes: $|\Gamma_j \triangle \Gamma_0| \to 0$. This is because $|K_j \triangle K_0| \to 0$.
    \item \emph{Convergence of "large" intervals:} For any fixed $\delta > 0$, the total length of the large jump intervals also converges:
    \[
    |\Gamma^j_\delta \triangle \Gamma^0_\delta| \to 0 \qquad \text{as $j \to \infty$}.
    \]
    To see this, note that the Hausdorff convergence $d_H(K_j,K_0) \to 0$ ensures that each interval $I \in \mathcal{F}^0_\delta$ is uniformly approximated by an interval in $\mathcal{F}^j_\delta$. Since the cardinality of each family $\mathcal{F}^j_\delta$ is finite and bounded by $\delta^{-1}$ (independently of $j$), we find that $\limsup_{j \to \infty} |\Gamma^0_\delta \smallsetminus \Gamma^j_\delta| = 0$. The analogous estimate holds for $\limsup_{j \to \infty} |\Gamma^j_\delta \smallsetminus \Gamma^0_\delta|$, proving the claim.
    \item \emph{Uniform measure convergence:} A stronger, quantified version of the previous point also holds. Due to the finiteness of $\mathcal{F}^j_\delta$ and the uniformity of the approximation, there exists a $j_1 = j_1(\delta)$ sufficiently large such that for all $j \ge j_1$:
    \[
    |\Gamma^j_\delta \triangle \Gamma^0_\delta|(B) \le \delta \qquad \text{for all Borel sets } B \subset (0,1).
    \]
    \item \emph{Convergence of small intervals:} For each $j \in \mathbb{N}_0$, the total length of the small jump intervals, $|\Gamma_j \smallsetminus \Gamma^j_\delta|$, tends to zero as $\delta \to 0$. This follows directly from the monotonicity of the limit $\Gamma^j_\delta \uparrow \Gamma_j$.
    \item \emph{$\delta$-uniform convergence of small intervals:} From observations (1) and (2), it follows that $\mathbf{1}_{\Gamma^j_\delta} \to \mathbf{1}_{\Gamma^0_\delta}$ and $\mathbf{1}_{\Gamma_j} \to \mathbf{1}_{\Gamma_0}$ in $L^1$ as $j \to \infty$. Consequently, there exists a $j_2 = j_2(\delta) \ge j_1$ such that for all $j \ge j_2$, the length of the remaining small intervals can be bounded:
    \[
    |\Gamma_j \smallsetminus \Gamma_\delta^j| = \|\mathbf{1}_{\Gamma_j} - \mathbf{1}_{\Gamma_\delta^j}\|_{L^1} \le \delta + \|\mathbf{1}_{\Gamma_0} - \mathbf{1}_{\Gamma_\delta^0}\|_{L^1} = \mathrm{O}_\delta(1).
    \]
\end{enumerate}

From our previous observations, for a fixed $\delta \in (0,1)$, there exists a $j_3 = j_3(\delta) \ge j_2$ such that for every $j \ge j_3$, we can find a bijective matching between the large jump intervals of $\Gamma_0$ and $\Gamma_j$. Specifically, there is a bijection $\mathcal{F}^0_\delta \to \mathcal{F}^j_\delta$, denoted by $I \mapsto \tilde{I}$, satisfying $|I \triangle \tilde{I}| \le \delta^2$.

This bijective matching allows us to construct a monotone bijection $\lambda_j^\delta: [0,1] \to [0,1]$ that maps each interval $I \in \mathcal{F}^0_\delta$ to its corresponding $\tilde{I} \in \mathcal{F}^j_\delta$. We construct this map by concatenating piecewise affine transformations shifting every point at most by $2\eps$ and sending each interval, into its corresponding matched interval. Due to the the difference of scales $\eps \ll \min\{\delta,m_\delta\}$, we can ensure each transformation only acts on a small neighborhood of its assigned interval, leaving all else invariant. In fact, this construction can be achieved with a bi-Lipschitz bijection $\lambda_j^\delta$ with achievable Lipschitz constants below $2$ (although that is will not play any role), satisfying:
\begin{equation}\label{eq:basta}
    \lambda_j(I) = \tilde I \text{ for all } I \in \mathcal F^0_\delta, \quad \max\{ \Lip(\lambda^\delta_j),\Lip( (\lambda^\delta_j)^{-1})\} \le 2, \quad \text{and} \quad \|\mathrm{id} - \lambda^\delta_j\|_\infty \le 2\epsilon \le \delta^2.
\end{equation}

To prove Skorokhod convergence, we must now estimate the uniform distance between the original curve $\gamma_0$ and the reparameterized curve $\gamma_j \circ \lambda_j^\delta$. We can express the monotone curve $\gamma_j(s)$ as a sum of its absolutely continuous part and its jumps. This expression is derived from Lemma~\ref{lem:arc-length}, which establishes the identity between the current associated with the absolutely continuous part of the curve, $\llbracket \gamma_j \rrbracket^a$, and the current associated with its speed function, $\llbracket K_j \rrbracket$. Similarly, the identity of currents $\llbracket \Gamma_j \rrbracket = \llbracket \gamma_j \rrbracket^j$ links the total length of the jump intervals to the jump part of the current. Therefore, the expression for $\gamma_j(s)$ is given by:
\[
    \gamma_j(s) = |K_j|s + \sum_{I \in \mathcal F_\delta^j\,,\,I \cap (0,s] \neq \emptyset} |I| +\sum_{I \in \mathcal F^j \smallsetminus \mathcal F_\delta^j\,,\,J \cap (0,s] \neq \emptyset} |J|, \qquad s \in (0,1).
\]
For all $j \ge j_3(\delta)$, we can bound the difference $|\gamma_0(s) - \gamma_j(\lambda_j^\delta(s))|$ by summing three terms:
\begin{itemize}
    \item \emph{absolutely continuous part:} The difference in the speed and reparameterization is bounded uniformly in $s$:
    \[
        (|K_j| - |K_0| + \|1 - \lambda_j^\delta\|_\infty )s \le j^{-1} + \delta^2.
    \]
    \item \emph{large jumps:} The difference in the lengths of the large intervals, which, by construction of $\lambda_j^\delta$ and our previous observations, is bounded uniformly in $s$:
    \[
        |\Gamma_\delta^j \triangle \Gamma_\delta^0| \le \delta.
    \]
    \item \emph{small jumps:} The contribution from the small intervals, which are the remainder of the total length, is bounded by a term dependent on $\delta$:
    \[
        |\Gamma_j \smallsetminus \Gamma_\delta^j| + |\Gamma_0 \smallsetminus \Gamma_\delta^0| = \mathrm{O}_\delta(1).
    \]
\end{itemize}
By collecting these uniform estimates, we find that for all $j \ge j_4(\delta) \coloneqq \max\{j_3(\delta),\delta^{-1}\}$, the overall difference is small:
\[
    \|\gamma_0 - \gamma_j \circ \lambda_j^\delta\|_\infty = \mathrm{O}_\delta(1).
\]
We can assume $j_4(\delta)$ is a monotone non-decreasing, surjective function of $\delta$.

To complete the proof, we construct the final sequence of reparameterizations. For each $j \in \mathbb{N}$, we define $\delta_j = k^{-1}$ for the unique integer $k$ such that $j_4(k^{-1}) < j \le j_4((k+1)^{-1})$. The surjectivity of $j_4$ guarantees that every $j$ belongs to such a unique interval. We then set $\lambda_j \coloneqq \lambda^{\delta_j}_j$. As $j \to \infty$, it follows that $\delta_j \downarrow 0$. By construction, $j > j_4(\delta_j)$, so we can apply the bounds above to conclude:
\[
    \|\mathrm{id} - \lambda_j\|_{\infty} \le 2\delta_j \to 0 \quad \text{and} \quad \|\gamma_0 - \gamma_j \circ \lambda_j\|_\infty = \mathrm{O}_{\delta_j}(1) \to 0 \quad \text{as } j \to \infty.
\]
This concludes the proof of the second step.

\subsubsection*{Step 3: Continuity of the Composition}

This step is straightforward. The composition map $C: (\theta, \gamma) \mapsto \theta \circ \gamma$ is well-defined for all $\theta \in \Theta(X)$ and $\gamma \in \mathbb D([0,1],X)$. Let a sequence $(\theta_j, \gamma_j)$ converge to $(\theta, \gamma)$ in the product topology on $\Theta(X) \times \mathbb D([0,1],X)$. By the definition of Skorokhod convergence, there exists a sequence of reparameterizations $\lambda_j \in \Lambda$ such that $\|\mathrm{id} - \lambda_j\|_\infty \to 0$ and $\|d(\gamma, \gamma_j \circ \lambda_j)\|_\infty \to 0$.

We can now estimate the distance between the composed maps using the triangle inequality and the Lipschitz property of $\theta_j$:
\begin{align*}
    \|d(\theta \circ \gamma, \theta_j \circ \gamma_j \circ \lambda_j)\|_\infty
    & \le \|d(\theta \circ \gamma, \theta_j \circ \gamma)\|_\infty + \|d(\theta_j \circ \gamma, \theta_j \circ \gamma_j \circ \lambda_j)\|_\infty \\
    & \le \|d(\theta, \theta_j)\|_\infty + \Lip(\theta_j) \|d(\gamma, \gamma_j \circ \lambda_j)\|_\infty \\
    & \le \|d(\theta, \theta_j)\|_\infty + \|d(\gamma, \gamma_j \circ \lambda_j)\|_\infty \to 0.
\end{align*}
The last inequality holds because the curves $\theta_j \in \Theta(X)$ have a Lipschitz constant of at most 1. Since both terms vanish as $j \to \infty$, the composition map is continuous.

\subsubsection*{Conclusion}

We have demonstrated that for any $\varepsilon > 0$, we can find a closed set $F$ with $\eta(F^c) < \varepsilon$ on which each of the three component maps—$L$, $\Gamma$, and $C$—is continuous. Since the composition of continuous functions is continuous, the full map $\beta = C(\theta, \Gamma(L(\theta)))$ is continuous on $F$. This, by Lusin's theorem, completes the proof that $\beta$ is $\eta$-measurable.
\end{proof}

\end{document}